\documentclass[10pt,twoside,a4paper]{article}

\usepackage[utf8]{inputenc}
\usepackage[T1]{fontenc}
\usepackage{amsmath}
\numberwithin{equation}{section}
\usepackage{amssymb}
\usepackage{mathtools}
\usepackage{bm}
\usepackage{microtype}
\usepackage[backend=biber,maxnames=10,backref=true,hyperref=true,giveninits=true,safeinputenc]{biblatex}
\DefineBibliographyStrings{english}{%
	backrefpage = {cited on page},
	backrefpages = {cited on pages},
}

\DeclareFieldFormat[report]{title}{``#1''}
\DeclareFieldFormat[book]{title}{``#1''}

\AtEveryBibitem{\clearfield{url}}
\AtEveryBibitem{\clearfield{note}}
\AtEveryBibitem{\clearfield{pagetotal}}
\usepackage{graphicx}
\usepackage{lmodern}
\title{Convergence Rates of First and Higher Order Dynamics for Solving Linear Ill-posed Problems}
\author{Radu Boţ$^1$\\{\footnotesize\href{mailto:radu.bot@univie.ac.at}{radu.bot@univie.ac.at}}
\and Guozhi Dong$^{2,3}$\\{\footnotesize\href{mailto:guozhi.dong@hu-berlin.de}{guozhi.dong@hu-berlin.de}}
\and Peter Elbau$^1$\\{\footnotesize\href{mailto:peter.elbau@univie.ac.at}{peter.elbau@univie.ac.at}}
\and Otmar Scherzer$^{1,4}$\\{\footnotesize\href{mailto:otmar.scherzer@univie.ac.at}{otmar.scherzer@univie.ac.at}}}
\date{}
\usepackage[pdftex,colorlinks=true,linkcolor=blue,citecolor=green,urlcolor=blue,bookmarks=true,bookmarksnumbered=true]{hyperref}
\hypersetup
{
    pdfauthor={Radu Bot, Guozhi Dong, Peter Elbau, Otmar Scherzer},
    pdfsubject={},
    pdftitle={Convergence Rates of First and Second Order Dynamics for Solving Linear Ill-posed Problems},
    pdfkeywords={}
}

\usepackage[hyperref,amsmath,thmmarks]{ntheorem}
\usepackage{aliascnt}

\newtheorem{lemma}{Lemma}[section]

\newaliascnt{proposition}{lemma}
\newtheorem{proposition}[proposition]{Proposition}
\aliascntresetthe{proposition}

\newaliascnt{corollary}{lemma}
\newtheorem{corollary}[corollary]{Corollary}
\aliascntresetthe{corollary}

\newaliascnt{theorem}{lemma}
\newtheorem{theorem}[theorem]{Theorem}
\aliascntresetthe{theorem}

\theorembodyfont{\normalfont}
\newaliascnt{definition}{lemma}
\newtheorem{definition}[definition]{Definition}
\aliascntresetthe{definition}

\newaliascnt{assumption}{lemma}

\aliascntresetthe{assumption}

\newaliascnt{notation}{lemma}

\aliascntresetthe{notation}

\newaliascnt{example}{lemma}
\newtheorem{example}[example]{Example}
\aliascntresetthe{example}

\theoremstyle{nonumberplain}
\theoremseparator{:}
\theoremheaderfont{\normalfont\itshape}

\newtheorem{remark}{Remark}

\theoremsymbol{\ensuremath{\square}}
\newtheorem{proof}{Proof}
\usepackage[a4paper,centering,bindingoffset=0cm,marginpar=2cm,margin=2.5cm]{geometry}
\usepackage[pagestyles]{titlesec}
\titleformat{\section}[block]{\large\sc\filcenter}{\thesection.}{0.5ex}{}[]
\titleformat{\subsection}[runin]{\bf}{\thesubsection.}{0.5ex}{}[.]

\newpagestyle{headers}%
{%
	\headrule
	\sethead%
		[\footnotesize\thepage]%
		[\footnotesize\sc R.~Boţ, G.~Dong, P.~Elbau, O.~Scherzer]%
		[]%
		{}%
		{\footnotesize\sc Convergence Rates of First and Higher Order Dynamics for Solving Linear Ill-posed Problems }%
		{\footnotesize\thepage}
	\setfoot{}{}{}
}
\usepackage[font=footnotesize,format=plain,labelfont=sc,textfont=sl,width=0.75\textwidth,labelsep=period]{caption}
\usepackage{dsfont}
\newcommand{\N}{\mathds{N}}
\newcommand{\Z}{\mathds{Z}}
\newcommand{\R}{\mathds{R}}

\newcommand{\abs}[1]{\left|#1\right|}
\newcommand{\norm}[1]{\left\|#1\right\|}
\newcommand{\set}[1]{\left\{#1\right\}}
\newcommand{\inner}[2]{\left<#1,#2\right>}
\newcommand{\e}{\mathrm e}
\renewcommand{\d}{\,\mathrm d}
\newcommand{\xdag}{x^\dagger}
\newcommand{\xa}{x_\alpha}
\newcommand{\ra}{r_\alpha}
\newcommand{\Ra}{R_\alpha}
\newcommand{\Xa}{X_\alpha}
\newcommand{\tra}{\tilde{r}_\alpha}
\newcommand{\Tra}{\tilde{R}_\alpha}
\newcommand{\phiH}{\varphi^{\mathrm H}}
\newcommand{\hatphiH}{\hat\varphi^{\mathrm H}}
\newcommand{\phiL}{\varphi^{\mathrm L}}

\newcommand{\ointerval}[1]{\left(0,#1\right)}
\newcommand{\cinterval}[1]{\left[0,#1\right]}
\newcommand{\cointerval}[1]{\left[0,#1\right)}
\begin{document}
%%%%%%%%%%%%%%%%%%%%%%%%%%%%%%
%%% Titlepage
%%%%%%%%%%%%%%%%%%%%%%%%%%%%%%
\maketitle
\thispagestyle{empty}
\hspace*{8em}
\parbox[t]{17em}{\footnotesize
\hspace*{-1ex}$^1$Faculty of Mathematics\\
University of Vienna\\
Oskar-Morgenstern-Platz 1\\
A-1090 Vienna, Austria}
\hfil
\parbox[t]{17em}{\footnotesize
\hspace*{-1ex}$^2$Institute for Mathematics\\
Humboldt University of Berlin\\
Unter den Linden 6\\
10099 Berlin, Germany}\\[1em]
\hspace*{8em}
\parbox[t]{17em}{\footnotesize
\hspace*{-1ex}$^3$Weierstrass Institute for Applied\\
\hspace*{1em}Analysis and Stochastics (WIAS)\\
Mohrenstraße 39\\
10117 Berlin, Germany}
\hfil
\parbox[t]{17em}{\footnotesize
\hspace*{-1ex}$^4$Johann Radon Institute for Computational\\
\hspace*{1em}and Applied Mathematics (RICAM)\\
Altenbergerstraße 69\\
A-4040 Linz, Austria}\\[1em]
\hspace*{8em}
%%% End

%%%%%%%%%%%%%%%%%%%%%%%%%%%%%%
%%% Abstract
%%%%%%%%%%%%%%%%%%%%%%%%%%%%%%
\begin{abstract}
Recently, there has been a great interest in analysing dynamical flows, where the stationary limit is 
the minimiser of a convex energy. Particular flows of great interest have been continuous limits 
of Nesterov's algorithm and the Fast Iterative Shrinkage-Thresholding Algorithm (FISTA), respectively.

In this paper we approach the solutions of linear ill-posed problems by dynamical flows. 
Because the squared norm of the residual of a linear operator equation is a convex functional, the
theoretical results from convex analysis for energy minimising flows are applicable. However, in the restricted situation of this paper they can often be significantly improved. Moreover, since we show that the proposed flows for minimising the norm of the residual of a linear operator equation are optimal regularisation methods and that they provide optimal convergence rates for the 
regularised solutions, the given rates can be considered the benchmarks for further studies in convex analysis. 
\end{abstract} 
%%% End

%%%%%%%%%%%%%%%%%%%%%%%
%%% seIntro
%%%%%%%%%%%%%%%%%%%%%%%
\section{Introduction}\label{seIntro}
We consider the problem of solving a linear operator equation
\begin{equation} \label{eq:linear}
 L x=y,
\end{equation}
where $L\colon\mathcal X \rightarrow \mathcal Y$ is a bounded linear operator 
between (infinite dimensional) real Hilbert spaces $\mathcal X$ and $\mathcal Y$. 
If the range of $L$ is not closed, \autoref{eq:linear} is 
ill-posed, see \cite{Gro84}, in the sense that small perturbations 
in the data $y$ can cause non-solvability of the \autoref{eq:linear} 
or large perturbations of the corresponding solution of \autoref{eq:linear} by perturbed 
right hand side. 
These undesirable effects are prevented by regularisation.

In this particular paper we consider dynamical regularisation 
methods for solving \autoref{eq:linear}.
That is, we approximate the minimum norm solution $\xdag$ of \autoref{eq:linear} by the solution $\xi$ of a dynamical system of the form
\begin{equation}
\label{eqODE}
\begin{aligned}
\xi^{(N)} (t) + \sum_{k=1}^{N-1} a_k(t)\xi^{(k)}(t) &= - L^*L \xi(t) + L^*\tilde y &&\text{for all } t \in \ointerval{\infty}, \\
\xi^{(k)}(0) &= 0 &&\text{for all } k=0,\ldots,N-1,
\end{aligned}
\end{equation}
at an appropriate time, where $N \in \N$, $a_k\colon(0,\infty) \to \R$, $k=1,\ldots,N-1$, are continuous functions, and $\tilde y$ is a perturbation of $y$. The stopping time is in practice often chosen via a standard discrepancy principle, see \cite[Chapter 3.3]{Gro84}. We are now interested under which conditions the regularised solution $\xi(t)$ can be guaranteed to converge to the solution $\xdag$ as $t\to\infty$ and how fast this convergence happens.

Studying first the case of exact data $\tilde y=y$, it turns out that the convergence rate, that is, the decay of $\|\xi(t)-\xdag\|^2$ in the limit $t\to\infty$, can be uniquely characterised by the spectral decomposition of the minimum norm solution $\xdag$ with respect to the operator $L^*L$, which allows us to get optimal convergence rates as a function of the ``regularity'' of the source $\xdag$. This regularity is usually described by so-called source conditions, the most common ones being of the form $\xdag\in\mathcal R((L^*L)^\frac\mu2)$ for some $\mu>0$; we refer to \cite[Chapter~2.2]{Gro84} and \cite[Chapter~3.2]{EngHanNeu96} for an introduction to the use of those source conditions for obtaining convergence rates. Moreover, these convergence rates for exact data are seen to be in a one-to-one correspondence to certain convergence rates for perturbed data as the perturbation $\|\tilde y-y\|^2$ goes to zero.

Outside the regularisation community source conditions might appear technical because they involve the operator $L$. However, it was demonstrated that for differential and integral operators $L$, these conditions very well coincide with smoothness conditions in Sobolev spaces. See for instance \cite{HanSch01}, where the analogy of smoothness and source conditions has been explained for the problem of numerical differentiation. For this analogy these conditions are also often termed smoothness conditions.

In particular, we will apply the general theory of this equivalent characterisation of convergence rates to the following three, well-studied examples:
\begin{enumerate}
\item
Showalter's method (also known as the gradient flow method), see \cite{Sho67,ShoBen70}, which corresponds to the case $N=1$ in \autoref{eqODE}:
\begin{equation}
\label{eq:Showalter}
\begin{aligned}
\xi'(t) &= -L^*L \xi(t)+ L^*\tilde y \text{  for all } t \in \ointerval{\infty}, \\
\xi(0) &= 0,
\end{aligned}
\end{equation}
see \autoref{tbShowalter} for an overview of the available convergence rates results;
\item
the heavy ball method, introduced in \cite{Pol64}, corresponding to $N=2$ with a constant function $a_1(t)=b>0$ in \autoref{eqODE}:
\begin{align}
\partial_{t t}\xi(t;\tilde y) + b\partial_t\xi(t;\tilde y) &= - L^*L \xi(t;\tilde y) + L^*\tilde y\text{ for all } t \in \ointerval{\infty}, \nonumber \\
\partial_t\xi(0;\tilde y) &= 0, \label{eqODESecond} \\
\xi(0;\tilde y) &= 0,\nonumber
\end{align}
where known convergence rates results are collected in \autoref{tbHeavyBall};
\item
the vanishing viscosity method, see \cite{SuBoyCan16}, which is the case of $N=2$ with $a_1(t)=\frac b t$ for some $b>0$ in \autoref{eqODE}:
\begin{align}
\partial_{t t}\xi(t;\tilde y) + \frac b t\partial_t\xi(t;\tilde y) &= - L^*L \xi(t;\tilde y) + L^*\tilde y\text{ for all } t \in \ointerval{\infty}, \nonumber \\
\partial_t\xi(0;\tilde y) &= 0, \label{eqODEVanishingViscosity} \\
\xi(0;\tilde y) &= 0.\nonumber
\end{align}
Some convergence rates from the literature are listed in \autoref{tbViscosity}.
\end{enumerate}

\begin{table}[htb]
{\footnotesize\def\arraystretch{1.5}
\hfil
\begin{tabular}{|p{2.4cm}|p{2cm}p{3.2cm}|p{2cm}p{3.4cm}|}\hline
Source Condition & $\norm{\xi(t)-\xdag}^2$ & & $\norm{L\xi(t)-y}^2$ & \\
\hline\hline
$\|L^\dag\|<\infty$ & $\mathcal O(\e^{-\|L^\dag\|^{-2}t})$ & \cite[Theorem 1]{Sho67} & $\mathcal O(\e^{-\|L^\dag\|^{-2}t})$ & \\ \hline
$\xdag\in\mathcal R((L^*L)^{\frac\mu2})$ & $\mathcal O(t^{-\mu})$ & \cite[Theorem 1]{ShoBen70} ($\mu=1$),\newline \autoref{thShoCR} & $\mathcal O(t^{-\mu-1})$ & \autoref{thShoCR} \\ \hline
$\xdag\in\mathcal N(L)^\perp$ & $o(1)$ & \cite[Theorem 1]{ShoBen70},\newline \autoref{thShoReg}\newline with \autoref{thRegErrorMonotone} & $\mathcal O(t^{-1})$\par\smallskip$o(t^{-1})$ & \cite[Theorem 1]{ShoBen70}\par\smallskip\autoref{thShoReg}\newline with \autoref{thResidualCR} \\ \hline
\end{tabular}}
\caption{Convergence rates for Showalter's method. (To compare the results from \cite{ShoBen70}, we remark that the condition $y\in\mathcal R(LL^*)$ given therein is equivalent to $\xdag\in\mathcal R((L^*L)^{\frac12})$, which can be directly seen, for example, from the characterisation of the range of a dual operator given in \cite[Lemma 8.31]{SchGraGroHalLen09}.)\\
We also remark that in the well-posed case $\|L^\dag\|<\infty$, the rates for $\|\xi(t)-\xdag\|^2$ and for $\|L\xi(t)-y\|^2$ are always the same, since $\|L^\dag\|^{-2}\|\xi(t)-\xdag\|^2\le\|L\xi(t)-y\|^2\le\|L\|^2\|\xi(t)-\xdag\|^2$.}\label{tbShowalter}
\end{table}

\begin{table}[htb]
{\footnotesize\def\arraystretch{1.5}
\hfil
\begin{tabular}{|p{2.4cm}|p{2.4cm}p{2.5cm}|p{2.4cm}p{3.5cm}|}\hline
Source Condition & $\norm{\xi(t)-\xdag}^2$ & & $\norm{L\xi(t)-y}^2$ & \\
\hline\hline
$\|L^\dag\|<\infty$ & $\mathcal O(\e^{\varepsilon t-\beta(\|L^\dag\|)\frac{bt}2})$ & \cite[Theorem 9.(5)]{Pol64} & $\mathcal O(\e^{\varepsilon t-\beta(\|L^\dag\|)\frac{bt}2})$ & \\ \hline
$\xdag\in\mathcal R((L^*L)^{\frac\mu2})$ & $\mathcal O(t^{-\mu})$ & \cite[Theorem 5.1]{ZhaHof20}, \newline \autoref{thSecOrdConclusion} & $\mathcal O(t^{-\mu-1})$ & \autoref{thSecOrdConclusion} \\ \hline
$\xdag\in\mathcal N(L)^\perp$ & $o(1)$ & \autoref{thSecOrdReg}\newline with \autoref{thSecOrdCompatibility}\newline and \autoref{thRegErrorMonotone} & $o(t^{-1})$ & \cite[Lemma 3.2]{ZhaHof20} ($b\ge\|L\|$),\newline\autoref{thSecOrdReg}\newline with \autoref{thSecOrdCompatibility}\newline and \autoref{thResidualCR} \\ \hline
\end{tabular}}
\caption{Convergence rates for the heavy ball method. Here, $\varepsilon>0$ denotes an arbitrarily small parameter and we have set $\beta(\|L^\dag\|)=1-(1-\frac4{b^2\|L^\dag\|^2})^{\frac12}$ for $\|L^\dag\|\ge\frac2b$ and $\beta(\|L^\dag\|)=1$ for~$\|L^\dag\|<\frac2b$.}\label{tbHeavyBall}
\end{table}

\begin{table}[htb]
{\footnotesize\def\arraystretch{1.5}
\hfil
\begin{tabular}{|p{2.4cm}|p{2cm}|p{1.7cm}p{2.4cm}|p{2.4cm}p{2.5cm}|}\hline
Source Condition & Parameters & $\norm{\xi(t)-\xdag}^2$ & & $\norm{L\xi(t)-y}^2$ & \\
\hline\hline
$\|L^\dag\|<\infty$ & $b>3$
	& $o(t^{-2})$ &
	& $o(t^{-2})$ & \cite[Theorem 4.16]{ApiAujDos18} \\
 & 
	& $\mathcal O(t^{-\frac{2b}3})$ & \cite[Theorem 3.4]{AttChbPeyRed18}
	& $\mathcal O(t^{-\frac{2b}3})$ & \cite[Theorem 3.4]{AttChbPeyRed18} \\
\hline
$\|L^\dag\|<\infty$ & $b>2$
	& $\mathcal O(t^{-2})$ &
	& $\mathcal O(t^{-2})$ & \cite[Theorem 7]{SuBoyCan16} \\
 & 
	& $\mathcal O(t^{-b})$ & 
	& $\mathcal O(t^{-b})$ & \cite[Theorem 4.2]{AujDosRon19} \\ \hline
$\|L^\dag\|<\infty$ & $0<b<3$
	& $\mathcal O(t^{-\frac{2b}3})$ & 
	& $\mathcal O(t^{-\frac{2b}3})$ & \cite[Theorem 4.19]{ApiAujDos18} \\
\hline\hline
$\xdag\in\mathcal R((L^*L)^{\frac\mu2})$ & $0<\mu<\frac b2$
	& $\mathcal O(t^{-2\mu})$ & \autoref{thSingularConclusion} 
	& & \\ \hline
$\xdag\in\mathcal R((L^*L)^{\frac\mu2})$ & $0<\mu<\frac b2-1$
	& & 
	& $\mathcal O(t^{-2\mu-2})$ & \autoref{thSingularConclusion} \\
\hline\hline
$\xdag\in\mathcal N(L)^\perp$ & $b\ge3$
	& &
	& $\mathcal O(t^{-2})$ & \cite[Theorem 2.7]{AttChbPeyRed18} \\ \hline
$\xdag\in\mathcal N(L)^\perp$ & $b>0$
	& $o(1)$ & \autoref{thSingReg}\newline with \autoref{thSingularCompatibility}\newline and \autoref{thRegErrorMonotone}
	& $\mathcal O(t^{-b+\varepsilon})+o(t^{-2})$ & \autoref{thSingReg}\newline with \autoref{thSingularCompatibility}\newline and \autoref{thCRImageSimple}\newline and \autoref{thResidualCR} \\ \hline
\end{tabular}}
\caption{Convergence rates for the vanishing viscosity flow. As before, $\varepsilon>0$ denotes an arbitrarily small parameter.}\label{tbViscosity}
\end{table}

Especially the vanishing viscosity method has recently been heavily investigated, see \cite{SuBoyCan16,BotCse16,AttChbPeyRed18,AttChbRia18}, for example, as it shows a faster convergence compared to the other two methods, and it was demonstrated to be a time continuous formulation of Nesterov's algorithm, see \cite{Nes83}, providing an explanation of the rapid convergence of this algorithm. Consequently, it was not only studied in the form of \autoref{eqODEVanishingViscosity}, but more generally with the right hand side (which in \autoref{eqODEVanishingViscosity} is the negative gradient of $\mathcal J_0(x)=\frac12\|L x-y\|^2$) replaced by the negative gradient of an arbitrary convex and differentiable functional $\mathcal J$. But, since our theory relies on spectral analysis, we limit our discussion to the quadratic functional $\mathcal J_0$.

In terms of convergence rates, however, the discussions for general functionals $\mathcal J$ are often limited to the estimation of the convergence of $\mathcal J(\xi(t))-\min_{x\in\mathcal X}\mathcal J(x)$, which for $\mathcal J=\mathcal J_0$ is given by $\frac12\|L\xi(t)-y\|^2$. In the well-posed case where the operator $L$ has a bounded pseudoinverse $L^\dag$, this convergence of the squared norm of the residual is equivalent to the convergence of the error $\|\xi(t)-\xdag\|^2$, but this is no longer true in the ill-posed case where the pseudoinverse is unbounded.
In contrast to this, our approach directly gives convergence rates for $\|\xi(t)-\xdag\|^2$, which then imply a convergence (typically of higher order) of the squared norm of the residual.

We will proceed as follows:
\begin{itemize}
 \item In \autoref{se:review} we revisit convergence rates results of regularisation methods from \cite{AlbElbHooSch16}, 
       which, in particular, allow to analyse first and higher order dynamics.
 \item In the following sections we apply the general results of \autoref{se:review} to regularising flow equations. 
       In \autoref{se:show} we 
       derive well-known convergence rates results of Showalter's method and prove optimality of this method. In 
       \autoref{seSecOrd} we prove regularising properties, optimality and convergence rates of the heavy ball dynamical 
       flow. In the context of inverse problems this method has already been analysed by \cite{ZhaHof20}, 
       however not in terms of optimality, as it is done here.
 \item In \autoref{se:sflow} we consider the vanishing viscosity flow. 
       We apply the general theory 
       of \autoref{se:review} and prove optimality of this method. In particular we prove under source conditions 
       (see for instance \cite{Gro84,EngHanNeu96}) optimal convergence rates (in the sense of regularisation theory) for $\|\xi(t)-\xdag\|^2$. These rates (and the resulting ones for the squared norm of the residual) are seen to interpolate nicely between the known rates in the well-posed (finite-dimensional) and those in the ill-posed setting when varying the regularity of the solution $\xdag$ (via changing the parameter $\mu$ in \autoref{tbViscosity}).
\end{itemize}

We want to emphasise that the terminologies \emph{optimal} from \cite{AujDosRon19} (a representative reference for this field) and 
\cite{AlbElbHooSch16} differ by the class of problems and the amount of a priori information taken into account. In \cite{AujDosRon19} 
best worst case error rates in the class of convex energies are derived, while we focus on squared functionals $\mathcal J$.
Moreover, we take into account prior knowledge on the solution. In view of this, it is not surprising that we get different ``optimal'' rates.

%%% End

%%%%%%%%%%%%%%%%%%%%%%%
%%% se:review
%%%%%%%%%%%%%%%%%%%%%%%
\section{Generalisations of Convergence Rates Results} \label{se:review}
In the following we slightly generalise convergence rates and saturation results from \cite{AlbElbHooSch16} so that they 
can be applied to prove convergence of the second order regularising flows in \autoref{seSecOrd} and \autoref{se:sflow}.
Thereby one needs to be aware that in classical regularisation theory, the regularisation parameter $\alpha > 0$ is considered 
a small parameter, meaning that we consider small perturbations of \autoref{eq:linear}. For dynamic 
regularisation methods of the form of \autoref{eqODE} we take large times to approximate the stationary state. To link these two 
theories, we will apply an inverse polynomial identification of optimal regularisation time and regularisation parameter.

Let $L\colon\mathcal X\to \mathcal Y$ be a bounded linear operator between two real Hilbert spaces $\mathcal X$ and $\mathcal Y$ with operator norm $\norm{L}$, 
$y \in \mathcal R(L)$, and let $\xdag \in \mathcal X$ be the minimum norm solution of $L x = y$ defined by
\[ L x^\dag=y\text{ and }\|x^\dag\|=\inf\{\norm x\mid L x=y\}. \]

\begin{definition}\label{deGenerator} 
We call a family $(\ra)_{\alpha>0}$ of continuous functions $\ra\colon[0,\infty)\to [0,\infty)$ the 
generator of a regularisation method if
\begin{enumerate}
\item \label{enGeneratorBounded} 
there exists a constant $\sigma\in(0,1)$ such that
\begin{equation}\label{eqGeneratorBounded}
\ra(\lambda) \le \min \set{\frac2\lambda,\frac\sigma{\sqrt{\alpha\lambda}}} \text{ for every } \lambda>0,\;\alpha>0; 
\end{equation}

\item \label{enGeneratorError}
the error function $\tra\colon(0,\infty)\to[-1,1]$, defined by
\begin{equation}\label{eqGeneratorError}
\tra(\lambda)=1-\lambda \ra(\lambda),\;\lambda > 0,
\end{equation}
is non-negative and monotonically decreasing on the interval $(0,\alpha)$;

\item \label{enGeneratorErrorReg}
there exists for every $\alpha>0$ a monotonically decreasing, continuous function $\Tra\colon(0,\infty)\to\cinterval{1} $ such that 
\[ \Tra \ge |\tra|  \text{ and } \alpha \mapsto \Tra(\lambda) \text{ is continuous and monotonically increasing for every fixed } \lambda > 0; \]

\item \label{enGeneratorLimit}
there exists for every $\bar\alpha>0$ a constant $\tilde\sigma\in(0,1)$ such that
\[ \Tra(\alpha)<\tilde\sigma \text{ for all } \alpha\in(0,\bar\alpha). \]
\end{enumerate}
\end{definition}

\begin{remark}
The definition of the generator of a regularisation method differs from the one in \cite{AlbElbHooSch16} by allowing the regularisation method to overshoot, meaning that $\ra(\lambda)>\frac1\lambda$ is possible at some points $\lambda>0$ (the choice $\ra(\lambda)=\frac1\lambda$, which is not a regularisation method in the sense of \autoref{deGenerator}, would correspond to taking the inverse without regularisation, see \autoref{eq:reg}). Consequently, we also relaxed the assumption that the error function $\tra$ is monotonically decreasing to the existence of a monotonically decreasing upper bound $\Tra$ for $\tra$. We also want to remark that in the definition of the error function in \cite{AlbElbHooSch16}, $\tilde r_\alpha^{\text{\cite{AlbElbHooSch16}}}$, there is an additional square included, that is, $\tilde r_\alpha^{\text{\cite{AlbElbHooSch16}}}=\tra^2$.
\end{remark}

\begin{definition}\label{de:error_f}
Let $(\ra)_{\alpha>0}$ be the generator of a regularisation method.
\begin{enumerate}
\item
The regularised solutions according to a generator $(\ra)_{\alpha>0}$ and data $\tilde{y}$ are defined by 
\begin{equation} \label{eq:reg}
\xa\colon\mathcal Y\to \mathcal X,\;\xa(\tilde{y}) = \ra(L^*L)L^* \tilde{y},
\end{equation}
where we use the bounded Borel functional calculus to identify the function $\ra\colon[0,\infty)\to[0,\infty)$ with a function acting on the space of positive semi-definite self-adjoint operators, see \cite[Chapter~XI.12]{Yos95}, for example.

\item
Let $(\Tra)_{\alpha>0}$ be as in \autoref{deGenerator}~\ref{enGeneratorErrorReg}. Then we define for all $\alpha>0$ the envelopes
\begin{equation}\label{eq:RA}
\Ra\colon(0,\infty)\to[0,\infty),\;\Ra(\lambda) = \frac1\lambda\left(1-\Tra(\lambda)\right), 
\end{equation}
and the corresponding regularised solutions 
\begin{equation}\label{eq:Xa}
\Xa\colon\mathcal Y\to \mathcal X,\;\Xa(\tilde{y})=\Ra(L^*L)L^*\tilde{y}.
\end{equation}
\end{enumerate}
\end{definition}

\begin{remark}
The family $(\Ra)_{\alpha>0}$ is also a generator of a regularisation method, since we have
\begin{equation}\label{eqUpperRegularisation}
\Ra(\lambda) = \frac{1-\Tra(\lambda)}\lambda \le \frac{1-\tra(\lambda)}\lambda = \ra(\lambda) \le \min \set{ \frac2\lambda,\frac\sigma{\sqrt{\alpha\lambda}} }\text{ for every }\lambda>0,\;\alpha>0,
\end{equation}
which verifies \autoref{deGenerator}~\ref{enGeneratorBounded}; and the other three conditions of \autoref{deGenerator} are tautologically fulfilled: \autoref{deGenerator}~\ref{enGeneratorError} by the definition of $\Tra$ via \autoref{deGenerator}~\ref{enGeneratorErrorReg}, and \autoref{deGenerator}~\ref{enGeneratorErrorReg} and~\ref{enGeneratorLimit} by choosing $\Tra$ itself as upper bound for $|\Tra|$.
\end{remark}

The idea of these regularised solutions is to replace the unbounded inverse of $L\colon\mathcal N(L)^\perp\to\mathcal R(L)$ by the bounded approximation $x_\alpha$, where the parameter $\alpha>0$ quantifies the regularisation. It should disappear in the limit $\alpha\to0$, where we typically expect $r_\alpha(\lambda)\to\frac1\lambda$ corresponding to $x_\alpha(y)\to(L^*L)^\dag L^*y = x^\dag$ (this is, however, not enforced by \autoref{deGenerator}, but we will add in \autoref{deCompatible} a compatibility condition to ensure this).
\begin{example}
The most prominent regularisation method is probably Tikhonov regularisation, where the regularised solution $x_\alpha(\tilde y)$ is defined as the minimisation point of the Tikhonov functional
\[ \mathcal T_{\alpha,\tilde y}\colon\mathcal X\to\R,\;\mathcal T_{\alpha,\tilde y}(x) = \|L x-\tilde y\|^2+\alpha\|x\|^2. \]
Solving the optimality condition, gives us for $x_\alpha(\tilde y)$ the expression
\[ x_\alpha(\tilde y) = (L^*L+\alpha I)^{-1}L^*\tilde y, \]
where $I\colon\mathcal X\to\mathcal X$ denotes the identity map on $\mathcal X$, which has with $r_\alpha(\lambda)\coloneqq\frac1{\lambda+\alpha}$ the form of \autoref{eq:reg} and $r_\alpha$ satisfies all the conditions of \autoref{deGenerator}, see \cite[Example~2.4]{AlbElbHooSch16}.
\end{example}

We will show later in \autoref{se:show}, \autoref{seSecOrd}, and \autoref{se:sflow} that also some common dynamical regularisation methods fall into this regularisation scheme so that all the convergence rates results from this section can be applied to these methods.

\begin{definition}\label{de:measure}
We denote by $A\mapsto\mathbf E_A$ and $A \mapsto\mathbf F_A$ the spectral measures of the operators $L^*L$ and~$LL^*$, 
respectively, on all Borel sets $A \subseteq[0,\infty)$; and we define the right-continuous and monotonically increasing function 
\begin{equation}
\label{eqSpectralFunction}
e\colon(0,\infty) \to[0,\infty),\; e(\lambda)=\|\mathbf E_{[0,\lambda]}\xdag \|^2.
\end{equation}
We remark that the minimum norm solution $\xdag$ is in the orthogonal complement of the null space $\mathcal N(L)$ of $L$ and we therefore have $\mathbf E_{[0,\lambda]}\xdag=\mathbf E_{(0,\lambda]}\xdag$.

Moreover, if $f\colon(0,\infty)\to\R$ is a right-continuous, monotonically increasing, and bounded function, we write  
\[ \int_a^b g(\lambda)\d f(\lambda) = \int_{(a,b]} g(\lambda)\d\mu_f(\lambda) \]
for the Lebesgue--Stieltjes integral of $f$, where $\mu_f$ denotes the unique non-negative Borel measure defined by $\mu_f((\lambda_1,\lambda_2])=f(\lambda_2)-f(\lambda_1)$ and $g\in L^1(\mu_f)$.

We introduce the following quantities, whose behaviour we want to relate to each other:
\begin{itemize}
\item
the {\bf spectral tail} of the minimum norm solution $\xdag$ with respect to the operator $L^*L$, that is, the asymptotic behaviour of $e(\lambda)$ as $\lambda$ tends to zero, see \cite{Neu97};

\item
the error between the minimum norm solution $\xdag$ and the regularised solution~$\xa(y)$ or~$\Xa(y)$ for the exact data $y$ called the {\bf noise free regularisation error}, that is,
\begin{equation}\label{eq:dD} 
d(\alpha) \coloneqq \norm{\xa(y)-\xdag}^2 \text{ and } D(\alpha)\coloneqq \norm{\Xa(y)-\xdag}^2, 
\end{equation}
respectively, as $\alpha$ tends to zero;

\item
the {\bf best worst case error} between the minimum norm solution $\xdag$ and the regularised solution~$\xa(\tilde{y})$ or~$\Xa(\tilde{y})$ for some data $\tilde{y}$ with distance less than or equal to $\delta>0$ to the exact data $y$ under optimal choice of the regularisation parameter $\alpha$, that is,
\begin{equation} \label{eq:tilde_dD} 
\tilde d(\delta) \coloneqq\sup_{\tilde{y}\in\bar B_\delta(y)}\inf_{\alpha>0} \norm{\xa(\tilde y)-\xdag}^2 
\text{ and } \tilde D(\delta) \coloneqq\sup_{\tilde{y}\in\bar B_\delta(y)}\inf_{\alpha>0} \norm{\Xa(\tilde y)-\xdag}^2,
\end{equation}
respectively, as $\delta$ tends to zero;

\item
the {\bf noise free residual error}, which is the error between the image of the regularised solution~$\xa(y)$ or~$\Xa(y)$ and the exact data $y$, that is,
\begin{equation}\label{eqQ}
q(\alpha) \coloneqq \|L\xa(y)-y\|^2\text{ and } Q(\alpha) \coloneqq\|L\Xa(y)-y\|^2,
\end{equation}
respectively, as $\alpha$ tends to zero.
\end{itemize}
\end{definition}

To describe the behaviour of these quantities, we consider, for example, convergence rates of the form
\[ d(\alpha) = \|\xa(y) - \xdag\|^2 \le C_d\varphi(\alpha)\text{ for all }\alpha>0, \]
with some constant $C_d>0$ for the noise free regularisation error $d$, characterised by the decay of a monotonically increasing function $\varphi\colon(0,\infty)\to(0,\infty)$ for $\alpha\to0$, and look for a corresponding (equivalent) characterisation of the convergence rates of the other quantities, such as $e(\lambda)=\|\mathbf E_{[0,\lambda]}\xdag \|^2$ or $q(\alpha)=\norm{L\xa(y)-y}^2$.

\begin{example}\label{exCR}
Common families of functions $\varphi$ used to describe the convergence rates are Hölder functions
\begin{equation}\label{eqCRHoelder}
\phiH_\mu\colon(0,\infty)\to\R,\;\phiH_\mu(\alpha) = \alpha^\mu\text{ for all }\mu>0,
\end{equation}
see \cite{Gro84}, for example;
and logarithmic
\begin{equation}\label{eqCRLog}
\phiL_\mu\colon(0,\infty)\to\R,\;\phiL_\mu(\alpha) = \begin{cases}\left|\log\alpha\right|^{-\mu},&\alpha<\e^{-1}, \\ 1,&\alpha\ge\e^{-1},\end{cases}\text{ for all }\mu>0,
\end{equation}
or even double logarithmic functions, see for instance \cite{Hoh00,Sch01a}.
See \autoref{fgCR} for a sketch of the graphs of these functions.
\end{example}

\begin{figure}[htb]
\centering\includegraphics{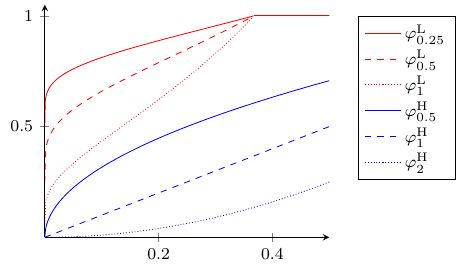}
\caption{Graphs of some common functions used to characterise convergence rates. See \autoref{exCR} for the definitions of these functions.}\label{fgCR}
\end{figure}

The main results are collected in \autoref{thCR} and \autoref{thCRImage}. We proceed in the following way to derive them:
\begin{itemize}
\item
In \autoref{thCrResidualRepresentations} and \autoref{thRegErrorMonotone}, we write the different regularisation errors in spectral form.
\item
In \autoref{thExactData} and \autoref{thExactData2}, we show the relations between the convergence rates of the noise free quantities $e$, $d$, and $D$. For this, we require the function $\varphi$, which describes the rate of convergence and is the same for all three quantities, to be compatible with the regularisation method, see \autoref{deCompatible}.
\item
In \autoref{thNoisyData} and \autoref{thNoisyData2}, we derive the relations of the best worst case errors $\tilde d$ and $\tilde D$ to the quantities $e$ and $D$. The corresponding rate of convergence is hereby of the form $\Phi[\varphi]$, where the mapping $\Phi$ is introduced in \autoref{deTransform} and some of its elementary properties are shown in \autoref{thTransform}, \autoref{thTransformMonotone}, \autoref{thTransformConstant}, and \autoref{thTransformHomogeneous}.
\item
The statements for the residual errors $q$ and $Q$ are then concluded from \autoref{thCR} by using the identification of $q$ and $Q$ for the minimum norm solution $x^\dag$ with the noise free errors $d$ and $D$ for the minimum norm solution $\bar x^\dag=(L^*L)^{\frac12}x^\dag$ of the problem $L x=\bar y$ with $\bar y=L\bar x^\dag$, and they are summarised in \autoref{thCRImage}, \autoref{thCRImageSimple}, and \autoref{thResidualCR}.
\end{itemize}

In the remaining of this section, we will always consider $(\ra)_{\alpha>0}$ to be the generator of a regularisation method with an envelope $(\Ra)_{\alpha>0}$ and corresponding regularised solutions $(\xa)_{\alpha>0}$ and $(\Xa)_{\alpha>0}$, respectively. Moreover, we use the functions $e$, $d$, $D$, $\tilde d$, $\tilde D$, $q$, and $Q$ as defined in \autoref{de:measure}, see \autoref{fgNotation} for a summary of the notation.

%%%%%%%%%%%%%%%%%%%%%%%%%%%%%%
%%% Table of the used notation
%%%%%%%%%%%%%%%%%%%%%%%%%%%%%%
\begin{table}
\begin{center}
\renewcommand{\arraystretch}{1.2}
\begin{tabular}{|l|l|l|}
\hline
Abbreviation & Description & Reference \\
\hline\hline
$\ra$ & Generator & \autoref{deGenerator} \\
\hline   
$\Ra$ & Envelope generator & \autoref{eq:RA} \\
\hline
$\tra$ & Error function & \autoref{eqGeneratorError}\\
\hline
$\Tra$ & Envelope error function & \autoref{deGenerator}~\ref{enGeneratorErrorReg}\\
\hline
$\xa(\tilde y)=\ra(L^*L)L^*\tilde y$ & Regularised solution according to $\ra$ & \autoref{eq:reg}\\
\hline
$\Xa(\tilde y)=\Ra(L^*L)L^*\tilde y$ & Regularised solution according to $\Ra$ & \autoref{eq:Xa}\\
\hline\hline
$d(\alpha)=\norm{\xa(y)-\xdag}^2$ & Noise free regularisation error for $\ra$ & \autoref{eq:dD}\\
\hline
$D(\alpha)=\norm{\Xa(y)-\xdag}^2$ & Noise free regularisation error for $\Ra$ & \autoref{eq:dD}\\
\hline
$\tilde{d}(\delta)$ & Best worst case error for $\ra$ & \autoref{eq:tilde_dD}\\
\hline
$\tilde{D}(\delta)$ & Best worst case error for $\Ra$ & \autoref{eq:tilde_dD}\\
\hline
$q(\alpha)=\norm{L\xa(y)-y}^2$ & Noise free residual error for $\ra$ & \autoref{eqQ}\\
\hline
$Q(\alpha)=\norm{L\Xa(y)-y}^2$ & Noise free residual error for $\Ra$ & \autoref{eqQ}\\
\hline\hline
$\mathbf E_A, \mathbf F_A$ & Spectral measures of $L^*L, LL^*$ & \autoref{de:measure}\\
\hline
$e(\lambda)=\|\mathbf E_{[0,\lambda]}\xdag \|^2$ & Spectral tail of $\xdag$ & \autoref{eqSpectralFunction}\\
\hline\hline 
$\hat{\varphi}$ & $\hat{\varphi}(\alpha) = \sqrt{\alpha \varphi(\alpha)}$ & \autoref{deTransform}\\
\hline
$\hat\varphi^{-1}$ & Generalised inverse of a function $\hat\varphi$ & \autoref{deTransform}\\
\hline
$\Phi$ & Noise-free to noisy transform & \autoref{deTransform}\\
\hline
\end{tabular}
\renewcommand{\arraystretch}{1}
\end{center}
\caption{Used variables and references to their definitions.}\label{fgNotation}
\end{table}
%%% End

%%%%%%%%%%%%%%%%%%%%%%%%%%%%%%
%%% Spectral Representations of the Regularisation Errors
%%%%%%%%%%%%%%%%%%%%%%%%%%%%%%
\subsection{Spectral Representations of the Regularisation Errors}

To do the analysis, we will expand the quantities of interest with respect to the measure $A\mapsto\norm{\mathbf E_A\xdag}^2$, which describes the spectral decomposition of $\xdag$ with respect to the operator $L^*L$. With the function $e$ defined in \autoref{eqSpectralFunction}, we can write the resulting integrals in the form of Lebesgue--Stieltjes integrals.

%%%%%%%%%%%%%%%%%%%%%%%
%%% thCrResidualRepresentations
%%%%%%%%%%%%%%%%%%%%%%%
\begin{lemma}\label{thCrResidualRepresentations}
We have the representations
\begin{equation}\label{eqCrResidum} 
d(\alpha) = \int_0^{\norm{L}^2}\tra^2(\lambda)\d e(\lambda) \text{ and } D(\alpha) = \int_0^{\norm{L}^2}\Tra^2(\lambda)\d e(\lambda)
\end{equation}
for the regularisation errors $d$ and $D$, respectively, and
\begin{equation}\label{eqCrResidual} 
q(\alpha) = \int_0^{\norm{L}^2}\lambda\tra^2(\lambda)\d e(\lambda) \text{ and } Q(\alpha) = \int_0^{\norm{L}^2}\lambda\Tra^2(\lambda)\d e(\lambda)
\end{equation}
for the residuals $q$ and $Q$, respectively.
\end{lemma}
\begin{proof}
We can write the differences between one of the regularised solutions $\xa(y)$ or $\Xa(y)$ and the minimum norm solution $x^\dag$ in the form
\begin{align*}
\xa(y)-\xdag  &= \ra(L^*L)L^*y - \xdag  = (\ra(L^*L)L^*L-I)\xdag \text{ and} \\
\Xa(y)-\xdag  &= (\Ra(L^*L)L^*L-I)\xdag,
\end{align*}
respectively, where $I\colon\mathcal X\to\mathcal X$ denotes the identity map on $\mathcal X$. According to spectral theory, we can formulate this with the definition of the error functions $\tra$ and $\Tra$, see \autoref{eqGeneratorError} and \autoref{eq:RA}, as
\[ \norm{\xa(y)-\xdag}^2 = \int_0^{\norm{L}^2}\tra^2(\lambda)\d e(\lambda) \text{ and } \norm{\Xa(y)-\xdag}^2 = \int_0^{\norm{L}^2}\Tra^2(\lambda)\d e(\lambda). \]

For the differences between the image of the regularised solution $\xa(y)$ or $\Xa(y)$ and the exact data, we find similarly
\begin{align*}
\|L\xa(y)-y\|^2  &= \|L\ra(L^*L)L^*L x^\dag - L\xdag\|^2  = \left<x^\dag,L^*L(\ra(L^*L)L^*L-I)^2x^\dag\right> \text{ and} \\
\|L\Xa(y)-y\|^2  &= \left<x^\dag,L^*L(\Ra(L^*L)L^*L-I)^2x^\dag\right>. 
\end{align*}
Thus, we have
\[ \norm{L\xa(y)-y}^2 = \int_0^{\norm{L}^2}\lambda\tra^2(\lambda)\d e(\lambda)\text{ and } \norm{L\Xa(y)-y}^2 = \int_0^{\norm{L}^2}\lambda\Tra^2(\lambda)\d e(\lambda). \]
\end{proof}
%%% End

From this representation, we immediately get that the regularised solutions $(\xa)_{\alpha>0}$ and $(\Xa)_{\alpha>0}$ converge to the minimum norm solution $\xdag$ if the error functions $(\tra)_{\alpha>0}$ and $(\Tra)_{\alpha>0}$ tend to zero as $\alpha\to0$.

%%%%%%%%%%%%%%%%%%%%%%%
%%% thRegErrorMonotone
%%%%%%%%%%%%%%%%%%%%%%%
\begin{corollary}\label{thRegErrorMonotone}
The regularisation errors $D$, $Q$, $\tilde d$, and $\tilde D$ (but not necessarily $d$ and $q$) are monotonically increasing functions and the functions $D$ and $Q$ are also continuous.

Moreover, if $\lim_{\alpha\to0}\tra(\lambda)=0$ (or $\lim_{\alpha\to0}\Tra(\lambda)=0$, respectively) for every $\lambda>0$, then the regularised solutions $\xa(y)$ (or $\Xa(y)$, respectively) converge for $\alpha\to0$ in the norm topology to the minimum norm solution $\xdag$.
\end{corollary}
\begin{proof}
By assumption, see \autoref{deGenerator}~\ref{enGeneratorErrorReg}, $\alpha\mapsto\Tra(\lambda)$ is monotonically increasing, and so are the functions
\[ \alpha\mapsto D(\alpha) = \int_0^{\norm{L}^2}\Tra^2(\lambda)\d e(\lambda)\text{ and }\alpha\mapsto Q(\alpha) = \int_0^{\norm{L}^2}\lambda\Tra^2(\lambda)\d e(\lambda). \]

The monotonicity of $\tilde d$ and $\tilde D$ follows directly from their definition in \autoref{eq:tilde_dD} as suprema over the increasing sets $\bar B_\delta(y)$, $\delta>0$.

Since $\Tra(\lambda)\in[0,1]$ for every $\alpha>0$ and every $\lambda>0$ and $\alpha\mapsto\Tra(\lambda)$ is for every $\lambda>0$ continuous, see \autoref{deGenerator}~\ref{enGeneratorErrorReg}, Lebesgue's dominated convergence theorem implies for every $\alpha_0>0$:
\[ \lim_{\alpha\to\alpha_0}D(\alpha) = \int_0^{\norm{L}^2}\lim_{\alpha\to\alpha_0}\Tra^2(\lambda)\d e(\lambda) = D(\alpha_0)\text{ and }\lim_{\alpha\to\alpha_0}Q(\alpha) = \int_0^{\norm{L}^2}\lim_{\alpha\to\alpha_0}\lambda\Tra^2(\lambda)\d e(\lambda) = Q(\alpha_0), \]
which proves the continuity of $D$ and $Q$.

Similarly, we get with $|\tra(\lambda)|\le\Tra(\lambda)\le1$ for every $\alpha>0$ and every $\lambda>0$ from Lebesgue's dominated convergence theorem that
\begin{alignat*}{3}
\lim_{\alpha\to0}\norm{\xa(y)-\xdag}^2 &= \lim_{\alpha\to0}d(\alpha) &&= \int_0^{\norm{L}^2}\lim_{\alpha\to0}\tra^2(\lambda)\d e (\lambda) &&= 0\text{ if }\lim_{\alpha\to0}\tra(\lambda)=0\text{ and} \\
\lim_{\alpha\to0}\norm{\Xa(y)-\xdag}^2 &= \lim_{\alpha\to0}D(\alpha) &&= \int_0^{\norm{L}^2}\lim_{\alpha\to0}\Tra^2(\lambda)\d e (\lambda) &&= 0\text{ if }\lim_{\alpha\to0}\Tra(\lambda)=0.
\end{alignat*}
\end{proof}
%%% End

%%% End

%%%%%%%%%%%%%%%%%%%%%%%%%%%%%%
%%% Bounds for the Noise Free Regularisation Errors
%%%%%%%%%%%%%%%%%%%%%%%%%%%%%%
\subsection{Bounds for the Noise Free Regularisation Errors}

The representations of the noise free regularisation errors as integrals over the spectral tail $e$ allow us to characterise the convergence of the regularisation errors $d(\alpha)$ and $D(\alpha)$ in the limit $\alpha\to0$ in terms of the behaviour of the spectral tail $e(\lambda)$ for $\lambda\to0$.

\begin{lemma}\label{thExactData}
With the constant $\sigma\in(0,1)$ from \autoref{deGenerator}~\ref{enGeneratorBounded}, we have for every $\alpha>0$ the relation
\begin{equation}\label{eqCRNoiseFreeDged}
(1-\sigma)^2e(\alpha) \le d(\alpha) \le D(\alpha).
\end{equation}
That is, $(1-\sigma)^2$ times the spectral tail is a lower bound for the noise free regularisation error of the regularisation method, which 
in turn is a lower bound for the error of the regularisation method of the envelope generator.
\end{lemma}
\begin{proof}
Let $\alpha>0$ be fixed. With \autoref{eqCrResidum} and $\Tra\ge|\tra|$, according to \autoref{deGenerator}~\ref{enGeneratorErrorReg}, we find for the errors $d$ and $D$ that
\[ D(\alpha) = \int_0^{\norm{L}^2}\Tra^2(\lambda)\d e(\lambda) \ge \int_0^{\norm{L}^2}\tra^2(\lambda)\d e(\lambda) = d(\alpha). \]
Furthermore, since $\tra^2$ is monotonically decreasing on $[0,\alpha]$, according to \autoref{deGenerator}~\ref{enGeneratorError}, and $e(\lambda)=e(\norm{L}^2)$ for all $\lambda\ge\norm{L}^2$, we can estimate
\[ d(\alpha)  \ge \int_0^{\min\{\alpha,\norm{L}^2\}}\tra^2(\lambda)\d e(\lambda) = \int_0^\alpha\tra^2(\lambda)\d e(\lambda) \ge \tra^2(\alpha)e(\alpha). \]
Inserting the expression of \autoref{eqGeneratorError} for $\tra$ and using the upper bound from \autoref{deGenerator}~\ref{enGeneratorBounded}, we thus have
\[ d(\alpha) \ge (1-\alpha \ra(\alpha))^2e(\alpha) \ge (1-\sigma)^2e(\alpha). \]
\end{proof}

Since we did not require so far that the error functions $\tra$ and $\Tra$ vanish as $\alpha\to0$, we cannot assure that the regularised solutions $\xa(y)$ and $\Xa(y)$ converge as $\alpha\to0$ to the minimum norm solution or even get an upper bound on the regularisation errors $d$ and $D$. We therefore impose the following additional constraint for a function $\varphi$ to serve as an upper bound for the regularisation error.

\begin{definition}\label{deCompatible}
We call a monotonically increasing function $\varphi\colon(0,\infty)\to(0,\infty)$ compatible with the regularisation method $(\ra)_{\alpha>0}$ with correspondingly chosen error functions $(\Tra)_{\alpha>0}$ according to \autoref{deGenerator}~\ref{enGeneratorErrorReg} if there exists for arbitrary $\Lambda>0$ a monotonically decreasing, integrable function $F\colon[1,\infty)\to\R$ such that
\begin{equation}\label{eqSourceConditionTail}
\Tra^2(\lambda) \le F\left(\frac{\varphi(\lambda)}{\varphi(\alpha)}\right)\text{ for }0<\alpha\le\lambda\le\Lambda.
\end{equation}
\end{definition}

In particular, a monotonically increasing function $\varphi\colon(0,\infty)\to(0,\infty)$ with $\lim_{\alpha\to0}\varphi(\alpha)=0$ can only be compatible with $(\ra)_{\alpha>0}$ if
\begin{equation}\label{deCompatibleRate}
\lim_{\alpha\to0}\frac{\Tra^2(\lambda)}{\varphi(\alpha)} = 0\text{ for every }\lambda>0,
\end{equation}
since the integrability of the monotonically decreasing function $F$ in \autoref{eqSourceConditionTail} implies the asymptotic behaviour $\lim_{z\to\infty}z F(z)=0$.

\begin{remark}
With $F(z)=(A z)^{-\frac1\mu}$, $A\in(0,\infty)$, $\mu\in(0,1)$, \autoref{eqSourceConditionTail} is exactly the condition from \cite[Equation~7]{AlbElbHooSch16} for the error function $\Tra$ (there we assume that $\tra$ satisfies \autoref{deGenerator}~\ref{enGeneratorErrorReg} and~\ref{enGeneratorLimit} such that we can take $\Tra=\tra$).

These sort of conditions for ensuring convergence rates of the method have a long history. For the special choice $F(z)=Az^{-2}$, it was introduced as qualification of the regularisation method in \cite[Definition~1 and~2]{MatPer03}, which is now commonly used for characterising convergence rates, see \cite{HofMat07,FleHofMat11}, for example. Even before that, the condition was used for the convergence rates $\phiH_\mu$, see, for example, the textbooks \cite[Theorem 4.3]{Vai82}, \cite[Theorem 1.1 in Chapter 3]{VaiVer86}, and \cite[Theorem 4.3, Corollary 4.4]{EngHanNeu96}.
\end{remark}

\begin{lemma}\label{thExactData2}
Let $\varphi\colon(0,\infty)\to(0,\infty)$ be a monotonically increasing function which is compatible with $(\ra)_{\alpha>0}$ 
in the sense of \autoref{deCompatible} and dominates the spectral tail, that is,
\begin{equation}\label{eq:ellphi}
e(\lambda) \le \varphi(\lambda) \text{ for all }\lambda>0. 
\end{equation}
Then, with a monotonically decreasing and integrable function $F\colon[1,\infty)\to\R$ fulfilling \autoref{eqSourceConditionTail} for $\Lambda=\|L\|^2$, 
we get
\[ D(\alpha)\le (\max\{1,F(1)\}+\|F\|_{L^1})\varphi(\alpha)\text{ for all }\alpha>0. \]
That is, the order of the noise free regularisation error $D$ of the envelope generator $(\Ra)_{\alpha>0}$ is given by the function $\varphi$.
\end{lemma}
\begin{proof}
We first extend the function $F$ to $\tilde F\colon[0,\infty)\to\R$ via $\tilde F(z)\coloneqq\max\{1,F(1)\}$ for $z\in[0,1]$ and $\tilde F(z)\coloneqq F(z)$ for $z\in(1,\infty)$ so that we have (because of $\Tra^2(\lambda)\le1$ for all $\alpha>0$ and $\lambda>0$)
\[ \Tra^2(\lambda) \le \tilde F\left(\frac{\varphi(\lambda)}{\varphi(\alpha)}\right)\text{ for all }\alpha>0\text{ and }0<\lambda\le\norm{L}^2. \]
Taking for $D$ the representation from \autoref{eqCrResidum} and using that $\tilde F$ is monotonically decreasing, we get
\[ D(\alpha) = \int_0^{\norm{L}^2}\Tra^2(\lambda)\d e(\lambda) \le \int_0^{\norm{L}^2}\tilde F\left(\frac{\varphi(\lambda)}{\varphi(\alpha)}\right)\d e(\lambda)\le \int_0^{\norm{L}^2}\tilde F\left(\frac{e(\lambda)}{\varphi(\alpha)}\right)\d e(\lambda). \]
Then, the substitution $z=\frac{e(\lambda)}{\varphi(\alpha)}$ gives us
\[ D(\alpha)  \le \varphi(\alpha)\int_0^\infty\tilde F(z)\d z = (\max\{1,F(1)\}+\|F\|_{L^1})\varphi(\alpha). \]
\end{proof}

\begin{remark}
The result of \autoref{thExactData2} is analogous to \cite[Proposition~2.3]{AlbElbHooSch16} where the noise free regularisation error produced by a generator $(r_\alpha)_{\alpha>0}$ is estimated.
\end{remark}

The compatibility condition in \autoref{eqSourceConditionTail} is essentially a way to measure if the regularisation method converges at each spectral value faster than a given convergence rate $\varphi$, see \autoref{deCompatibleRate}. It is therefore not surprising that if some convergence rate is compatible with $(\ra)_{\alpha>0}$, then all slower convergence rates are also compatible with it.

%%%%%%%%%%%%%%%%%%%%%%%%%%%%%%
%%% thCompatibilitySlower
%%%%%%%%%%%%%%%%%%%%%%%%%%%%%%
\begin{lemma}\label{thCompatibilitySlower}
Let $\varphi_1,\varphi_2\colon(0,\infty)\to(0,\infty)$ be two monotonically increasing, continuous functions such that the ratio $\psi\coloneqq\frac{\varphi_1}{\varphi_2}$ is monotonically increasing on $(0,\alpha_0]$ for some $\alpha_0>0$.

Then $\varphi_2$ is compatible with the regularisation method $(\ra)_{\alpha>0}$ in the sense of \autoref{deCompatible} if $\varphi_1$ is compatible with $(\ra)_{\alpha>0}$.
\end{lemma}

\begin{proof}
Let $\Lambda>0$ be arbitrary. Since $\psi$ is continuous and everywhere positive, we have the positive bounds $m\coloneqq\min_{\alpha\in[\alpha_0,\Lambda]}\psi(\alpha)>0$ and $M\coloneqq\max_{\alpha\in[\alpha_0,\Lambda]}\psi(\alpha)\ge m$. Then, the monotonicity of $\psi$ on the interval $(0,\alpha_0]$ implies for every $\alpha\in(0,\Lambda]$ that
\[ \min_{\lambda\in[\alpha,\Lambda]}\frac{\psi(\lambda)}{\psi(\alpha)} \ge \frac{\min\{\psi(\alpha),m\}}{\psi(\alpha)} = \min\left\{1,\frac m{\psi(\alpha)}\right\} \ge \min\left\{1,\frac m M\right\} = \frac m M. \]
By definition of $\psi$, this means that
\[ \frac{\varphi_1(\lambda)}{\varphi_1(\alpha)} \ge \frac m M\,\frac{\varphi_2(\lambda)}{\varphi_2(\alpha)}\text{ for all }0<\alpha\le\lambda\le\Lambda. \]
Thus, if $F$ is a monotonically decreasing, integrable function $F\colon[1,\infty)\to\R$ such that \autoref{eqSourceConditionTail} holds for $\varphi=\varphi_1$, then
\[ \Tra^2(\lambda) \le F\left(\frac{\varphi_1(\lambda)}{\varphi_1(\alpha)}\right) \le F\left(\frac m M\,\frac{\varphi_2(\lambda)}{\varphi_2(\alpha)}\right)\text{ for all }0<\alpha\le\lambda\le\Lambda. \]
Since the function $\tilde F\colon[1,\infty)\to\R$ given by $\tilde F(z)\coloneqq F(\frac m M z)$ is also monotonically decreasing and integrable, this proves that $\varphi_2$ is compatible with $(\ra)_{\alpha>0}$, too.
\end{proof}
%%% End

In particular, if one of the Hölder rates from \autoref{exCR} is compatible with $(\ra)_{\alpha>0}$, then all the logarithmic rates are compatible.

%%%%%%%%%%%%%%%%%%%%%%%
%%% thTransComp
%%%%%%%%%%%%%%%%%%%%%%%
\begin{corollary}\label{thTransComp}
Let $\phiH_\mu$ and $\phiL_\mu$, $\mu>0$, be the rates defined in \autoref{exCR}.

Then $\phiL_\mu$ is for every $\mu>0$ compatible with the regularisation method $(\ra)_{\alpha>0}$ in the sense of \autoref{deCompatible} if there exists a parameter $\nu>0$ such that $\phiH_\nu$ is compatible with $(\ra)_{\alpha>0}$.
\end{corollary}

\begin{proof}
Let $\phiH_\nu$ be compatible with $(\ra)_{\alpha>0}$ for some $\nu>0$ and consider for arbitrary $\mu>0$ the function $\psi\coloneqq\frac{\phiH_\nu}{\phiL_\mu}$. Since
\[ \psi'(\alpha)=\alpha^{\nu-1}\left|\log\alpha\right|^{\mu-1}\left(\nu\left|\log\alpha\right|-\mu\right)>0\text{ for }0<\alpha\le\min\{\e^{-1},\e^{-\frac\mu\nu}\}\eqqcolon\alpha_0, \]
the function $\psi$ is monotonically increasing on $(0,\alpha_0]$. Thus, \autoref{thCompatibilitySlower} implies the compatibility of the function $\phiL_\mu$.
\end{proof}
%%% End

%%% End

%%%%%%%%%%%%%%%%%%%%%%%%%%%%%%
%%% Noise-free to noisy transform
%%%%%%%%%%%%%%%%%%%%%%%%%%%%%%
\subsection{Relation between Convergence Rates for Noise Free and for Noisy Data}
We will see that when applying the regularisation to noisy data, the convergence rates $D$ give rise to convergence rates of the form $\tilde D(\delta)\le C_{\tilde D}\Phi[D](\delta)$ for some constant $C_{\tilde D}>0$ and the transform $\Phi[D]$ of the function $D$ which satisfies the equation system
\[ \Phi[D](\delta) = D(\alpha_\delta) = \frac{\delta^2}{\alpha_\delta} \]
for some suitable function $\delta\mapsto\alpha_\delta$.

\begin{definition}\label{deTransform}
Let $\varphi\colon(0,\infty) \to [0,\infty)$ be a monotonically increasing function which is not everywhere zero. 
We define the noise-free to noisy transform $\Phi[\varphi]\colon(0,\infty)\to(0,\infty)$ of $\varphi$ by 
\[ \Phi[\varphi](\delta) \coloneqq \frac{\delta^2}{\hat\varphi^{-1}(\delta)}, \]
where we introduce the function
\[ \hat\varphi\colon(0,\infty)\to[0,\infty),\;\hat\varphi(\alpha)=\sqrt{\alpha\varphi(\alpha)} \]
and write $\hat\varphi^{-1}$ for the generalised inverse
\[ \hat\varphi^{-1}(\delta)\coloneqq\inf\{\alpha>0\mid\hat\varphi(\alpha)\ge\delta\}. \]
\end{definition}
\begin{remark}
We emphasise that the considered functions need to be neither continuous nor surjective to be able to define a generalised inverse. In particular the function $\hat{e}\colon(0,\infty) \to [0,\infty)$, $\lambda\mapsto\sqrt{\lambda e(\lambda)}$, with $e$ defined in \autoref{eqSpectralFunction}, is only right-continuous and not surjective in general. Nevertheless, a generalised inverse exists.

We also note that if $\varphi\colon(0,\infty)\to[0,\infty)$ is a monotonically increasing function which is not everywhere zero and $\alpha_0\coloneqq\inf\set{\alpha>0\mid\varphi(\alpha)>0}$, then $\hat\varphi\colon(0,\infty) \to [0,\infty)$, $\alpha\mapsto\sqrt{\alpha\varphi(\alpha)}$ is a strictly increasing function on $(\alpha_0,\infty)$ so that we have $\alpha=\hat\varphi^{-1}(\hat\varphi(\alpha))$ for every $\alpha\in(\alpha_0,\infty)$.
\end{remark}

Later on, we will apply this transform to the functions describing the convergence rates. We therefore calculate (at least in leading order) the noise-free to noisy transforms for the families of convergence rates introduced in \autoref{exCR}.
\begin{lemma}
Let $\phiH_\mu$ and $\phiL_\mu$ be the functions introduced in \autoref{exCR}.

Then, we have for every $\mu>0$ that
\begin{enumerate}
\item
$\displaystyle\Phi[\phiH_\mu] = \phiH_{\frac{2\mu}{\mu+1}}$ and
\item
$\displaystyle0<\liminf_{\delta\to0}\frac{\Phi[\phiL_\mu](\delta)}{\phiL_\mu(\delta)}\le\limsup_{\delta\to0}\frac{\Phi[\phiL_\mu](\delta)}{\phiL_\mu(\delta)}<\infty$.
\end{enumerate}
\end{lemma}

\begin{proof}\mbox{}\makeatletter\nobreak\@afterheading\makeatother
\begin{enumerate}
\item
We find directly from \autoref{deTransform} that
\[ \Phi[\phiH_\mu](\delta)=\frac{\delta^2}{(\hatphiH_\mu)^{-1}(\delta)}\text{ with }\hatphiH_\mu(\alpha)=\alpha^{\frac{1+\mu}2},\text{ which gives } \Phi[\phiH_\mu](\delta)=\frac{\delta^2}{\delta^{\frac2{1+\mu}}} = \delta^{\frac{2\mu}{\mu+1}}. \]
\item
This is shown in \cite[Example~3.4\ {\em(ii)}]{AlbElbHooSch16}.
\end{enumerate}
\end{proof}

Let us collect some elementary properties of the transform $\Phi$ before estimating the quantities $\tilde d$ and $\tilde D$.

\begin{lemma}\label{thTransform}
Let $\varphi\colon(0,\infty)\to[0,\infty)$ be a monotonically increasing function which is not everywhere zero and $\hat\varphi(\alpha)\coloneqq\sqrt{\alpha\varphi(\alpha)}$.

Then, we have
\begin{enumerate}
\item
for every $\delta\in\hat\varphi\big((0,\infty)\big)\setminus\set{0}$ that
\[ \Phi[\varphi](\delta) = \varphi(\hat\varphi^{-1}(\delta))\text{ and,} \]
\item
if $\varphi$ is additionally right-continuous, that
\[ \Phi[\varphi](\delta) \le \varphi(\hat\varphi^{-1}(\delta))\text{ for every $\delta>0$}. \]
\end{enumerate}
\end{lemma}
\begin{proof}\mbox{}\makeatletter\nobreak\@afterheading\makeatother
\begin{enumerate}
\item
Since $\hat\varphi$ is strictly increasing on $\{\alpha>0\mid\hat\varphi(\alpha)>0\}$ and $\delta\in\hat\varphi\big((0,\infty)\big)\setminus\set{0}$, there exists exactly one point $\alpha>0$ with $\hat\varphi(\alpha)=\delta$, which then is by definition $\alpha=\hat\varphi^{-1}(\delta)$. Thus, we have that $\hat\varphi(\hat\varphi^{-1}(\delta))=\delta$, which means that
\[ \varphi(\hat\varphi^{-1}(\delta)) = \frac{\delta^2}{\hat\varphi^{-1}(\delta)} = \Phi[\varphi](\delta). \]
\item
Since $\varphi$ is right-continuous and monotonically increasing, it is upper semi-continuous and so is $\hat\varphi$. Thus, the set $\{\alpha>0\mid\hat\varphi(\alpha)\ge\delta\}$ is closed and therefore $\hat\varphi^{-1}(\delta)=\min\{\alpha>0\mid\hat\varphi(\alpha)\ge\delta\}$. In particular, we have that the inequality
\begin{equation}\label{eqTransformMonotoneRightCont}
\hat\varphi(\hat\varphi^{-1}(\delta))\ge\delta,\text{ that is, } \varphi(\hat\varphi^{-1}(\delta))\ge\frac{\delta^2}{\hat\varphi^{-1}(\delta)}=\Phi[\varphi](\delta),
\end{equation}
holds.
\end{enumerate}
\end{proof}

%%%%%%%%%%%%%%%%%%%%%%%
%%% thTransformMonotone
%%%%%%%%%%%%%%%%%%%%%%%
\begin{lemma}\label{thTransformMonotone}
Let $\varphi,\psi\colon(0,\infty)\to[0,\infty)$ be monotonically increasing functions which are not everywhere zero.

Then,
\begin{enumerate}
\item
$\psi\le\varphi$ implies that $\Phi[\psi]\le\Phi[\varphi]$ and,
\item
if $\varphi$ is additionally right-continuous, then $\Phi[\psi]\le\Phi[\varphi]$ also implies $\psi\le\varphi$.
\end{enumerate}
\end{lemma}
\begin{proof}
We set $\hat\varphi(\alpha)\coloneqq\sqrt{\alpha\varphi(\alpha)}$ and $\hat\psi(\alpha)\coloneqq\sqrt{\alpha\psi(\alpha)}$.
\begin{enumerate}
\item
Let $\psi\le\varphi$. Then, we have
\[ \hat\psi^{-1}(\delta) = \inf\{\alpha>0\mid\alpha\psi(\alpha)\ge\delta^2\} \ge \inf\{\alpha>0\mid\alpha\varphi(\alpha)\ge\delta^2\} = \hat\varphi^{-1}(\delta) \]
and thus
\[ \Phi[\psi](\delta) = \frac{\delta^2}{\hat\psi^{-1}(\delta)} \le \frac{\delta^2}{\hat\varphi^{-1}(\delta)} = \Phi[\varphi](\delta). \]
\item
Conversely, if $\Phi[\psi]\le\Phi[\varphi]$, then we get immediately that $\hat\varphi^{-1}\le\hat\psi^{-1}$.

Now, let $\alpha>0$ be arbitrary. If $\hat\psi(\alpha)=0$, there is nothing to show; so we assume $\hat\psi(\alpha)>0$ and define $\delta\coloneqq\hat\psi(\alpha)$. Then, $\alpha=\hat\psi^{-1}(\delta)\ge\hat\varphi^{-1}(\delta)$, so that we find with \autoref{eqTransformMonotoneRightCont} (using the right-continuity of $\varphi$) that
\[ \sqrt{\alpha\varphi(\alpha)}\ge\hat\varphi(\hat\varphi^{-1}(\delta))\ge\delta=\sqrt{\alpha\psi(\alpha)}. \]
So, $\varphi(\alpha)\ge\psi(\alpha)$.
\end{enumerate}
\end{proof}
%%% End

%%%%%%%%%%%%%%%%%%%%%%%
%%% thTransformConstant
%%%%%%%%%%%%%%%%%%%%%%%
\begin{lemma}\label{thTransformConstant}
Let $C>0$, $c>0$, and $\varphi\colon(0,\infty)\to[0,\infty)$ be a monotonically increasing function which is not everywhere zero. We set
\[ \psi(\alpha) \coloneqq C^2\varphi(c^2\alpha). \]

Then,
\[ \Phi[\psi](\delta) = C^2\Phi[\varphi](\tfrac c C\delta). \]
\end{lemma}

\begin{proof}
We define again $\hat\varphi(\alpha)\coloneqq\sqrt{\alpha\varphi(\alpha)}$ and $\hat\psi(\alpha)\coloneqq\sqrt{\alpha\psi(\alpha)}$. Then, we have for every $\delta>0$ that
\begin{align*}
\hat\psi^{-1}(\delta) &= \inf\{\alpha>0\mid \alpha\psi(\alpha)\ge\delta^2\} = \inf\{\alpha>0\mid C^2\alpha\varphi(c^2\alpha)\ge\delta^2\} \\
&= \frac1{c^2}\inf\{\tilde\alpha>0\mid \tilde\alpha\varphi(\tilde\alpha)\ge(\tfrac c C\delta)^2\} = \frac1{c^2}\hat\varphi^{-1}(\tfrac c C\delta),
\end{align*}
which gives us
\[ \Phi[\psi](\delta) = \frac{\delta^2}{\hat\psi^{-1}(\delta)} = \frac{(c\delta)^2}{\hat\varphi^{-1}(\tfrac c C\delta)} = C^2\Phi[\varphi](\tfrac c C\delta). \]
\end{proof}
%%% End

%%%%%%%%%%%%%%%%%%%%%%%
%%% thTransformHomogeneous
%%%%%%%%%%%%%%%%%%%%%%%
\begin{lemma}\label{thTransformHomogeneous}
Let $\varphi\colon(0,\infty)\to(0,\infty)$ be a monotonically increasing function and assume there exists a continuous, monotonically increasing function $G\colon(0,\infty)\to(0,\infty)$ such that
\[ \varphi(\gamma\alpha)\le G(\gamma)\varphi(\alpha)\text{ for all } \gamma>0,\;\alpha>0. \]

Then,
\[ \Phi[\varphi](\tilde\gamma\delta)\le\Phi[G](\tilde\gamma)\Phi[\varphi](\delta)\text{ for all }\tilde\gamma>0,\;\delta>0. \]
\end{lemma}
\begin{proof}
We get from $\varphi(\tilde\alpha)\le G(\gamma)\varphi(\tfrac1\gamma\tilde\alpha)$ with \autoref{thTransformMonotone} and \autoref{thTransformConstant} that
\[ \Phi[\varphi](\tilde\delta) \le G(\gamma)\Phi[\varphi]\left(\tfrac1{\sqrt{\gamma G(\gamma)}}\tilde\delta\right). \]
Thus, switching to the variable $\tilde\gamma\coloneqq\hat G(\gamma)\coloneqq\sqrt{\gamma G(\gamma)}$ 
(which means that $\gamma=\hat G^{-1}(\tilde\gamma)$ and thus, by \autoref{thTransform}, $\Phi[G](\tilde\gamma)=G(\gamma)$), we find with $\delta\coloneqq\frac1{\tilde\gamma}\tilde\delta$:
\[ \Phi[\varphi](\tilde\gamma\delta) \le \Phi[G](\tilde\gamma)\Phi[\varphi](\delta). \]
\end{proof}
%%% End

%%% End

%%%%%%%%%%%%%%%%%%%%%%%%%%%%%%
%%% Bounds for the Best Worst Case Errors
%%%%%%%%%%%%%%%%%%%%%%%%%%%%%%
\subsection{Bounds for the Best Worst Case Errors}
Let us finally come back to the functions $\tilde d$ and $\tilde D$, the best worst case errors of the regularisation methods defined by the generators $(\ra)_{\alpha>0}$ and $(\Ra)_{\alpha>0}$, respectively. Here we derive an estimate between the best worst case errors and the noise free regularisation errors.

%%%%%%%%%%%%%%%%%%%%%%%%%%%%%%
%%% thNoisyData
%%%%%%%%%%%%%%%%%%%%%%%%%%%%%%
\begin{lemma}\label{thNoisyData}
Let $\xdag\ne0$. Then, we have with the constant $\sigma\in(0,1)$ from \autoref{deGenerator}~\ref{enGeneratorBounded} that
\[ \tilde d(\delta) \le (1+\sigma)^2\Phi[D](\delta) \text{ and }\tilde D(\delta) \le (1+\sigma)^2\Phi[D](\delta)\text{ for all }\delta>0. \]
\end{lemma}
\begin{proof}
To estimate the distance between the regularised solutions for exact data $y$ and inexact data $\tilde{y} \in \bar B_\delta(y)$, we define the Borel measure
\[ \mu(A)=\|\mathbf F_A(\tilde{y}-y)\|^2, \]
where $\mathbf F$ denotes the spectral measure of the operator $LL^*$.
Then, we get with \autoref{eqUpperRegularisation} the relation
\begin{align*}
 \norm{\Xa(\tilde y)-\Xa(y)}^2 &= \left<\tilde{y}-y,\Ra^2(LL^*)LL^*(\tilde{y}-y)\right> \\
&= \int_{(0,\norm{L}^2]}\lambda \Ra^2(\lambda)\d\mu(\lambda) \le \int_{(0,\norm{L}^2]}\lambda \ra^2(\lambda)\d\mu(\lambda) =  \norm{\xa(\tilde y)-\xa(y)}^2.
\end{align*}
Thus, we have with \autoref{eqGeneratorBounded} the upper bound
\[ \norm{\Xa(\tilde y)-\Xa(y)}^2 \le \norm{\xa(\tilde y)-\xa(y)}^2 = \int_{(0,\norm{L}^2]}\lambda\ra^2(\lambda)\d\mu(\lambda) \le \delta^2\sup_{\lambda\in(0,\|L\|^2]}\lambda \ra^2(\lambda) \le \sigma^2\frac{\delta^2}\alpha. \]
The triangular inequality gives us then
\begin{equation}\label{eqTriangularD}
\tilde D(\delta) = \sup_{\tilde{y}\in\bar B_\delta(y)}\inf_{\alpha>0} \norm{\Xa(\tilde y)-\xdag}^2 \le \inf_{\alpha>0}\left( \norm{\Xa(y)-\xdag}+\sigma\frac\delta{\sqrt{\alpha}}\right)^2.
\end{equation}
We estimate the infimum therein from above by the value at $\alpha\coloneqq\hat D^{-1}(\delta)$, where we set $\hat D(\alpha)\coloneqq\sqrt{\alpha D(\alpha)}$.
Since the function $D$ is according to \autoref{thRegErrorMonotone} monotonically increasing and continuous, we get from \autoref{thTransform} and \autoref{deTransform} the identity 
$D(\hat D^{-1}(\delta))=\frac{\delta^2}{\hat D^{-1}(\delta)}=\Phi[D](\delta)$, so that both terms in the infimum are for this choice of $\alpha$ of the same order.
This gives us
\begin{equation}\label{eq:ddelta}
\tilde D(\delta) \le \left( \sqrt{D(\hat D^{-1}(\delta))} + \sigma\sqrt{\frac{\delta^2}{\hat D^{-1}(\delta)}} \right)^2 = (1+\sigma)^2\Phi[D](\delta).
\end{equation}

Because of \autoref{eqCRNoiseFreeDged}, we get in the same way
\begin{equation}\label{eqTriangulard}
\begin{aligned}
\tilde d(\delta) = \sup_{\tilde{y}\in\bar B_\delta(y)}\inf_{\alpha>0} \norm{\xa(\tilde y)-\xdag}^2 &\le \inf_{\alpha>0}\left( \norm{\xa(y)-\xdag}+\sigma\frac\delta{\sqrt{\alpha}}\right)^2 \\
&\le \inf_{\alpha>0}\left( \norm{\Xa(y)-\xdag}+\sigma\frac\delta{\sqrt{\alpha}}\right)^2 \le (1+\sigma)^2\Phi[D](\delta),
\end{aligned}
\end{equation}
where we used \autoref{eq:ddelta} in the last inequality.
\end{proof}
%%% End

The following lemma provides relations between the best worst case errors $\tilde d$ and $\tilde D$ of the regularisation methods generated by $(\ra)_{\alpha>0}$ and $(\Ra)_{\alpha>0}$, respectively, and the spectral tail $e$.

%%%%%%%%%%%%%%%%%%%%%%%%%%%%%%
%%% thNoisyData2
%%%%%%%%%%%%%%%%%%%%%%%%%%%%%%
\begin{lemma}\label{thNoisyData2}
Let $\xdag\ne0$. Then, there exist constants $c>0$ and $C>0$ such that we have the inequalities
\[ \tilde d(\delta) \ge c\Phi[e](\delta) \text{ and }\tilde D(\delta) \ge C\Phi[e](\delta)\text{ for all }\delta>0. \]
\end{lemma}
\begin{proof}
To obtain a lower bound on $\tilde d$, we write
\begin{equation}\label{eqCRlowerBound}
\begin{split}
\norm{\xa(\tilde y)-\xdag}^2 &= \norm{\xa(y)-\xdag}^2+ \norm{\xa(\tilde y)-\xa(y)}^2+2\left<\xa(\tilde{y})-\xa(y),\xa(y)-\xdag \right> \\
&=
\begin{multlined}[t][0.6\textwidth]
\norm{\xa(y)-\xdag}^2+\left<\tilde{y}-y,\ra^2(LL^*)LL^*(\tilde{y}-y)\right> \\[1ex]
+2\left<\ra(LL^*)(\tilde{y}-y),\ra(LL^*)LL^*y-y\right>.
\end{multlined}
\end{split}
\end{equation}
We set $\hat e(\alpha)\coloneqq\sqrt{\alpha e(\alpha)}$ and choose an arbitrary $\bar\alpha>0$ with the property that $\bar\delta\coloneqq\hat e(\bar\alpha)>0$. Then, we find according to \autoref{deGenerator}~\ref{enGeneratorLimit} a parameter $\tilde\sigma\in(0,1)$ with
\begin{equation}\label{eqRalphaalpha}
\Tra(\alpha)<\tilde\sigma\text{ for all }\alpha\in(0,\bar\alpha).
\end{equation} 
We now consider for $\delta\in(0,\bar\delta)$ the two cases $\hat e^{-1}(\delta)\in\bm\sigma(L^*L)\setminus\{0\}$ and $\hat e^{-1}(\delta)\notin\bm\sigma(L^*L)\setminus\{0\}$, where $\bm\sigma(L^*L)$ denotes the spectrum of the operator $L^*L$.

\begin{itemize}
\item
Assume that $\delta\in(0,\bar\delta)$ is such that $\alpha_\delta\coloneqq\hat e^{-1}(\delta)\in\bm\sigma(L^*L)\setminus \set{0}$.
From the continuity of $\tilde R_{\alpha_\delta}$ and \autoref{eqRalphaalpha}, we find that 
there exists a parameter $a_\delta\in(0,\alpha_\delta)$ such that 
\begin{equation}\label{eqadelta}
\tilde R_{\alpha_\delta}(a_\delta)<\tilde\sigma.
\end{equation} 
Then, the assumption $\alpha_\delta\in\bm\sigma(L^*L)\setminus\set{0}$ implies that the spectral projection 
$\mathbf F$ of the operator $LL^*$ fulfils $\mathbf F_{[a_\delta,2\alpha_\delta]}\ne0$.
To estimate \autoref{eqCRlowerBound} further, we will choose for given values of $\alpha>0$ and $\delta\in(0,\bar\delta)$ a 
particular point $\tilde y$. For this choice, we differ again between two cases.

\begin{itemize}
\item If
\[ z_{\alpha,\delta} \coloneqq \mathbf F_{[a_\delta,2\alpha_\delta]}(\ra(LL^*)LL^*y-y) \ne 0, \]
we pick 
\[ \tilde{y}=y+\delta\frac{z_{\alpha,\delta}}{\norm{z_{\alpha,\delta}}} \]
in \autoref{eqCRlowerBound} and obtain
\begin{multline*}
\norm{\xa \left(y+\delta \tfrac{z_{\alpha,\delta}}{\norm{z_{\alpha,\delta}}} \right) - \xdag}^2 = \norm{\xa(y)-\xdag}^2 +\frac{\delta^2}{\norm{z_{\alpha,\delta}}^2}\inner{z_{\alpha,\delta}}{\ra^2(LL^*)LL^*z_{\alpha,\delta}} \\
+\frac{2\delta}{\norm{z_{\alpha,\delta}}}\inner{\ra(LL^*)z_{\alpha,\delta}}{z_{\alpha,\delta}}.
\end{multline*}
Here, we may drop the last term as it is non-negative, which gives us the lower bound
\[ \norm{ \xa\left( y+\delta \tfrac{z_{\alpha,\delta}}{\norm{z_{\alpha,\delta}}} \right) - \xdag}^2 \ge \norm{\xa(y)-\xdag}^2+\delta^2\min_{\lambda\in[a_\delta,2\alpha_\delta]}\lambda \ra^2(\lambda). \]

\item Otherwise, if 
\[ \mathbf F_{[a_\delta,2\alpha_\delta]}(\ra(LL^*)LL^*y-y) = 0, \]
we choose $z_{\alpha,\delta}\in\mathcal R(\mathbf F_{[a_\delta,2\alpha_\delta]})\setminus \set{0}$ arbitrarily. 
Then, with $\tilde{y}=y+\delta\frac{z_{\alpha,\delta}}{\norm{z_{\alpha,\delta}}}$, the last term in 
\autoref{eqCRlowerBound} vanishes and we find again
\[ \norm{ \xa\left( y+\delta \tfrac{z_{\alpha,\delta}}{\norm{z_{\alpha,\delta}}} \right) - \xdag}^2 \ge \norm{\xa(y)-\xdag}^2+\delta^2\min_{\lambda\in[a_\delta,2\alpha_\delta]}\lambda \ra^2(\lambda). \]
\end{itemize}

Therefore, we end up with
\[ \tilde d(\delta) = \sup_{\tilde{y}\in\bar B_\delta(y)}\inf_{\alpha>0} \norm{\xa(\tilde y)-\xdag}^2 \ge \inf_{\alpha>0}\left( \norm{\xa(y)-\xdag}^2+\delta^2\min_{\lambda\in[a_\delta,2\alpha_\delta]}\lambda \ra^2(\lambda)\right). \]
Using \autoref{eqUpperRegularisation} and that $\Tra$ is by \autoref{deGenerator}~\ref{enGeneratorErrorReg} monotonically decreasing, we get the inequality
\[ \lambda \ra^2(\lambda) \ge \frac1\lambda\left(1-\Tra(\lambda)\right)^2 \ge \frac1{2\alpha_\delta}\left(1-\Tra(a_\delta)\right)^2\text{ for all } \lambda\in[a_\delta,2\alpha_\delta], \]
and since we already proved in \autoref{thExactData} that $d\ge (1-\sigma)^2e$, we can estimate further 
\[ \tilde d(\delta) \ge \inf_{\alpha>0}\left((1-\sigma)^2e(\alpha)+\frac{\delta^2}{2\alpha_\delta}\left(1-\Tra(a_\delta)\right)^2\right). \]
Now, the first term is monotonically increasing in $\alpha$ and, since $\alpha\mapsto\Tra(\lambda)$ is for every $\lambda>0$ monotonically increasing, see \autoref{deGenerator}~\ref{enGeneratorErrorReg}, the second term is monotonically decreasing in $\alpha$. Thus, we can estimate the expression for $\alpha<\alpha_\delta$ from below by the second term at $\alpha=\alpha_\delta$, and for $\alpha\ge\alpha_\delta$ by the first term at $\alpha=\alpha_\delta$: 
\[ \tilde d(\delta) \ge \min \set{(1-\sigma)^2e(\alpha_\delta), \frac{\delta^2}{2\alpha_\delta} \left(1-\tilde R_{\alpha_\delta}(a_\delta)\right)^2}. \]

Recalling that $\alpha_\delta=\hat e^{-1}(\delta)$ and that the function $e$ is right-continuous, we get from \autoref{thTransform} that 
$e(\alpha_\delta)\ge\Phi[e](\delta)$ and have by \autoref{deTransform} that $\frac{\delta^2}{\alpha_\delta}=\Phi[e](\delta)$. 
Thus, we obtain with \autoref{eqadelta} that
\begin{equation}\label{eqResultInSpectrum}
\tilde d(\delta) \ge c_0\Phi[e](\delta)\text{ with }c_0 \coloneqq \min \set{(1-\sigma)^2,\tfrac12(1-\tilde\sigma)^2}.
\end{equation}
\item It remains the case where $\alpha_\delta\coloneqq\hat e^{-1}(\delta)\notin\bm\sigma(L^*L)\setminus \set{0}$. We define
\[ \alpha_0\coloneqq\inf\{\alpha>0\mid e(\alpha)\ge e(\alpha_\delta)\} \in(0,\alpha_\delta]. \]
Since $e$ is right-continuous and monotonically increasing, the infimum is achieved and we have that $e(\alpha_0)=e(\alpha_\delta)$. Moreover, $\alpha_0\in\bm\sigma(L^*L)$, since $e$ is constant on every interval in $(0,\infty)\setminus\bm\sigma(L^*L)$ and so $\alpha_0\notin\bm\sigma(L^*L)$ would imply that $e(\lambda)=e(\alpha_\delta)$ for all $\lambda\in(\alpha_0-\varepsilon,\alpha_0+\varepsilon)$ for some $\varepsilon>0$ which would contradict the minimality of $\alpha_0$.

Setting $\delta_0\coloneqq\hat e(\alpha_0)$ (so $\hat e^{-1}(\delta_0)=\alpha_0$ and, according to \autoref{thTransform}, $e(\alpha_0)=\Phi[e](\delta_0)$), we have that $\delta_0=\hat e(\alpha_0)\le\hat e(\alpha_\delta)=\delta$ and we therefore find with the monotonicity of $\tilde d$, see \autoref{thRegErrorMonotone}, \autoref{eqResultInSpectrum}, and \autoref{thTransform} that
\[ \tilde d(\delta) \ge \tilde d(\delta_0)\ge c_0\Phi[e](\delta_0) = c_0e(\alpha_0) = c_0e(\alpha_\delta) \ge c_0\Phi[e](\delta). \]
\end{itemize}

Thus, we have shown for every $\delta\in(0,\bar\delta)$ that
\begin{equation}\label{eqNoisyData2interval}
\tilde d(\delta) \ge c_0\Phi[e](\delta),
\end{equation}
where $c_0$ is given by \autoref{eqResultInSpectrum}.

Now, we know from \autoref{thTransform} that $\Phi[e](\delta)\le e(\hat e^{-1}(\delta))\le e(\|L\|^2)$ for every $\delta>0$. Thus, setting $c\coloneqq\min\{c_0,\frac{\tilde d(\bar\delta)}{e(\norm{L}^2)}\}$, it follows with \autoref{eqNoisyData2interval} that the inequality $\tilde d(\delta)\ge c\Phi[e](\delta)$ holds for every $\delta>0$.

Following exactly the same lines, we also get that there exists a constant $C>0$ with
\[ \tilde D(\delta) \ge C\Phi[e](\delta)\text{ for every }\delta>0. \]
\end{proof}
%%% End

%%% End

%%%%%%%%%%%%%%%%%%%%%%%
%%% Optimal Convergence Rates
%%%%%%%%%%%%%%%%%%%%%%%
\subsection{Optimal Convergence Rates}

Putting together all these results, we can characterise the convergence of the regularisation errors for noise free data and the best worst case errors equivalently in terms of the regularity of the minimum norm solution, concretely, in the behaviour of the spectral tail. And we have shown in \cite{AlbElbHooSch16} that this can also be written in the form of variational source conditions.

%%%%%%%%%%%%%%%%%%%%%%%%%%%%%%
%%% thCR
%%%%%%%%%%%%%%%%%%%%%%%%%%%%%%
\begin{theorem}\label{thCR}
Let $\eta\in(0,1)$ be an arbitrary parameter and $\varphi\colon(0,\infty)\to(0,\infty)$ be a monotonically increasing function which is compatible with $(\ra)_{\alpha>0}$ in the sense of \autoref{deCompatible}. (The function $\varphi$ represents the expected convergence rate of the regularisation method.)

Then, the following statements are equivalent:
\begin{enumerate}
\item\label{enCRe}
There exists a constant $C_e>0$ such that $e(\lambda)\le C_e\varphi(\lambda)$ for every $\lambda>0$, meaning that the ratio of the spectral tail and the expected convergence rate is bounded.
\item\label{enCRd}
There exists a constant $C_d>0$ such that $d(\alpha)\le C_d\varphi(\alpha)$ for every $\alpha>0$, meaning that the ratio 
of the noise free rate of the regularisation method and the expected convergence rate is bounded.
\item\label{enCRD}
There exists a constant $C_D>0$ such that $D(\alpha)\le C_D\varphi(\alpha)$ for every $\alpha>0$, meaning that the ratio 
of the noise free rate of the envelope generated regularisation method and the expected convergence rate is bounded.
\item\label{enCRvar} The expected convergence rate satisfies the variational source condition that there exists a constant $C_\eta>0$ with
\begin{equation}\label{eqCRvar}
\left<\xdag ,x\right>\le C_\eta\|\varphi^{\frac1{2\eta}}(L^*L)x\|^\eta\|x\|^{1-\eta}\text{ for all } x\in \mathcal X.
\end{equation}
\end{enumerate}

\bigskip
If the function $\varphi$ is additionally right-continuous and $G$-subhomogeneous in the sense that there exists a continuous and monotonically increasing function $G\colon(0,\infty)\to(0,\infty)$ such that
\begin{equation}\label{eqCRsubhomogeneous}
\varphi(\gamma\alpha)\le G(\gamma)\varphi(\alpha)\text{ for all } \gamma>0,\;\alpha>0,
\end{equation}
then every one of these statements is also equivalent to each of the following two:
\begin{enumerate}
\addtocounter{enumi}{4}
\item\label{enCRtd}
There exists a constant $C_{\tilde d}>0$ such that $\tilde d(\delta)\le C_{\tilde d}\Phi[\varphi](\delta)$ for every $\delta>0$, meaning that the best worst case error of the regularisation method and the noise-free to noisy transformed expected convergence rate is bounded (in fact this justifies the name of the noise-free to noisy transform).
\item\label{enCRtD}
There exists a constant $C_{\tilde D}>0$ such that $\tilde D(\delta)\le C_{\tilde D}\Phi[\varphi](\delta)$ for every $\delta>0$, meaning that the best worst case error of the envelope regularisation method and the noise-free to noisy transformed expected convergence rate is bounded.
\end{enumerate}
\end{theorem}

\begin{proof}
We first note that there is nothing to show if $\xdag=0$, since then $e=d=D=\tilde d=\tilde D=0$, see \autoref{eqCrResidum}, \autoref{eqTriangularD}, and \autoref{eqTriangulard}. So, we assume that $\xdag\ne0$.

We also remark that if $\varphi$ is compatible with a regularisation method in the sense of \autoref{deCompatible} and $C>0$, 
then $C\varphi$ is compatible with the regularisation method.
\begin{description}
\item[\ref{enCRe}$\implies$\ref{enCRD}:] This follows directly from \autoref{thExactData2}.
\item[\ref{enCRD}$\implies$\ref{enCRd}:] This follows directly from \autoref{thExactData}.
\item[\ref{enCRd}$\implies$\ref{enCRe}:] This follows again directly from \autoref{thExactData}.
\item[\ref{enCRe}$\iff$\ref{enCRvar}:] This equivalence was proved in \cite[Proposition~4.1]{AlbElbHooSch16}.

\item[\ref{enCRD}$\implies$\ref{enCRtd}:]
Since $D\le C_D\varphi$, we get from \autoref{thTransformMonotone} and \autoref{thTransformConstant} that
\[ \Phi[D](\delta)\le\Phi[C_D\varphi](\delta)=C_D\Phi[\varphi](C_D^{-\frac12}\delta)\text{ for every }\delta>0. \]
Now, using the assumption from \autoref{eqCRsubhomogeneous}, we find with \autoref{thTransformHomogeneous} that
\[ \Phi[D](\delta)\le C_D\Phi[G](C_D^{-\frac12})\Phi[\varphi](\delta)\text{ for every }\delta>0. \]
We therefore get from \autoref{thNoisyData} that
\[ \tilde d(\delta)\le(1+\sigma)^2\Phi[D](\delta) \le (1+\sigma)^2C_D\Phi[G](C_D^{-\frac12})\Phi[\varphi](\delta)\text{ for every }\delta>0, \]
where $\sigma\in(0,1)$ is the constant from \autoref{deGenerator}~\ref{enGeneratorBounded}.

\item[\ref{enCRD}$\implies$\ref{enCRtD}:]
As before, \autoref{thNoisyData} implies
\[ \tilde D(\delta)\le(1+\sigma)^2\Phi[D](\delta) \le (1+\sigma)^2C_D\Phi[G](C_D^{-\frac12})\Phi[\varphi](\delta)\text{ for every }\delta>0. \]

\item[\ref{enCRtd}$\implies$\ref{enCRe}:] 
The estimate $\tilde d\le C_{\tilde d}\Phi[\varphi]$ together with the constant $c>0$ found in \autoref{thNoisyData2} yields that
\[ \Phi[e](\delta) \le \frac1c\tilde d(\delta) \le \frac{C_{\tilde d}}c\Phi[\varphi](\delta)\text{ for every }\delta>0. \]
Since we know from \autoref{thTransformConstant} that the function $\psi\colon(0,\infty)\to(0,\infty)$, defined by
\[ \psi(\alpha)\coloneqq\frac{C_{\tilde d}}c\varphi\left(\frac{C_{\tilde d}}c\alpha\right), \text{ fulfils } \Phi[\psi](\delta) = \frac{C_{\tilde d}}c\Phi[\varphi](\delta)\text{ for every }\delta>0, \]
it follows that $\Phi[e] \leq \Phi[\psi]$ and we get with \autoref{thTransformMonotone} and \autoref{eqCRsubhomogeneous} that
\[ e(\alpha) \le \psi(\alpha) = \frac{C_{\tilde d}}c\varphi\left(\frac{C_{\tilde d}}c\alpha\right) \le \frac{C_{\tilde d}}c G\left(\frac{C_{\tilde d}}c\right)\varphi(\alpha)\text{ for every }\alpha>0. \]

\item[\ref{enCRtD}$\implies$\ref{enCRe}:]
The estimate $\tilde D\le C_{\tilde D}\Phi[\varphi]$ 
yields with the constant $C>0$ found in \autoref{thNoisyData2} the inequality
\[ \Phi[e](\delta) \le \frac1C\tilde D(\delta) \le \frac{C_{\tilde D}}C\Phi[\varphi](\delta)\text{ for every }\delta>0 \]
and thus with \autoref{eqCRsubhomogeneous} as above:
\[ e(\alpha) \le \frac{C_{\tilde D}}C\varphi \left(\frac{C_{\tilde D}}C\alpha\right) \le \frac{C_{\tilde D}}CG\left(\frac{C_{\tilde D}}C\right)\varphi(\alpha)\text{ for every }\alpha>0. \]
\end{description}
\end{proof}
%%% End

\begin{remark}
We note that the conditions in \autoref{thCR}~\ref{enCRd},~\ref{enCRD},~\ref{enCRtd}, and~\ref{enCRtD} are convergence rates for the regularised solutions, which are equivalent to the spectral tail condition in \autoref{thCR}~\ref{enCRe} and to the variational source conditions in \autoref{thCR}~\ref{enCRvar}. We also want to stress, and this is a new result in comparison to \cite{AlbElbHooSch16}, that this holds for regularisation methods $(\ra)_{\alpha>0}$ whose error functions $\tra$ are not necessarily non-negative and monotonically decreasing and that this also enforces optimal convergence rates for the regularisation methods generated by the envelopes $(\Ra)_{\alpha>0}$.

The first work on equivalence of optimality of regularisation methods is \cite{Neu97}, which has served as a basis for the results in \cite{AlbElbHooSch16}. The equivalence of the optimal rate in \autoref{thCR}~\ref{enCRe} and the variational source condition in \autoref{thCR}~\ref{enCRvar} has been analysed in a more general setting in \cite{HeiHof09,FleHof10,FleHofMat11,Fle11}
\end{remark}
 
In particular, all the equivalent statements of \autoref{thCR} follow (under the assumptions of \autoref{thCR}) from the standard source condition, see \cite[e.g.~Corollary 3.1.1]{Gro84}. However, the standard source condition is not equivalent to these statements, see, for example, \cite[Corollary~4.2]{AlbElbHooSch16}.

%%%%%%%%%%%%%%%%%%%%%%%%%%%%%%
%%% thSsc
%%%%%%%%%%%%%%%%%%%%%%%%%%%%%%
\begin{proposition}\label{thSsc}
Let $\varphi\colon(0,\infty)\to(0,\infty)$ be a monotonically increasing, continuous function such that the standard source condition
\[ \xdag \in \mathcal R(\varphi^{\frac12}(L^*L)) \]
is fulfilled.

Then, there exists for every $\eta\in(0,1]$ a constant $C_\eta>0$ such that
\[ \left<\xdag ,x\right>\le C_\eta\|\varphi^{\frac1{2\eta}}(L^*L)x\|^\eta\|x\|^{1-\eta}\text{ for all } x\in \mathcal X. \]
\end{proposition}
\begin{proof}
This statement is shown in \cite[Corollary~4.2]{AlbElbHooSch16}.
\end{proof}
%%% End

Let us finally take a look at the additional condition of $G$-subhomogeneity introduced in \autoref{eqCRsubhomogeneous} in \autoref{thCR} to prove optimal convergence rates for the best worst case errors and check that the convergence rates from \autoref{exCR} satisfy this condition.

%%%%%%%%%%%%%%%%%%%%%%%%%%%%%%
%%% thConvRatesSub
%%%%%%%%%%%%%%%%%%%%%%%%%%%%%%
\begin{lemma}\label{thConvRatesSub}
Let $\phiH_\mu$ and $\phiL_\mu$ denote the families of convergence rates defined in \autoref{exCR}.

Then we have for every parameter $\mu>0$ that
\begin{enumerate}
\item
the function $\phiH_\mu$ is $G$-subhomogeneous for $G(\gamma)\coloneqq\gamma^\mu$ in the sense of \autoref{eqCRsubhomogeneous} and
\item
there exists a monotonically increasing, continuous function $G\colon(0,\infty)\to(0,\infty)$ such that function $\phiL_\mu$ is $G$-subhomogeneous in the sense of \autoref{eqCRsubhomogeneous}.
\end{enumerate}
\end{lemma}
\begin{proof}
\mbox{}
\begin{enumerate}
\item
We clearly have $\phiH_\mu(\gamma\alpha)=\gamma^\mu\phiH_\mu(\alpha)$ for all $\gamma>0$ and $\alpha>0$.
\item
We consider the function $g(\alpha;\gamma)\coloneqq\frac{\phiL_\mu(\gamma\alpha)}{\phiL_\mu(\alpha)}$. Since $g\colon(0,\infty)\times(0,\infty)\to(0,\infty)$ is continuous, $g(\alpha;\gamma) \le 1$ for $\alpha\ge\e^{-1}$, and
\[ \lim_{\alpha\to0}g(\alpha;\gamma) = \lim_{\alpha\to0}\left(\frac{\left|\log\alpha\right|}{\left|\log\alpha\right|-\log\gamma}\right)^\mu = 1, \]
the function $\tilde G\colon(0,\infty)\to(0,\infty)$, $\tilde G(\gamma)\coloneqq\sup_{\alpha\in(0,\infty)}g(\alpha;\gamma)$ is well-defined, monotonically increasing and satisfies by construction $\phiL_\mu(\gamma\alpha)\le\tilde G(\gamma)\phiL_\mu(\alpha)$ for all $\gamma>0$ and $\alpha>0$. Thus, $\phiL$ is $G$-subhomogeneous for every monotonically increasing, continuous function $G$ with $G\ge\tilde G$.
\end{enumerate}
\end{proof}
%%% End

%%% End

%%%%%%%%%%%%%%%%%%%%%%%%%%%%%%
%%% Optimal Convergence Rates for the Residual Error
%%%%%%%%%%%%%%%%%%%%%%%%%%%%%%
\subsection{Optimal Convergence Rates for the Residual Error}
By applying \autoref{thCR} to the source $(L^*L)^{\frac12}x^\dag$, we can directly establish a relation to the convergence rates for the noise free residual errors $q$ and $Q$ of the regularisation method and the envelope generated regularisation method as defined in \autoref{eqQ}.

%%%%%%%%%%%%%%%%%%%%%%%
%%% thCRImage
%%%%%%%%%%%%%%%%%%%%%%%
\begin{corollary}\label{thCRImage}
We introduce the squared norm of the spectral projection of $\bar x^\dag=(L^*L)^{\frac12}x^\dag$ as
\begin{equation}\label{eqE}
\bar e(\lambda)\coloneqq\norm{\mathbf E_{[0,\lambda]}\bar x^\dag}^2=\int_0^\lambda\tilde\lambda\d e(\tilde\lambda).
\end{equation}
Let $\bar\varphi\colon(0,\infty)\to(0,\infty)$ be a monotonically increasing function which is compatible with $(\ra)_{\alpha>0}$ in the 
sense of \autoref{deCompatible}. Then, the following statements are equivalent:
\begin{enumerate}
\item
There exists a constant $C_{\bar e}>0$ such that $\bar e(\lambda)\le C_{\bar e}\bar\varphi(\lambda)$ for every $\lambda>0$.
\item
There exists a constant $C_q>0$ such that $q(\alpha)\le C_q\bar\varphi(\alpha)$ for every $\alpha>0$.
\item
There exists a constant $C_Q>0$ such that $Q(\alpha)\le C_Q\bar\varphi(\alpha)$ for every $\alpha>0$.
\end{enumerate}
\end{corollary}
\begin{proof}
We first remark that since $x^\dag\in\mathcal N(L)^\perp=\mathcal N(L^*L)^\perp$, also $\bar x^\dag\in\mathcal N(L)^\perp$ and is therefore the minimum norm solution of the equation $L x=\bar y$ with $\bar y=L\bar x^\dag=L(L^*L)^{\frac12}x^\dag$. 
The claim now follows from \autoref{thCR} for the minimum norm solution $\bar x^\dag$ by identifying the function $e$ with $\bar e$ and the distances $d$ and~$D$ because of 
\begin{equation}\label{eqQInt}
q(\alpha) = \int_0^{\norm{L}^2}\lambda\tra^2(\lambda)\d e(\lambda) = \int_0^{\norm{L}^2}\tra^2(\lambda)\d\bar e(\lambda)\text{ and }Q(\alpha) = \int_0^{\norm{L}^2}\Tra^2(\lambda)\d\bar e(\lambda),
\end{equation}
see \autoref{thCrResidualRepresentations}, with $q$ and $Q$, respectively.
\end{proof}
%%% End

From \autoref{thCRImage}, we can obtain a non-optimal characterisation for the convergence rates of the noise free residual errors $q$ and $Q$ in terms of the spectral tail $e$ of the minimum norm solution $\xdag$ instead of having to rely on the spectral tail $\bar e$ of the point $(L^*L)^{\frac12}\xdag$.

%%%%%%%%%%%%%%%%%%%%%%%
%%% thCRImageSimple
%%%%%%%%%%%%%%%%%%%%%%%
\begin{corollary}\label{thCRImageSimple}
Let $\bar\varphi\colon(0,\infty)\to(0,\infty)$ be a monotonically increasing function which is compatible with $(\ra)_{\alpha>0}$ 
in the sense of \autoref{deCompatible} and fulfils
\begin{equation}\label{eqCRImageSimpleCond}
\lambda e(\lambda)\le\bar\varphi(\lambda)\text{ for all }\lambda>0,
\end{equation}
meaning that the ratio of the spectral tail and $\bar\varphi$ is bounded by the spectral representation of the inverse of~$L^*L$.

Then, there exists a constant $C>0$ such that we have
\begin{equation}\label{eqResidualCR}
q(\alpha)\le Q(\alpha)\le C\bar\varphi(\alpha)\text{ for all }\alpha>0.
\end{equation}
\end{corollary}
\begin{proof}
The first inequality follows with \autoref{deGenerator}~\ref{enGeneratorErrorReg} directly from the representation in \autoref{eqCrResidual} for $q$ and $Q$:
\begin{equation}\label{eqqleQ}
q(\alpha) = \int_0^{\norm{L}^2}\lambda\tra^2(\lambda)\d e(\lambda) \le \int_0^{\norm{L}^2}\lambda\Tra^2(\lambda)\d e(\lambda) = Q(\alpha).
\end{equation}
For the second inequality, we use that the function $\bar e$ defined in \autoref{eqE} fulfils
\begin{equation}\label{eqElelambdae}
\bar e(\lambda) = \int_0^\lambda\tilde\lambda\d e(\tilde\lambda) \le \lambda \int_0^\lambda\d e(\tilde\lambda) = \lambda e(\lambda) \le \bar\varphi(\lambda)\text{ for every }\lambda>0.
\end{equation}
Thus, \autoref{thCRImage} implies that there exists a constant $C>0$ with $Q(\alpha)\le C\bar\varphi(\alpha)$ for all $\alpha>0$.
\end{proof}
%%% End

\begin{remark}
In particular, \autoref{thCRImage} implies that \autoref{eqResidualCR} holds for all monotonically increasing functions $\bar\varphi$ with $\bar\varphi(\alpha)\ge c\alpha$ for some $c>0$ which are compatible with $(r_\alpha)_{\alpha>0}$.
\end{remark}

The condition in \autoref{eqCRImageSimpleCond} is, however, not equivalent to those in \autoref{thCRImage}.
\begin{example}
Let $\xdag$ be such that its spectral tail $e$ has the form
\begin{equation}\label{eqRelResidualCRe}
e(\lambda) = \frac1{\abs{\log\lambda}}\text{ for }\lambda\in(0,\lambda_0]
\end{equation}
for some $\lambda_0\in(0,1)$.

Then, we claim that $\bar e$, defined by \autoref{eqE}, converges faster to zero than $\lambda\mapsto\lambda e(\lambda)$, that is,
\begin{equation}\label{eqRelResidualCR}
\lim_{\lambda\to0}\frac{\bar e(\lambda)}{\lambda e(\lambda)} = 0,
\end{equation}
proving that the condition in \autoref{eqCRImageSimpleCond} is stronger than those in \autoref{thCRImage}.

To verify \autoref{eqRelResidualCR}, we plug in \autoref{eqE} and perform an integration by parts in the numerator to obtain
\[ \lim_{\lambda\to0}\frac{\bar e(\lambda)}{\lambda e(\lambda)} = 1-\lim_{\lambda\to0}\frac{\int_0^\lambda e(\tilde\lambda)\d\tilde\lambda}{\lambda e(\lambda)}. \]
Now, L'Hospital's rule implies that
\[ \lim_{\lambda\to0}\frac{\bar e(\lambda)}{\lambda e(\lambda)} = 1-\lim_{\lambda\to0}\frac{e(\lambda)}{e(\lambda)+\lambda e'(\lambda)} = 1-\frac1{1+\lim_{\lambda\to0}\frac{\lambda e'(\lambda)}{e(\lambda)}}. \]
Inserting our expression for $e$ from \autoref{eqRelResidualCRe}, we find that
\[ \lim_{\lambda\to0}\frac{\lambda e'(\lambda)}{e(\lambda)} = \lim_{\lambda\to0}\frac1{\abs{\log\lambda}} = 0 \]
herein, which shows \autoref{eqRelResidualCR}.
\end{example}

Since $\bar e$ tends by definition faster to zero than the identity $\bar\varphi\colon(0,\infty)\to(0,\infty)$, $\bar\varphi(\alpha)=\alpha$, the noise free residual errors $q$ and $Q$ also convergence (without imposing an additional source condition) faster than the identity provided that $\bar\varphi$ is compatible with $(\ra)_{\alpha>0}$.

%%%%%%%%%%%%%%%%%%%%%%%
%%% thResidualCR
%%%%%%%%%%%%%%%%%%%%%%%
\begin{corollary}\label{thResidualCR}
If the convergence rate $\bar\varphi\colon(0,\infty)\to(0,\infty)$, $\bar\varphi(\alpha)=\alpha$ is compatible with $(\ra)_{\alpha>0}$ in the sense of \autoref{deCompatible}, then we have that
\[ \lim_{\alpha\to0}\frac{q(\alpha)}\alpha = \lim_{\alpha\to0}\frac{Q(\alpha)}\alpha = 0. \]
\end{corollary}

\begin{proof}
Since $q\le Q$, see \autoref{eqqleQ}, it is enough to prove it for the function $Q$. We define $\bar e$ as in \autoref{eqE} and differ between two cases.
\begin{itemize}
\item
If $\bar e(\lambda)=0$ for all $\lambda\in[0,\lambda_0]$ for some $\lambda_0>0$, then we estimate, using the integral representation for~$Q$ from \autoref{eqQInt},
\[ Q(\alpha) = \int_{\lambda_0}^{\norm{L}^2}\Tra^2(\lambda)\d\bar e(\lambda) \le \Tra^2(\lambda_0)\|(L^*L)^{\frac12}\xdag\|^2. \]
Since $\bar\varphi$ is compatible to $(\ra)_{\alpha>0}$, we known from \autoref{deCompatibleRate} that
\[ \lim_{\alpha\to0}\frac{Q(\alpha)}\alpha = \|L\xdag\|^2\lim_{\alpha\to0}\frac{\Tra^2(\lambda_0)}\alpha = 0. \]

\item
If $\bar e(\lambda)>0$ for all $\lambda>0$, then we first construct using the compatibility of $\bar\varphi$, as in the proof of \autoref{thExactData2}, a monotonically decreasing and integrable function $\tilde F\colon[0,\infty)\to\R$ with
\[ \Tra^2(\lambda) \le\tilde F(\tfrac\lambda\alpha)\text{ for all }\alpha>0\text{ and }0<\lambda\le\norm{L}^2. \]

Next, we pick a monotonically increasing function $f\colon(0,\infty)\to(0,\norm{L}^2)$ with
\begin{equation}\label{eqResidualf}
\lim_{\alpha\to0}f(\alpha) = 0\text{ and }\lim_{\alpha\to0}\frac{\bar e(f(\alpha))}\alpha = \infty
\end{equation}
and split the integral in \autoref{eqQInt} for $Q$ at the point $f(\alpha)$ into two giving us
\begin{equation}\label{eqResidualFirst}
Q(\alpha) = \int_0^{\norm{L}^2}\Tra^2(\lambda)\d\bar e(\lambda) \le \int_0^{f(\alpha)}\tilde F(\tfrac\lambda\alpha)\d\bar e(\lambda)+\int_{f(\alpha)}^{\norm{L}^2}\tilde F(\tfrac\lambda\alpha)\d\bar e(\lambda).
\end{equation}

We check that both terms decay faster than $\alpha$.
\begin{itemize}
\item
Since $\bar e$ fulfils by its definition in \autoref{eqE} that
\[ \lim_{\lambda\to0}\frac{\bar e(\lambda)}\lambda = 0, \]
we find for every $\varepsilon>0$ a value $\alpha_0>0$ such that
\begin{equation}\label{eqEasym}
\bar e(\lambda) \le \varepsilon\lambda\text{ for all }0<\lambda<f(\alpha_0).
\end{equation}

Therefore, we get for the first term in \autoref{eqResidualFirst} with the substitution $z=\frac{\bar e(\lambda)}{\varepsilon\alpha}$ that
\[ \int_0^{f(\alpha)}\tilde F(\tfrac\lambda\alpha)\d\bar e(\lambda) \le \int_0^{f(\alpha)}\tilde F(\tfrac{\bar e(\lambda)}{\varepsilon\alpha})\d\bar e(\lambda)\le \varepsilon\alpha\|\tilde F\|_{L^1}\text{ for all }\alpha<\alpha_0. \]
And since this holds for arbitrary $\varepsilon>0$, we see that
\[ \lim_{\alpha\to0}\frac1\alpha\int_0^{f(\alpha)}\tilde F(\tfrac\lambda\alpha)\d\bar e(\lambda) = 0. \]
\item
For the second term in \autoref{eqResidualFirst}, we remark that \autoref{eqEasym} also implies that there exists a constant $C>0$ with
\[ \bar e(\lambda) \le C\lambda\text{ for all }\lambda>0. \]
Thus, we find with the substitution $z=\frac{\bar e(\lambda)}{C\alpha}$ that
\[ \int_{f(\alpha)}^{\norm{L}^2}\tilde F(\tfrac\lambda\alpha)\d\bar e(\lambda) \le \int_{f(\alpha)}^{\norm{L}^2}\tilde F(\tfrac{\bar e(\lambda)}{C\alpha})\d\bar e(\lambda)\le C\alpha\int_{\frac{\bar e(f(\alpha))}{C\alpha}}^\infty\tilde F(z)\d z\text{ for all }\alpha>0. \]
According to our choice of $f$, see \autoref{eqResidualf}, the integral converges to zero for $\alpha\to0$ and we therefore obtain
\[ \lim_{\alpha\to0}\frac1\alpha\int_{f(\alpha)}^{\norm{L}^2}\tilde F(\tfrac\lambda\alpha)\d\bar e(\lambda) = 0. \]
\end{itemize}

\end{itemize}

\end{proof}
%%% End

%%% End

\medskip
The results of this section explain the interplay of the convergence rates of the spectral tail of the minimum norm solution, the noise free regularisation error, and the best worst case error. For these different concepts equivalent rates can be derived. Moreover, these rates also infer rates for the noise free residual error. In addition to standard regularisation theory, we proved rates on the associated regularisation method defined in \autoref{eq:RA}.

%%% End

%%%%%%%%%%%%%%%%%%%%%%%
%%% seSpectral
%%%%%%%%%%%%%%%%%%%%%%%
\section{Spectral Decomposition Analysis of Regularising Flows}\label{seSpectral}
We now turn to the applications of these results to the method in \autoref{eqODE} with some continuous functions 
$a_k\in C((0,\infty);\R)$, $k=0,\ldots,N-1$. We hereby consider the solution as a function of the possibly not exact data $\tilde y\in \mathcal Y$. Thus, we look for a solution $\xi\colon \cointerval{\infty}\times \mathcal Y \to \mathcal X$ of
\begin{subequations}\label{eqODE1}
\begin{alignat}{2}
\partial_t^N\xi(t;\tilde y) + \sum_{k=1}^{N-1} a_k(t)\partial_t^k\xi(t;\tilde y) &= - L^*L \xi(t;\tilde y) + L^*\tilde y&& \text{ for all } t \in \ointerval{\infty}, \label{eqODE1a} \\
\partial_t^k\xi(0;\tilde y) &= 0 &&\text{ for all }k\in\set{0,\ldots,N-1}, \label{eqODE1b}
\end{alignat}
\end{subequations}
such that $\xi(\cdot;\tilde y)$ is $N$ times continuously differentiable for every $\tilde y$.

The following proposition provides an existence and uniqueness of the solution of flows of higher order. In case 
that the coefficients $a_k$ are in $C^\infty([0,\infty);\R)$ the result can also be derived simpler from an abstract Picard--Lindelöf theorem, see, for example, \cite[Section II.2.1]{Kre72}.
However, in our case $a_k$ might also have a singularity at the origin, such as in \autoref{eqODEVanishingViscosity}, and the proof gets more involved. 

%%%%%%%%%%%%%%%%%%%%%%%%%%%%%%
%%% thSpectral
%%%%%%%%%%%%%%%%%%%%%%%%%%%%%%
\begin{proposition}\label{thSpectral}
Let $N\in\N$ and $\tilde y\in \mathcal Y$ be arbitrary, and let $A\mapsto\mathbf E_A$ denote the spectral measure of the operator~$L^*L$.

Assume that the initial value problem
\begin{subequations} \label{eqSpectralODE}
\begin{alignat}{2}
\partial_t^N\tilde\rho(t;\lambda)+\sum_{k=1}^{N-1}a_k(t)\partial_t^k\tilde\rho(t;\lambda)&=-\lambda\tilde\rho(t;\lambda) &&\text{ for all }\lambda\in\cointerval{\infty},\;t \in \ointerval{\infty}, \\
\partial_t^k\tilde\rho(0;\lambda)&=0 &&\text{ for all }\lambda\in\cointerval{\infty},\;k\in\set{1,\ldots,N-1},\\
\tilde\rho(0;\lambda)&=1 &&\text{ for all }\lambda\in\cointerval{\infty},\phantom{\sum_{k=1}^{N-1}}
\end{alignat}
\end{subequations}
has a unique solution $\tilde\rho\colon[0,\infty)\times[0,\infty)\to\R$ which is $N$ times partially differentiable with respect to $t$. 
Moreover, we assume that $\partial_t^k\tilde\rho\in C^1([0,\infty)\times[0,\infty);\R)$ for every $k\in\{0,\ldots,N\}$.

We define the function $\rho\colon[0,\infty)\times(0,\infty)\to\R$ by
\begin{equation}\label{eqRho}
\rho(t;\lambda) \coloneqq \frac{1-\tilde\rho(t;\lambda)}\lambda.
\end{equation}

Then, the function $\xi(\cdot;\tilde y)$, given by
\begin{equation}\label{eqSolutionODE}
\xi(t;\tilde y) = \int_{(0,\norm{L}^2]}\rho(t;\lambda)\d\mathbf E_\lambda L^*\tilde y\text{ for every }t\in[0,\infty),
\end{equation}
is the unique solution of \autoref{eqODE1} in the class of $N$ times strongly continuously differentiable functions.
\end{proposition}

%%%%%%%%%%%%%%%%%%%%%%%%%%%%%%
%%% Proof of thSpectral
%%%%%%%%%%%%%%%%%%%%%%%%%%%%%%
\begin{proof}
We split the proof in multiple parts. First, we will show that $\rho$ and $\xi$, defined by \autoref{eqRho} and \autoref{eqSolutionODE}, 
are sufficiently regular. Then, we conclude from this that $\xi$ satisfies the \autoref{eqODE1}. And finally, we show that every other 
solution of \autoref{eqODE1} coincides with $\xi$.
\begin{itemize}
% \rho is differentiable
\item
We start by showing that the function $\rho$ defined by \autoref{eqRho} can be extended to a function 
$\rho\colon[0,\infty)\times[0,\infty)\to\R$ which is $N$ times continuously differentiable with respect to $t$ by setting
\begin{equation} \label{eq:rho_tn}
 \rho(t;0) \coloneqq -\partial_\lambda\tilde\rho(t;0).
\end{equation}
For this, we only have to check the continuity of all the derivatives at the points $(t,0)$, $t\in[0,\infty)$. We observe that the solution of
\autoref{eqSpectralODE} for $\lambda=0$ is given by
\[ \tilde\rho(t;0)=1\text{ for every }t\in[0,\infty). \]
For the derivatives $\partial_t^k\rho$, $k\in\{0,\ldots,N\}$, we therefore find with the mean value theorem (recall that 
$\partial_\lambda\partial_t^k\tilde\rho=\partial_t^k\partial_\lambda\tilde\rho$ according to Schwarz's theorem, see, e.g., 
\cite[Theorem~9.1]{Rud76}, since $\partial_t^\ell\tilde\rho\in C^1([0,\infty)\times[0,\infty);\R)$ for every $\ell\in\{0,\ldots,k\}$) and 
\autoref{eq:rho_tn} that
\begin{align*}
\lim_{(\tilde t,\tilde\lambda)\to(t,0)}\left(\partial_t^k\rho(\tilde t,\tilde\lambda)-\partial_t^k\rho(t,0)\right)
&= \lim_{(\tilde t,\tilde\lambda)\to(t,0)}
      \left(\frac{\partial_t^k\tilde\rho(\tilde t,0)-\partial_t^k\tilde\rho(\tilde t,\tilde\lambda)}{\tilde\lambda}
                 +\partial_t^k\partial_\lambda\tilde\rho(t;0)\right) \\
&= \lim_{(\tilde t,\hat\lambda)\to(t,0)}\left(\partial_t^k\partial_\lambda\tilde\rho(t;0) 
   -\partial_\lambda\partial_t^k\tilde\rho(\tilde t,\hat\lambda)\right) = 0,
\end{align*}
which proves that $\partial_t^k\rho$ is for every $k\in\{0,\ldots,N\}$ continuous in $[0,\infty)\times[0,\infty)$.
% \xi is differentiable
\item
Next, we are going to show that the function $\xi$ is $N$ times continuously differentiable with respect to~$t$ and that its partial derivatives are for every $k\in\{0,\ldots,N\}$ given by
\begin{equation}\label{eqSolDerivative}
\partial_t^k\xi(t;\tilde y) = \int_{(0,\norm{L}^2]}\partial_t^k\rho(t;\lambda)\d\mathbf E_\lambda L^*\tilde y.
\end{equation}

To see this, we assume by induction that \autoref{eqSolDerivative} holds for $k=\ell$ for some $\ell\in\{0,\ldots,N-1\}$. Then, we get with the Borel measure $\mu_{L^*\tilde y}$ on $[0,\infty)$ defined by $\mu_{L^*\tilde y}(A)=\|\mathbf E_AL^*\tilde y\|^2$ that
\begin{align*}
\lim_{h\to0}&\left\|\frac{\partial_t^\ell\xi(t+h;\tilde y)-\partial_t^\ell\xi(t;\tilde y)}h-\int_{(0,\norm{L}^2]}\partial_t^{\ell+1}\rho(t;\lambda)\d\mathbf E_\lambda L^*\tilde y\right\|^2 \\
&= \lim_{h\to0}\left\|\int_{(0,\norm{L}^2]}\left(\frac{\partial_t^\ell\rho(t+h;\lambda)-\partial_t^\ell\rho(t;\lambda)}h-\partial_t^{\ell+1}\rho(t;\lambda)\right)\d\mathbf E_\lambda L^*\tilde y\right\|^2 \\
&= \lim_{h\to0}\int_{(0,\norm{L}^2]}\left(\frac{\partial_t^\ell\rho(t+h;\lambda)-\partial_t^\ell\rho(t;\lambda)}h-\partial_t^{\ell+1}\rho(t;\lambda)\right)^2\d\mu_{L^*\tilde y}(\lambda).
\end{align*}
Now, since $\partial_t^{\ell+1}\rho$ is continuous, it is in particular bounded on every compact set $[0,T]\times[0,\norm{L}^2]$, $T>0$. And since the measure $\mu_{L^*\tilde y}$ is finite, Lebesgue's dominated convergence theorem implies that
\begin{align*}
\lim_{h\to0}&\left\|\frac{\partial_t^\ell\xi(t+h;\tilde y)-\partial_t^\ell\xi(t;\tilde y)}h-\int_{(0,\norm{L}^2]}\partial_t^{\ell+1}\rho(t;\lambda)\d\mathbf E_\lambda L^*\tilde y\right\|^2 \\
&= \int_{(0,\norm{L}^2]}\lim_{h\to0}\left(\frac{\partial_t^\ell\rho(t+h;\lambda)-\partial_t^\ell\rho(t;\lambda)}h-\partial_t^{\ell+1}\rho(t;\lambda)\right)^2\d\mu_{L^*\tilde y}(\lambda)= 0,
\end{align*}
which proves \autoref{eqSolDerivative} for $k=\ell+1$. Since \autoref{eqSolDerivative} holds by definition of $\xi$ for $k=0$, this implies by induction that \autoref{eqSolDerivative} holds for all $k\in\{0,\ldots,N\}$.

Finally, the continuity of the $N$th derivative $\partial_t^N\xi$ follows in the same way directly from Lebesgue's dominated convergence theorem:
\[ \lim_{\tilde t\to t}\left\|\partial_t^N\xi(\tilde t;\tilde y)-\partial_t^N\xi(t;\tilde y)\right\|^2 = \lim_{\tilde t\to t}\int_{(0,\norm{L}^2]}\left(\partial_t^N\rho(\tilde t;\lambda)-\partial_t^N\rho(t;\lambda)\right)^2\d\mu_{L^*\tilde y} = 0. \]
% \xi solves the ODE
\item
To prove that $\xi$ solves \autoref{eqODE1}, we plug the definition of $\rho$ from \autoref{eqRho} into \autoref{eqSolDerivative} and find
\[ \partial_t^N\xi(t;\tilde y)+\sum_{k=1}^{N-1}a_k(t)\partial_t^k\xi(t;\tilde y) = -\int_{(0,\norm{L}^2]}\frac1\lambda\left(\partial_t^N\tilde\rho(t;\lambda)+\sum_{k=1}^{N-1}a_k(t)\partial_t^k\tilde\rho(t;\lambda)\right)\d\mathbf E_\lambda L^*\tilde y. \]
Making use of \autoref{eqSpectralODE}, we get that $\xi$ fulfils \autoref{eqODE1a}:
\begin{align*}
\partial_t^N\xi(t;\tilde y)+\sum_{k=1}^{N-1}a_k(t)\partial_t^k\xi(t;\tilde y) &= \int_{(0,\norm{L}^2]}\tilde\rho(t;\lambda)\d\mathbf E_\lambda L^*\tilde y \\
&= \int_{(0,\norm{L}^2]}(-\lambda\rho(t;\lambda)+1)\d\mathbf E_\lambda L^*\tilde y = -L^*L\xi(t;\tilde y)+L^*\tilde y.
\end{align*}
(We remark that $\mathcal R(L^*)\subset\mathcal N(L)^\perp=\mathcal N(L^*L)^\perp$ which implies that $\mathbf E_{(0,\|L\|^2]}L^*\tilde y=L^*\tilde y$.)

And for the initial conditions, we get, in agreement with \autoref{eqODE1b}, from \autoref{eqSolDerivative} that
\begin{align*}
\partial_t^k\xi(0;\tilde y) &= -\int_{(0,\norm{L}^2]}\partial_t^k\tilde\rho(0;\lambda)\d\mathbf E_\lambda L^*\tilde y = 0,\;k\in\set{1,\ldots,N-1},\text{ and} \\
\xi(0;\tilde y) &= \int_{(0,\norm{L}^2]}\frac{1-\tilde\rho(0;\lambda)}\lambda\d\mathbf E_\lambda L^*\tilde y = 0.
\end{align*}
% Uniqueness
\item
It remains to show that \autoref{eqSolutionODE} defines the only solution of \autoref{eqODE1}.

So assume that we have two different solutions of \autoref{eqODE1} and call $\xi_0$ the difference between the two solutions. We choose an arbitrary $t_0>0$ and write $\partial_t^k\xi_0(t_0;\tilde y)=\xi^{(k)}$ for every $k\in\set{0,\ldots,N-1}$. Then, $\xi_0$ is a solution of the initial value problem
\begin{subequations}\label{eqODEDiff}
\begin{alignat}{2}
\partial_t^N\xi_0(t;\tilde y) + \sum_{k=1}^{N-1} a_k(t)\partial_t^k\xi_0(t;\tilde y) &= - L^*L \xi_0(t;\tilde y) &&\text{ for all } t \in \ointerval{\infty} \\
\partial_t^k\xi_0(t_0;\tilde y) &= \xi^{(k)} &&\text{ for all } k\in\set{0,\ldots,N-1}.
\end{alignat}
\end{subequations}
We know, for example, from \cite[Section II.2.1]{Kre72}, that \autoref{eqODEDiff} has a unique solution on every interval~$[t_1,t_2]$, $0<t_1<t_0<t_2$. Thus, we can write $\xi_0$ in the form
\[ \xi_0(t;\tilde y) = \sum_{\ell=0}^{N-1}\int_{[0,\infty)}\rho_\ell(t;\lambda)\d\mathbf E_\lambda\xi^{(\ell)} \]
with the functions $\rho_\ell$ solving for every $\lambda\in[0,\infty)$ the initial value problems
\begin{alignat*}{2}
\partial_t^N\rho_\ell(t;\lambda)+\sum_{k=1}^{N-1}a_k(t)\partial_t^k\rho_\ell(t;\lambda)&=-\lambda\rho_\ell(t;\lambda)&&\text{ for all }t \in \ointerval{\infty}, \\
\partial_t^k\rho_\ell(t_0;\lambda)&=\delta_{k\ell}&&\text{ for all }k,\ell\in\set{0,\ldots,N-1}.
\end{alignat*}
(Since $a_k$ is continuous on $(0,\infty)$, Lebesgue's dominated convergence theorem is applicable to every compact set $[t_1,t_2]\times[0,\norm{L}^2]$, $0<t_1<t_0<t_2$.)

Now, we have for every measurable subset $A\subset[0,\infty)$ and every $k\in\{0,\ldots,N-1\}$ that
\[ \|\mathbf E_A\partial_t^k\xi_0(t;\tilde y)\|^2 = \sum_{\ell,m=0}^{N-1}\int_A\partial_t^k\rho_\ell(t;\lambda)\partial_t^k\rho_m(t;\lambda)\d\mu_{\xi^{(\ell)},\xi^{(m)}}(\lambda), \]
where the signed measures $\mu_{\eta_1,\eta_2}$, $\eta_1,\eta_2\in \mathcal X$, are defined by $\mu_{\eta_1,\eta_2}(A)=\left<\eta_1,\mathbf E_A\eta_2\right>$. 

The measures $\mu_{\xi^{(\ell)},\xi^{(m)}}$ with $\ell\ne m$ are absolutely continuous with respect to $\mu_{\xi^{(\ell)},\xi^{(\ell)}}$ and 
with respect to $\mu_{\xi^{(m)},\xi^{(m)}}$. Moreover, we can use Lebesgue's decomposition theorem, see, e.g., \cite[Theorem~6.10]{Rud87}, to split 
the measures $\mu_{\xi^{(\ell)},\xi^{(\ell)}}$, $\ell\in\{0,\ldots,N-1\}$, into measures $\mu_j$, $j\in\{0,\ldots,J\}$, $J\le N-1$, which are mutually 
singular to each other, so, explicitly, we write
\[ \mu_{\xi^{(\ell)},\xi^{(m)}} = \sum_{j=0}^J f_{j\ell m}\mu_j \]
for some measurable functions $f_{j\ell m}$ with $f_{j\ell m}=f_{j m\ell}$. Since then
\[ 0\le\left\|\sum_{\ell=0}^{N-1}\int_Ag_\ell(\lambda)\d\mathbf E_\lambda\xi^{(\ell)}\right\|^2 = \sum_{j=0}^J\int_A\sum_{\ell,m=0}^{N-1}f_{j\ell m}(\lambda)g_\ell(\lambda)g_m(\lambda)\d\mu_j(\lambda) \]
has to hold for all functions $g_\ell\in C([0,\infty);\R)$, $\ell\in\{0,\ldots,N-1\}$, and all measurable sets $A\subset[0,\infty)$, the matrices $F_j(\lambda)=(f_{j\ell m}(\lambda))_{\ell,m=0}^{N-1}$ are (after possibly redefining $f_{j\ell m}$ on sets $A_{j\ell m}$ with $\mu_j(A_{j\ell m})=0$) positive semi-definite.
Thus, we have for every measurable set $A\subset[0,\infty)$ that
\[ \|\mathbf E_A\partial_t^k\xi_0(t;\tilde y)\|^2 = \sum_{j=0}^J\int_A\sum_{\ell,m=0}^{N-1}f_{j\ell m}(\lambda)\partial_t^k\rho_\ell(t;\lambda)\partial_t^k\rho_m(t;\lambda)\d\mu_j(\lambda), \]
where the integrand is a positive semi-definite quadratic form of $\partial_t^k\rho$, namely $(\partial_t^k\rho)^{\mathrm T}F_j(\partial_t^k\rho)$, where $\rho=(\rho_\ell)_{\ell=0}^{N-1}$. We can therefore find for every $j\in\{0,\ldots,J\}$ and every $\lambda$ a change of coordinates $O_j(\lambda)\in\mathrm{SO}_N(\R)$ such that the matrix $O_j^{\mathrm T}(\lambda)F_j(\lambda)O_j(\lambda)=\mathrm{diag}(d_{j\ell}(\lambda))_{\ell=0}^{N-1}$ is diagonal with non-negative diagonal entries $d_{j\ell}(\lambda)$. Setting $\bar\rho_{j\ell}(t;\lambda)=(O_j(\lambda)\rho(t;\lambda))_\ell$ and $\bar\mu_{j\ell}=d_{j\ell}\mu_j$, we get
\begin{equation}\label{eqSolDiff}
\|\mathbf E_A\partial_t^k\xi_0(t;\tilde y)\|^2 = \sum_{j=0}^J\sum_{\ell=0}^{N-1}\int_A\left(\partial_t^k\bar\rho_{j\ell}(t;\lambda)\right)^2\d\bar\mu_{j\ell}(\lambda).
\end{equation}

Since $\xi_0\colon[0,\infty)\to \mathcal X$ is $N$ times continuously differentiable, it follows from \autoref{eqSolDiff} that
\[ \int_0^{t_0}\int_{[0,\infty)}\left(\partial_t^k\bar\rho_{j\ell}(t;\cdot)\right)^2\d\bar\mu_{j\ell}(\lambda)\d t < \infty\text{ for every }k\in\{0,\ldots,N\}, \]
and therefore, there exists a set $\Lambda_{j\ell}\subset[0,\infty)$ with $\bar\mu_{j\ell}([0,\infty)\setminus\Lambda_{j\ell})=0$ such that
\[ \int_0^{t_0}\left(\partial_t^k\bar\rho_{j\ell}(t;\lambda)\right)^2\d t < \infty\text{ for every }\lambda\in\Lambda_{j\ell}\text{ and every }k\in\{0,\ldots,N\}. \]
So, $\bar\rho_{j\ell}(\cdot;\lambda)$ is for every $\lambda\in\Lambda_{j\ell}$ in the Sobolev space $H^N([0,t_0],\bar\mu_{j\ell})$. By the Sobolev embedding theorem, see, e.g., \cite[Theorem 5.4]{Ada75}, we thus have that $\partial_t^k\bar\rho_{j\ell}(\cdot;\lambda)$ extends for every $\lambda\in\Lambda_{j\ell}$ and every $k\in\{0,\ldots,N-1\}$ continuously to a function on $[0,t_0]$.

Since $\xi_0$ is the difference of two solutions of \autoref{eqODE1}, we have in particular that
\[ \lim_{t\to0}\|\partial_t^k\xi_0(t;\tilde y)\|^2 = 0\text{ for every } k\in\{0,\ldots,N-1\}. \]
Thus, \autoref{eqSolDiff} implies that $\partial_t^k\bar\rho_{j\ell}(t;\cdot)\to0$ in $L^2([0,\infty),\bar\mu_{j\ell})$ with respect to the norm topology as $t\to0$.
Because of the continuity of $\partial_t^k\bar\rho_{j\ell}(\cdot;\lambda)$, this means that there exists a set $\tilde\Lambda_{j\ell}$ with $\bar\mu_{j\ell}([0,\infty)\setminus\tilde\Lambda_{j\ell})=0$ such that we have for every $k\in\{0,\ldots,N-1\}$:
\[ \lim_{t\to0}\partial_t^k\bar\rho_{j\ell}(t;\lambda) = 0\text{ for every }\lambda\in\tilde\Lambda_{j\ell}. \]
But since \autoref{eqSpectralODE} has a unique solution, this implies that $\bar\rho_{j\ell}(t;\lambda)=0$ for all $t\in[0,\infty)$, $\lambda\in\tilde\Lambda_{j\ell}$, and therefore, because of \autoref{eqSolDiff}, that $\xi_0(t;\tilde y)=0$ for every $t\in[0,\infty)$, which proves the uniqueness of the solution of \autoref{eqODE1}.
\end{itemize}
\end{proof}
%%% End

%%% End

In the following sections, we want to show for various choices of coefficients $a_k$ that there exists a mapping $T\colon(0,\infty)\to(0,\infty)$ between the regularisation parameter $\alpha$ and the time $t$ such that the solution $\xi$ corresponds to a regularised solution $x_\alpha$, as defined in \autoref{de:error_f}, via
\[ \xi(T(\alpha);\tilde y) = x_\alpha(\tilde y) \]
for some appropriate generator $(r_\alpha)_{\alpha>0}$ of a regularisation method as introduced in \autoref{deGenerator}. Since we have by \autoref{de:error_f} of the regularised solution that
\[ x_\alpha(\tilde y) = r_\alpha(L^*L)L^*\tilde y = \int_{(0,\norm{L}^2]}r_\alpha(\lambda)\d\mathbf E_\lambda L^*\tilde y \]
and the solution $\xi$ is according to \autoref{thSpectral} of the form of \autoref{eqSolutionODE}, this boils down to finding a mapping $T$ such that if we define the functions $r_\alpha$ by
\[ r_\alpha(\lambda) = \rho(T(\alpha);\lambda), \]
they generate a regularisation method in the sense of \autoref{deGenerator}.
%%% End

%%%%%%%%%%%%%%%%%%%%%%%
%%% se:show
%%%%%%%%%%%%%%%%%%%%%%%
\section{Showalter's method}
\label{se:show}

Showalter's method, given by \autoref{eq:Showalter}, is the gradient flow method for the functional $\mathcal J$. According to \autoref{thSpectral}, we rewrite it as a system of first order 
ordinary differential equations for the error function $\tilde\rho$ of the spectral values $\lambda$ of $L^*L$, 
which in this particular case reads
\begin{equation} \label{eq:Showalter_spectral}
\begin{aligned}
\partial_t\tilde{\rho}(t;\lambda)+\lambda\tilde\rho(t;\lambda)&=0 \text{ for all } \lambda\in\ointerval{\infty},\;t \in \ointerval{\infty}, \\
\tilde{\rho}(0;\lambda) &= 1 \text{ for all }\lambda\in\ointerval{\infty}.
\end{aligned}
\end{equation}

%%%%%%%%%%%%%%%%%%%%%%%%%%%%%%
%%% thShoSpectral
%%%%%%%%%%%%%%%%%%%%%%%%%%%%%%
\begin{lemma}\label{thShoSpectral}
The solution $\tilde{\rho}$ of \autoref{eq:Showalter_spectral} is given by
\begin{equation}\label{eqShowSolSpectral}
\tilde{\rho}(t;\lambda) = \e^{-\lambda t}\text{ for all }(t,\lambda) \in \cointerval{\infty} \times \ointerval{\infty}.
\end{equation}
In particular, the solution of Showalter's method, that is, the solution of \autoref{eqODE1} with $N=1$, is given by
\begin{equation}\label{eqShowSol}
\xi(t;\tilde y) = \int_{(0,\norm{L}^2]}\frac{1-\e^{-\lambda t}}\lambda\d\mathbf E_\lambda L^*\tilde y,
\end{equation}
where $A\mapsto\mathbf E_A$ denotes the spectral measure of $L^*L$.
\end{lemma}
\begin{proof}
Clearly, the smooth function $\tilde\rho$ defined in \autoref{eqShowSolSpectral} is the unique solution of \autoref{eq:Showalter_spectral} and the function $\rho$ defined in \autoref{eqRho} is $\rho(t;\lambda)=\frac{1-\e^{-\lambda t}}\lambda$, $t\ge0$, $\lambda>0$. So, \autoref{thSpectral} gives us the solution \autoref{eqShowSol}.
\end{proof}
%%% End

Next, we want to show that, by identifying $\alpha=\frac1t$ as regularisation parameter, the solution $\xi(\frac1\alpha;\tilde y)$ is a regularised solution of the equation $L x=y$ in the sense of \autoref{de:error_f}. For the verification of the property in \autoref{deGenerator}~\ref{enGeneratorBounded} of the regularisation method, it is convenient to be able to estimate the function $1-\e^{-z}$ by $\sqrt z$.

%%%%%%%%%%%%%%%%%%%%%%%%%%%%%%
%%% lem:aux1
%%%%%%%%%%%%%%%%%%%%%%%%%%%%%%
\begin{lemma}\label{lem:aux1}
There exists a constant $\sigma_0\in(0,1)$ such that
\begin{equation}\label{eqExpSqrt}
1-\e^{-z}\leq \sigma_0 \sqrt{z}\text{ for every }z\ge0.
\end{equation}
\end{lemma}
\begin{proof}
We consider the function $f\colon(0,\infty)\to(0,\infty)$, $f(z)=\frac{1-\e^{-z}}{\sqrt{z}}$. Since $\lim_{z\to0}f(z)=0$ and $\lim_{z\to\infty}f(z)=0$, $f$ attains its maximum at the only critical point $z_0>0$ given as the unique solution of the equation
\[ 0=f'(z)= \frac{\e^{-z}}{\sqrt{z}}-\frac{1-\e^{-z}}{2 z^\frac{3}{2}} = \frac{\e^{-z}}{2 z^\frac{3}{2}}(2z+1-\e^z),\; z>0, \]
where the uniqueness follows from the convexity of the exponential function. Since $2z+1>\e^z$ at $z=1$, we know additionally that $z_0>1$.
Therefore, we have in particular
\[ f(z) \le f(z_0) < 1-\e^{-z_0} < 1\text{ for every }z>0, \]
which gives \autoref{eqExpSqrt} upon setting $\sigma_0\coloneqq1-\e^{-z_0}$.
\end{proof}
%%% End

In order to show that Showalter's method is a regularisation method we verify now all the assumptions in \autoref{deGenerator}.

%%%%%%%%%%%%%%%%%%%%%%%%%%%%%%
%%% thShoReg
%%%%%%%%%%%%%%%%%%%%%%%%%%%%%%
\begin{proposition}\label{thShoReg}
Let $\tilde\rho$ be the solution of \autoref{eq:Showalter_spectral} given in \autoref{eqShowSolSpectral}. Then, the functions $(\ra)_{\alpha>0}$ defined by
\begin{equation}\label{eqShoReg}
\ra(\lambda) \coloneqq \frac1\lambda\left(1-\tilde\rho(\tfrac1\alpha;\lambda)\right) = \frac{1-\e^{-\frac\lambda\alpha}}\lambda
\end{equation}
generate a regularisation method in the sense of \autoref{deGenerator}.
\end{proposition}
\begin{proof}
We verify that $(\ra)_{\alpha>0}$ satisfies the four conditions from \autoref{deGenerator}.

\begin{enumerate}
\item
We clearly have $\ra(\lambda)\le\frac1\lambda\le\frac2\lambda$. To prove the second part of the inequality \autoref{deGenerator}~\ref{enGeneratorBounded}, we use \autoref{lem:aux1} and find
\[ \ra(\lambda) \le \frac{\sigma_0}{\sqrt{\alpha \lambda}}, \]
where $\sigma_0\in(0,1)$ denotes the constant found in \autoref{lem:aux1}.
\item
Moreover, the function $\tra$, given by $\tra(\lambda)=\tilde\rho(\frac1\alpha;\lambda)=\e^{-\frac\lambda\alpha}$, is non-negative and monotonically decreasing.
\item
Since $\tra$ is monotonically decreasing and $\alpha\mapsto\tra(\lambda)$ is monotonically increasing, we can choose $\Tra\coloneqq\tra$ to fulfil \autoref{deGenerator}~\ref{enGeneratorErrorReg}.
\item
We have $\Tra(\alpha)=\tra(\alpha)=\e^{-1}<1$ for every $\alpha>0$.
\end{enumerate}
\end{proof}
%%% End

Finally, we check that the common convergence rate functions are compatible with this regularisation method.

%%%%%%%%%%%%%%%%%%%%%%%%%%%%%%
%%% le:sh_comp
%%%%%%%%%%%%%%%%%%%%%%%%%%%%%%
\begin{lemma}\label{le:sh_comp}
The functions $\phiH_\mu$ and $\phiL_\mu$ defined in \autoref{exCR} are for all $\mu>0$ compatible with the regularisation method $(r_\alpha)_{\alpha>0}$, defined by \autoref{eqShoReg}, in the sense of \autoref{deCompatible}.
\end{lemma}
\begin{proof}
According to \autoref{thTransComp}, it is enough to prove that $\phiH_\mu$ is for arbitrary $\mu>0$ compatible with $(r_\alpha)_{\alpha>0}$.
To see this, we remark that
\[ \Tra^2(\lambda) = \e^{-2\frac\lambda\alpha} = F_\mu\left(\frac{\phiH_\mu(\lambda)}{\phiH_\mu(\alpha)}\right)\text{ with }F_\mu(z)=\exp(-2z^{\frac1\mu}). \]
Since $\int_1^\infty\exp(-2z^{\frac1\mu})\d z = \mu\int_1^\infty\e^{-2w}w^{\mu-1}\d w < \infty$ for every $\mu>0$, $F_\mu$ is integrable and thus, $\phiH_\mu$ is compatible with $(r_\alpha)_{\alpha>0}$.
\end{proof}
%%% End

We have thus shown that we can apply \autoref{thCR} to the regularisation method which is induced by \autoref{eq:Showalter}, that is, the regularisation method generated by the functions $(\ra)_{\alpha>0}$ defined in \autoref{eqShoReg}, and the convergence rate functions $\phiH_\mu$ or $\phiL_\mu$ for arbitrary $\mu>0$.
This gives us optimal convergence rates under variational source conditions as defined in \autoref{eqCRvar}, for example.

However, to compare with the literature, see \cite[Example~4.7]{EngHanNeu96}, we formulate the result under the slightly stronger standard source condition, see \autoref{thSsc}.

%%%%%%%%%%%%%%%%%%%%%%%%%%%%%%
%%% thShoCR
%%%%%%%%%%%%%%%%%%%%%%%%%%%%%%
\begin{corollary}\label{thShoCR}
Let $y\in\mathcal R(L)$ be given such that the corresponding minimum norm solution $x^\dag\in \mathcal X$, fulfilling $L x^\dag=y$ and $\|x^\dag\|=\inf\{\norm{x}\mid L x=y\}$, satisfies for some $\mu>0$ the source condition
\begin{equation}\label{eqShoSsc}
\xdag\in\mathcal R\big((L^*L)^{\frac\mu2}\big).
\end{equation}

Then, if $\xi$ is the solution of the initial value problem in \autoref{eq:Showalter},
\begin{enumerate}
\item
there exists a constant $C_1>0$ such that
\[ \norm{\xi(t;y)-x^\dag}^2 \le C_1t^{-\mu}\text{ for all }t>0; \]
\item
there exists a constant $C_2>0$ such that
\[ \inf_{t>0}\norm{\xi(t;\tilde y)-x^\dag}^2 \le C_2\norm{\tilde y-y}^{\frac{2\mu}{\mu+1}}\text{ for all }\tilde y\in \mathcal Y; \]
and
\item
there exists a constant $C_3>0$ such that
\[ \norm{L\xi(t;y)-y}^2 \le C_3t^{-\mu-1}\text{ for all }t>0. \]
\end{enumerate}
\end{corollary}
\begin{proof}
We consider the regularisation method defined by the functions $(r_\alpha)_{\alpha>0}$ from \autoref{eqShoReg}. We have already seen in \autoref{thConvRatesSub} and \autoref{le:sh_comp} that the function $\phiH_\mu(\alpha) = \alpha^\mu$ is $G$-subhomogeneous in the sense of \autoref{eqCRsubhomogeneous} with $G(\gamma)=\gamma^\mu$ and compatible with the regularisation method given by $(\ra)_{\alpha>0}$.
\begin{enumerate}
\item
According to \autoref{thSsc} and \autoref{thCR} with the convergence rate function $\varphi=\phiH_\mu$, the source condition in \autoref{eqShoSsc} implies the existence of a constant~$C_d$ such that
\[ d(\alpha) \le C_d\phiH_\mu(\alpha) = C_d\alpha^\mu, \]
where $d$ is given by \autoref{eq:dD} with the regularised solution $\xa$ defined in \autoref{eq:reg} fulfilling according to \autoref{eqShoReg} and \autoref{eqShowSol} that
\begin{equation}\label{eqShoRegSol}
\xa(\tilde y)=\ra(L^*L)L^*\tilde y=\int_{(0,\norm{L}^2]}\frac{1-\e^{-\frac\lambda\alpha}}\lambda\d\mathbf E_\lambda L^*\tilde y = \xi(\tfrac1\alpha;\tilde y).
\end{equation}
Thus, by definition of $d$, we have that
\[ \norm{\xi(t;y)-x^\dag}^2 = \norm{x_{\frac1t}(y)-x^\dag}^2 = d(\tfrac1t) \le \frac{C_d}{t^\mu}\text{ for every }t>0. \]
\item
According to \autoref{thCR}, we also find a constant $C_{\tilde d}$ such that
\[ \tilde d(\delta) \le C_{\tilde d}\Phi[\phiH_\mu](\delta) = C_{\tilde d}\delta^{\frac{2\mu}{\mu+1}}, \]
where $\Phi$ denotes the noise-free to noisy transform defined in \autoref{deTransform} and $\tilde d$ is given by \autoref{eq:tilde_dD} with the regularised solution $\xa$ given by \autoref{eqShoRegSol}. Therefore, we have that
\[ \inf_{t>0}\norm{\xi(t;\tilde y)-x^\dag}^2 = \inf_{\alpha>0}\norm{\xi(\tfrac1\alpha;\tilde y)-x^\dag}^2\le \tilde d(\norm{\tilde y-y}) \le C_{\tilde d}\norm{\tilde y-y}^{\frac{2\mu}{\mu+1}}\text{ for every }\tilde y\in \mathcal Y. \]
\item
Furthermore, \autoref{thCR} implies that there is a constant $C_e>0$ such that $e(\lambda)\le C_e\phiH_\mu(\lambda)$. In particular, we then have $\lambda e(\lambda)\le \phiH_{\mu+1}(\lambda)$. And since $\phiH_{\mu+1}$ is by \autoref{le:sh_comp} compatible with $(r_\alpha)_{\alpha>0}$, we can apply \autoref{thCRImageSimple} and find a constant $C>0$ such that the function $q$, defined in \autoref{eqQ} with the regularised solution $\xa$ as in \autoref{eqShoRegSol}, fulfils
\[ q(\alpha) \le C\phiH_{\mu+1}(\alpha)\text{ for all }\alpha>0. \]
Thus, by definition of $q$, we have
\[ \norm{L\xi(t;y)-y}^2 = \norm{L x_{\frac1t}(y)-y}^2 = q(\tfrac1t) \le \frac C{t^{\mu+1}}\text{ for all }t>0. \]
\end{enumerate}
\end{proof}
%%% End

We emphasise that for Showalter's method we did {\em not} make use of the extended theory involving envelopes of regularisation methods (cf.~\autoref{de:error_f}), and this theory could have been developed also with the regularisation results from \cite{AlbElbHooSch16}.
%%% End

%%%%%%%%%%%%%%%%%%%%%%%%%%%%%%
%%% seSecOrd
%%%%%%%%%%%%%%%%%%%%%%%%%%%%%%
\section{Heavy Ball Dynamics}\label{seSecOrd}
The heavy ball method consists of the \autoref{eqODE} for $N=2$ and $a_1(t)=b$ for some $b>0$, that is, \autoref{eqODESecond}.

According to \autoref{thSpectral}, this corresponds to the initial value problems for every $\lambda > 0$
\begin{equation}\label{eqSecondOrder}
\begin{aligned}
\partial_{t t}\tilde{\rho}(t;\lambda) + b\partial_t\tilde{\rho}(t;\lambda) +\lambda \tilde{\rho}(t;\lambda)&=0 \text{ for all } t \in \ointerval{\infty},\\
\partial_t\tilde{\rho}(0;\lambda)&=0, \\
\tilde{\rho}(0;\lambda)&=1.
\end{aligned}
\end{equation}

%%%%%%%%%%%%%%%%%%%%%%%%%%%%%%
%%% thSecOrdSpectral
%%%%%%%%%%%%%%%%%%%%%%%%%%%%%%
\begin{lemma}\label{thSecOrdSpectral}
The solution of \autoref{eqSecondOrder} is given by
\begin{equation}\label{eqSecTildeRho}
\tilde{\rho}(t;\lambda) = 
\begin{cases}
\e^{-\frac{b t}2}\left(\cosh\left(\beta_-(\lambda)\frac{b t}2\right)+\frac1{\beta_-(\lambda)}\sinh\left(\beta_-(\lambda)\frac{b t}2\right)\right)&\text{if}\;\lambda\in(0,\frac{b^2}4),\\ 
\e^{-\frac{b t}2}\left(\cos\left(\beta_+(\lambda)\frac{b t}2\right)+\frac1{\beta_+(\lambda)}\sin\left(\beta_+(\lambda)\frac{b t}2\right)\right)&\text{if}\;\lambda\in(\frac{b^2}4,\infty),\\ 
\e^{-\frac{b t}2}(1+\frac{b t}2)&\text{if}\;\lambda=\frac{b^2}4,
\end{cases}
\end{equation}
where
\begin{equation}\label{eqSecBeta}
\beta_-(\lambda)=\sqrt{1-\frac{4\lambda}{b^2}}\text{ and }\beta_+(\lambda)=\sqrt{\frac{4\lambda}{b^2}-1},
\end{equation}
see \autoref{fgSecOrdGraph}. In particular, the solution of \autoref{eqODESecond} is given by
\begin{equation}\label{eqSecOrdSol}
\xi(t;\tilde y) = \int_{(0,\|L\|^2]}\frac{1-\tilde\rho(t;\lambda)}\lambda\d\mathbf E_\lambda L^*\tilde y,
\end{equation}
where $A\mapsto\mathbf E_A$ denotes the spectral measure of $L^*L$.
\end{lemma}
\begin{proof}
The characteristic equation of \autoref{eqSecondOrder} is
\[ z^2(\lambda)+b z(\lambda)+\lambda = 0 \]
and has the solutions
\[ z_1(\lambda)=-\frac b2-\sqrt{\frac{b^2}4-\lambda}\text{ and } z_2(\lambda)=-\frac b2+\sqrt{\frac{b^2}4-\lambda}. \]
Thus, for $\lambda<\frac{b^2}4$, we have the solution
\[ \tilde{\rho}(t;\lambda) = \e^{-\frac{b t}2}\left(C_1(\lambda)\cosh\left(t\sqrt{\frac{b^2}4-\lambda}\right)+C_2(\lambda)\sinh\left(t\sqrt{\frac{b^2}4-\lambda}\right)\right); \]
for $\lambda>\frac{b^2}4$, we get the oscillating solution
\[ \tilde{\rho}(t;\lambda) = \e^{-\frac{b t}2}\left(C_1(\lambda)\cos\left(t\sqrt{\lambda-\frac{b^2}4}\right)+C_2(\lambda)\sin\left(t\sqrt{\lambda-\frac{b^2}4}\right)\right); \]
and for $\lambda=\frac{b^2}4$, we have
\[ \tilde{\rho}(t;\lambda) = \e^{-\frac{b t}2}(C_1(\lambda)+C_2(\lambda)t). \]
Plugging in the initial condition $\tilde{\rho}(0;\lambda)=1$, we find that $C_1(\lambda)=1$ for all $\lambda>0$, and the initial condition $\partial_t\tilde{\rho}(0;\lambda)=0$ then implies
\begin{align*}
C_2(\lambda)\sqrt{\frac{b^2}4-\lambda}&=\frac b2\text{ for }\lambda<\frac{b^2}4, \\
C_2(\lambda)\sqrt{\lambda-\frac{b^2}4}&=\frac b2\text{ for }\lambda>\frac{b^2}4,\text{ and}\\
C_2(\tfrac{b^2}4)&=\frac b2.
\end{align*}
Moreover, since $\tilde\rho$ is smooth and the unique solution of \autoref{eqSecondOrder}, the function $\xi$ defined in \autoref{eqSecOrdSol} is by \autoref{thSpectral} the unique solution of \autoref{eqODESecond}.
\end{proof}
%%% End

%%%%%%%%%%%%%%%%%%%%%%%%%%%%%%
%%% fgSecOrdGraph
%%%%%%%%%%%%%%%%%%%%%%%%%%%%%%
\begin{figure}
\hfil
\parbox{0.45\textwidth}{\includegraphics{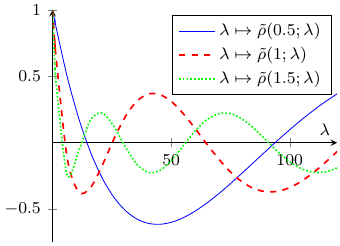}}
\hfil
\parbox{0.45\textwidth}{\includegraphics{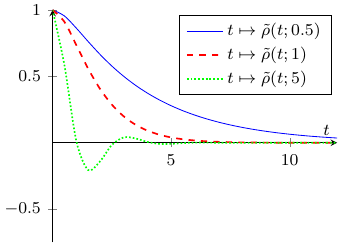}}
\hfil
\caption{Graphs for the function $\tilde{\rho}$ for the value $b=2$.
The non-monotonicity of the functions $t\mapsto\tilde{\rho}(t;\lambda)$ and $\lambda\mapsto\tilde{\rho}(t;\lambda)$ requires to compare the rates with the regularisation methods derived from the envelope.}\label{fgSecOrdGraph}
\end{figure}
%%% End

To see that this solution gives rise to a regularisation method as introduced in \autoref{deGenerator}, we first verify that the function $\lambda\mapsto\tilde\rho(t;\lambda)$, which corresponds to the error function $\tra$ in \autoref{deGenerator}, is non-negative and monotonically decreasing for sufficiently small values of $\lambda$ as required for $\tra$ in \autoref{deGenerator}~\ref{enGeneratorError}.

%%%%%%%%%%%%%%%%%%%%%%%%%%%%%%
%%% thSecOrdDecreasing
%%%%%%%%%%%%%%%%%%%%%%%%%%%%%%
\begin{lemma}\label{thSecOrdDecreasing}
The function $\lambda\mapsto \tilde{\rho}(t;\lambda)$ defined by \autoref{eqSecTildeRho} is for every $t\in(0,\infty)$ non-negative and monotonically decreasing on the interval $(0,\frac{b^2}4+\frac{\pi^2}{4t^2})$.
\end{lemma}
\begin{proof}
We proof this separately for $\lambda\in(0,\tfrac{b^2}4)$ and for $\lambda\in(\frac{b^2}4,\frac{b^2}4+\frac{\pi^2}{4t^2})$.
\begin{itemize}
\item
We remark that the function
\[ g_\tau\colon(0,\infty)\to\R,\; g_\tau(\beta)=\cosh(\beta\tau)+\frac{\sinh(\beta\tau)}\beta, \]
is non-negative and fulfils for arbitrary $\tau>0$ that
\[ g_\tau'(\beta) = \tau\sinh(\beta\tau)+\frac{\tau\cosh(\beta\tau)}\beta-\frac{\sinh(\beta\tau)}{\beta^2} = \tau\sinh(\beta\tau)+\frac{\cosh(\beta\tau)}{\beta^2}(\beta\tau-\tanh(\beta\tau)) \ge 0, \]
since $\tanh(z)\le z$ for all $z\ge0$.
Thus, writing the function $\tilde\rho$ for $\lambda\in(0,\frac{b^2}4)$ with the function $\beta_-$ given by \autoref{eqSecBeta} in the form
\[\tilde{\rho}(t;\lambda)=\e^{-\frac{b t}2}g_{\frac{b t}2}(\beta_-(\lambda)), \]
we find that
\[ \partial_\lambda \tilde{\rho}(t;\lambda) = \e^{-\frac{b t}2}g_{\frac{b t}2}'(\beta_-(\lambda))\beta_-'(\lambda) \le 0, \]
since $\beta_-'(\lambda)=-\frac2{b^2\beta_-(\lambda)}\le0$. Therefore, the function $\lambda\mapsto\tilde\rho(t;\lambda)$ is non-negative and monotonically decreasing on $(0,\frac{b^2}4)$.

\item
Similarly, we consider for $\lambda\in(\frac{b^2}4,\infty)$ the function
\[ G_\tau\colon(0,\infty)\to\R,\; G_\tau(\beta)=\cos(\beta\tau)+\frac{\sin(\beta\tau)}\beta, \]
for arbitrary $\tau>0$. Since $\lim_{\beta\to0}G_\tau(\beta)=1+\tau>0$ and since the smallest zero $\beta_\tau$ of $G_\tau$ is the smallest non-negative solution of the equation $\tan(\beta\tau)=-\beta$, implying that $\beta_\tau\tau\in(\frac\pi2,\pi)$, we have that $G_\tau(\beta)\ge0$ for all $\beta\in(0,\frac\pi{2\tau})\subset(0,\beta_\tau)$.

Moreover, the derivative of $G_\tau$ satisfies for every $\beta\in(0,\tfrac\pi{2\tau})$ that
\[ G_\tau'(\beta) = -\tau\sin(\beta\tau)+\frac{\tau\cos(\beta\tau)}\beta-\frac{\sin(\beta\tau)}{\beta^2} = -\frac{\cos(\beta\tau)}{\beta^2}\left((\beta^2\tau+1)\tan(\beta\tau)-\beta\tau\right) \le 0, \]
since $\tan(z)\ge z$ for every $z\ge0$. Therefore, we find for the function $\tilde\rho$ on the domain $(0,\infty)\times(\frac{b^2}4,\infty)$, where it has the form
\[\tilde{\rho}(t;\lambda)=\e^{-\frac{b t}2}G_{\frac{b t}2}(\beta_+(\lambda)) \]
with $\beta_+$ given by \autoref{eqSecBeta}, that
\[ \tilde\rho(t;\lambda)\ge0\text{ and }\partial_\lambda \tilde{\rho}(t;\lambda) = \e^{-\frac{b t}2}G_{\frac{b t}2}'(\beta_+(\lambda))\beta_+'(\lambda) \le 0\text{ for }\beta_+(\lambda)<\frac\pi{b t},\text{ that is, for }\lambda<\frac{b^2}4+\frac{\pi^2}{4t^2}, \]
since $\beta_+'(\lambda)=\frac2{b^2\beta_+(\lambda)}\ge0$.
\end{itemize}

Because $\tilde{\rho}$ is continuous, this implies that $\lambda\mapsto \tilde{\rho}(t;\lambda)$ is for every $t\in(0,\infty)$ non-negative and monotonically decreasing on $(0,\frac{b^2}4+\frac{\pi^2}{4t^2})$.
\end{proof}
%%% End

In a next step, we introduce the function $\tilde P(t;\cdot)$ as a correspondence to the upper bound $\Tra$ and show that it fulfils the properties necessary for \autoref{deGenerator}~\ref{enGeneratorErrorReg}.

%%%%%%%%%%%%%%%%%%%%%%%%%%%%%%
%%% thSecOrdEnvelope
%%%%%%%%%%%%%%%%%%%%%%%%%%%%%%
\begin{lemma}\label{thSecOrdEnvelope}
We define the function
\begin{equation}\label{eqSecTildeP}
\tilde P(t;\lambda) = 
\begin{cases}
\e^{-\frac{b t}2}\left(\cosh\left(\beta_-(\lambda)\frac{b t}2\right)+\frac1{\beta_-(\lambda)}\sinh\left(\beta_-(\lambda)\frac{b t}2\right)\right)&\text{if}\;\lambda\in(0,\frac{b^2}4),\\ 
\e^{-\frac{b t}2}(1+\frac{b t}2)&\text{if}\; \lambda\in [\frac{b^2}4,\infty),
\end{cases}
\end{equation}
where the function $\beta_-$ shall be given by \autoref{eqSecBeta}.

Then, $\tilde P$ is an upper bound for the absolute value of the function $\tilde\rho$ defined by \autoref{eqSecTildeRho}: $\tilde P\ge|\tilde{\rho}|$.
\end{lemma}
\begin{proof}
Since $\tilde{\rho}(t;\lambda)=\tilde P(t;\lambda)$ for $\lambda\le\frac{b^2}4$ for every $t>0$, we only need to consider the case $\lambda>\frac{b^2}4$. Using that $\left|\cos(z)\right|\le1$ and $\left|\sin(z)\right|\le|z|$ for all $z\in\R$, we find with $\beta_+$ as in \autoref{eqSecBeta} for every $\lambda>\frac{b^2}4$ and every $t>0$ that
\[ |\tilde{\rho}(t;\lambda)|=\e^{-\frac{b t}2}\left|\cos\left(\beta_+(\lambda)\frac{b t}2\right)+\frac1{\beta_+(\lambda)}\sin\left(\beta_+(\lambda)\frac{b t}2\right)\right| \le \e^{-\frac{b t}2}\left(1+\frac{b t}2\right) = \tilde P(t;\lambda). \]
\end{proof}
%%% End

%%%%%%%%%%%%%%%%%%%%%%%%%%%%%%
%%% thSecOrdUpper
%%%%%%%%%%%%%%%%%%%%%%%%%%%%%%
\begin{lemma}\label{thSecOrdUpper}
Let $\tilde P$ be given by \autoref{eqSecTildeP}. Then, $\lambda\mapsto\tilde P(t;\lambda)$ is monotonically decreasing and $t\mapsto \tilde P(t;\lambda)$ is strictly decreasing.
\end{lemma}
\begin{proof}
For the derivative of $\tilde P$ with respect to $t$, we get
\[ \partial_t \tilde P(t;\lambda) = 
	\begin{cases}
		\frac b2\e^{-\frac{b t}2}\left(\beta_-(\lambda)-\frac1{\beta_-(\lambda)}\right)\sinh\left(\beta_-(\lambda)\frac{b t}2\right)
			&\text{if }\lambda\in(0,\frac{b^2}4),\\
		-\frac{b^2t}4\e^{-\frac{b t}2}
			&\text{if }\lambda\in[\frac{b^2}4,\infty),
	\end{cases} \]
with $\beta_-$ defined in \autoref{eqSecBeta}; and since $\beta_-(\lambda)\in(0,1)$ for every $\lambda\in(0,\frac{b^2}4)$, we thus have $\partial_t \tilde P(t;\lambda)<0$ for every $t>0$ and every $\lambda>0$.

Since $\tilde P(t;\lambda)=\tilde\rho(t;\lambda)$ for $\lambda\in(0,\frac{b^2}4]$, where $\tilde\rho$ denotes the solution of \autoref{eqSecondOrder}, given by \autoref{eqSecTildeRho}, we already know from \autoref{thSecOrdDecreasing} that $\lambda\mapsto \tilde P(t;\lambda)$ is monotonically decreasing on $(0,\frac{b^2}4]$. And since $\lambda\mapsto \tilde P(t;\lambda)$ is constant on $[\frac{b^2}4,\infty)$, it is monotonically decreasing on $(0,\infty)$.
\end{proof}
%%% End

To verify later the compatibility of the convergence rate functions $\phiH_\mu$ and $\phiL_\mu$ introduced in \autoref{eqCRHoelder} and \autoref{eqCRLog}, we derive here an appropriate upper bound for $\tilde P$.

%%%%%%%%%%%%%%%%%%%%%%%%%%%%%%
%%% thSecOrdEnvUpper
%%%%%%%%%%%%%%%%%%%%%%%%%%%%%%
\begin{lemma}\label{thSecOrdEnvUpper}
We have for every $\Lambda>0$ that the function $\tilde P$ defined in \autoref{eqSecTildeP} can be bounded from above by
\[ \tilde P(t;\lambda) \le \Psi_\Lambda(\lambda t)\text{ for all }t>0,\;\lambda\in(0,\Lambda], \]
where
\begin{equation}\label{eqPsi}
\Psi_\Lambda(z) = \max\left\{2\e^{-\frac z b},\e^{-\frac{b z}{2\Lambda}}\left(1+\frac{b z}{2\Lambda}\right)\right\}.
\end{equation}
\end{lemma}
\begin{proof}
We consider the two cases $\lambda\in(0,\frac{b^2}4)$ and $\lambda\in[\frac{b^2}4,\Lambda]$ separately.
\begin{itemize}
\item
For $\lambda\in(0,\frac{b^2}4)$, we use the two inequalities $\cosh(z)\le\e^z$ and $\frac{\sinh(z)}z\le\e^z$ for all $z\ge0$, where the latter follows from the fact that $f(z)=2z\e^z(\e^z-\frac{\sinh(z)}z)=(2z-1)\e^{2z}+1$ is because of $f'(z)=4z\e^{2z}\ge0$ monotonically increasing on $[0,\infty)$ and thus fulfils $f(z)\ge f(0)=0$ for every $z\ge0$. With this, we find from \autoref{eqSecTildeP} that
\[ \tilde P(t;\lambda) \le 2\exp\left(\left(\sqrt{1-\frac{4\lambda}{b^2}}-1\right)\frac{b t}2\right). \]
Since $\sqrt{1-z}\le 1-\frac z2$ for all $z\in(0,1)$, we then obtain
\[ \tilde P(t;\lambda) \le 2\e^{-\frac{\lambda t}b}\text{ for every }t>0,\;\lambda\in(0,\tfrac{b^2}4). \]
\item
For $\lambda\in[\frac{b^2}4,\Lambda]$, we use that $t\mapsto\tilde P(t;\lambda)$ is according to \autoref{thSecOrdUpper} for every $\lambda\in(0,\infty)$ monotonically decreasing and obtain from \autoref{eqSecTildeP} that
\[ \tilde P(t;\lambda) \le \tilde P\left(\frac{\lambda t}\Lambda;\lambda\right) = \e^{-\frac{b\lambda t}{2\Lambda}}\left(1+\frac{b\lambda t}{2\Lambda}\right)\text{ for every }t>0. \]
\end{itemize}
\end{proof}
%%% End

Next, we give an upper bound for the function $\rho$, $\rho(t;\lambda)\coloneqq\frac1\lambda(1-\tilde\rho(t;\lambda))$, which allows us to verify the property in \autoref{deGenerator}~\ref{enGeneratorBounded} for the corresponding generator $(\ra)_{\alpha>0}$ of the regularisation method.

%%%%%%%%%%%%%%%%%%%%%%%%%%%%%%
%%% thSecOrdLowerBound
%%%%%%%%%%%%%%%%%%%%%%%%%%%%%%
\begin{lemma}\label{thSecOrdLowerBound}
Let $\tilde\rho$ be given by \autoref{eqSecTildeRho}. Then, there exists a constant $\sigma_1\in(0,1)$ such that
\[ \frac1\lambda(1-\tilde\rho(t;\lambda)) \le \sigma_1{\sqrt{\frac{2t}{b\lambda}}}\text{ for all }t>0,\;\lambda>0. \]
\end{lemma}

\begin{proof}
We consider the two cases for $\lambda\in(0,\frac{b^2}4)$ and $\lambda\in(\frac{b^2}4,\infty)$ separately. The estimate for $\lambda=\frac{b^2}4$ then follows directly from the fact that the function $\lambda\mapsto\tilde\rho(t;\lambda)$ is continuous for every $t\in[0,\infty)$.
\begin{itemize}
\item
For $\lambda\in(0,\frac{b^2}4)$, we use that $\cosh(z)=\e^z-\sinh(z)$ for every $z\in\R$ and obtain with the function $\beta_-$ from \autoref{eqSecBeta} that
\[ \frac1\lambda(1-\tilde\rho(t;\lambda)) = \frac1\lambda\left(1-\e^{-(1-\beta_-(\lambda))\frac{b t}2}-\left(\frac1{\beta_-(\lambda)}-1\right)\e^{-\frac{b t}2}\sinh(\beta_-(\lambda)\tfrac{b t}2)\right). \]
Since $\beta_-(\lambda)\in(0,1)$, we can therefore estimate this with the help of \autoref{lem:aux1} by
\[ \frac1\lambda(1-\tilde\rho(t;\lambda)) \le \frac1\lambda\left(1-\e^{-(1-\beta_-(\lambda))\frac{b t}2}\right) \le \frac{\sigma_0}\lambda\sqrt{1-\beta_-(\lambda)}\sqrt{\frac{b t}2}, \]
where $\sigma_0\in(0,1)$ is the constant found in \autoref{lem:aux1}. Since $\lambda=\frac{b^2}4(1-\beta_-^2(\lambda))$, this means
\[ \frac1\lambda(1-\tilde\rho(t;\lambda)) \le \frac{\sigma_0}{\sqrt{1+\beta_-(\lambda)}}\sqrt{\frac{2t}{b\lambda}} \le \sigma_0\sqrt{\frac{2t}{b\lambda}}. \]

\item
For $\lambda\in(\frac{b^2}4,\infty)$, we remark that
\[ \partial_t\tilde\rho(t;\lambda) = -\frac b2\left(\beta_+(\lambda)+\frac1{\beta_+(\lambda)}\right)\e^{-\frac{b t}2}\sin\left(\beta_+(\lambda)\frac{b t}2\right), \]
where $\beta_+$ is given by \autoref{eqSecBeta}. Since the function $[0,\infty)\to\R$, $z\mapsto (\e^{-z}\sin(a z))^2$, $a>0$, attains its maximal value at its smallest non-negative critical point $z=\frac1a\arctan(a)$, we have that
\[ \left|\partial_t\tilde\rho(t;\lambda)\right| \le \frac b2\left(\beta_+(\lambda)+\frac1{\beta_+(\lambda)}\right)\e^{-\frac{\arctan(\beta_+(\lambda))}{\beta_+(\lambda)}}\left|\sin(\arctan(\beta_+(\lambda)))\right|. \]
Using that $\sin(z)=\frac{\tan(z)}{\sqrt{1+\tan^2(z)}}$ for all $z\in(-\frac\pi2,\frac\pi2)$, this reads
\begin{equation}\label{eqSecOrdLowerBoundStep}
\left|\partial_t\tilde\rho(t;\lambda)\right| \le \frac b2\sqrt{1+\beta_+^2(\lambda)}\,\e^{-\frac{\arctan(\beta_+(\lambda))}{\beta_+(\lambda)}}.
\end{equation}

We further realise that the function $f\colon(0,\infty)\to\R$, $f(z)=\frac1{\sqrt{1+z^2}}\e^{-\frac{\arctan(z)}{z}}$, is monotonically decreasing because of
\begin{align*}
f'(z) &= -\frac1{\sqrt{1+z^2}}\e^{-\frac{\arctan(z)}{z}}\left(\frac z{1+z^2}+\frac1{z(1+z^2)}-\frac{\arctan(z)}{z^2}\right) \\
&= -\frac1{z^2\sqrt{1+z^2}}\e^{-\frac{\arctan(z)}{z}}(z-\arctan(z)) \le 0.
\end{align*}
Thus, $f(z)\le\lim_{z\to0}f(z)=\e^{-1}$ and \autoref{eqSecOrdLowerBoundStep} therefore implies that
\[ \left|\partial_t\tilde\rho(t;\lambda)\right| \le \frac b{2\e}(1+\beta_+^2(\lambda)). \]
With $\frac4{b^2}\lambda=(1+\beta_+^2(\lambda))$, the mean value theorem therefore gives us
\[ \frac1\lambda(1-\tilde\rho(t;\lambda)) = \frac1\lambda(\tilde\rho(0;\lambda)-\tilde\rho(t;\lambda)) \le \frac{2t}{\e b}\text{ for all }t>0. \]
Since we know from \autoref{thSecOrdEnvelope} and \autoref{thSecOrdUpper} that we can estimate $\tilde\rho$ with the function $\tilde P$ from \autoref{eqSecTildeP} by
\begin{equation}\label{eqSecOrdUpperBound}
|\tilde\rho(t;\lambda)|\le \tilde P(t;\lambda)\le \tilde P(0;\lambda)=1,
\end{equation}
we find by using the estimate $\min\{a,b\}\le\min\{\sqrt a,\sqrt b\}\max\{\sqrt a,\sqrt b\}=\sqrt{ab}$ for all $a,b>0$ that 
\[ \frac1\lambda(1-\tilde\rho(t;\lambda))\le\min\left\{\frac2\lambda,\frac{2t}{\e b}\right\} \le\sqrt{\frac2\e}\sqrt{\frac{2t}{b\lambda}} \text{ for all }t>0. \]
\end{itemize}
\end{proof}
%%% End

Finally, we can put together all the estimates to obtain a regularisation method corresponding to the solution $\xi$ of the heavy ball equation, \autoref{eqODESecond}.

%%%%%%%%%%%%%%%%%%%%%%%%%%%%%%
%%% thSecOrdReg
%%%%%%%%%%%%%%%%%%%%%%%%%%%%%%
\begin{proposition}\label{thSecOrdReg}
Let $\tilde{\rho}$ be the solution of \autoref{eqSecondOrder}. 
Then, the functions $(\ra)_{\alpha>0}$,
\begin{equation}\label{eqSecOrdGenerator}
\ra(\lambda) \coloneqq \frac1\lambda(1-\tilde{\rho}(\tfrac{b}{2\alpha};\lambda)),
\end{equation}
define a regularisation method in the sense of \autoref{deGenerator}.
\end{proposition}
\begin{proof}
We verify the four conditions in \autoref{deGenerator}.
\begin{enumerate}
\item
We have already seen in \autoref{eqSecOrdUpperBound} that $|\tilde{\rho}(t;\lambda)|\le1$ and thus $\ra(\lambda)\le\frac2\lambda$ for every $\lambda>0$.

Moreover, \autoref{thSecOrdLowerBound} implies that there exists a parameter $\sigma_1\in(0,1)$ such that
\[ \ra(\lambda) = \frac1\lambda(1-\tilde{\rho}(\tfrac{b}{2\alpha};\lambda)) \le \frac{\sigma_1}{\sqrt{\alpha\lambda}}, \]
which is \autoref{eqGeneratorBounded}.
\item
The corresponding error function
\[ \tra\colon(0,\infty)\to[-1,1],\;\tra(\lambda)=\tilde\rho(\tfrac{b}{2\alpha};\lambda), \]
is according to \autoref{thSecOrdDecreasing} non-negative and monotonically decreasing on $(0,\frac{b^2}4+\frac{\pi^2\alpha^2}{b^2})$. Using that $a^2+b^2\ge2ab$ for all $a,b\in\R$, we find that
\[ \frac{b^2}4+\frac{\pi^2\alpha^2}{b^2} \ge 2\sqrt{\frac{\pi^2\alpha^2}{b^2}\,\frac{b^2}4} = \pi\alpha > \alpha, \]
which implies that $\tra$ is for every $\alpha>0$ non-negative and monotonically decreasing on $(0,\alpha)$.
\item 
Choosing
\begin{equation}\label{eqSecOrdGenEnv}
\Tra(\lambda) \coloneqq \tilde P(\tfrac{b}{2\alpha};\lambda)
\end{equation}
with the function $\tilde P$ from \autoref{eqSecTildeP}, we know from \autoref{thSecOrdEnvelope} that $\Tra(\lambda)\ge\abs{\tra(\lambda)}$ holds for all $\lambda>0$ and $\alpha>0$. Moreover, \autoref{thSecOrdUpper} tells us that $\Tra$ is for every $\alpha>0$ monotonically decreasing and that $\alpha\mapsto\Tra(\alpha;\lambda)$ is for every $\lambda>0$ monotonically increasing.

\item 
To estimate the values $\Tra(\alpha)$ for $\alpha$ in a neighbourhood of zero, we calculate the limit
\[ \lim_{\alpha\to0}\Tra(\alpha) = \lim_{\alpha\to0}\tilde P(\tfrac{b}{2\alpha};\alpha)) = \lim_{\alpha\to0}\e^{-\frac{b^2}{4\alpha}}\left(\cosh\left(\beta_-(\alpha)\frac{b^2}{4\alpha}\right)+\frac1{\beta_-(\alpha)}\sinh\left(\beta_-(\alpha)\frac{b^2}{4\alpha}\right)\right), \]
where $\beta_-$ is given by \autoref{eqSecBeta}. Setting $\tilde\alpha=\frac{4\alpha}{b^2}$ and using that then $\beta_-(\alpha)=\sqrt{1-\frac{4\alpha}{b^2}}=\sqrt{1-\tilde\alpha}$, we get that
\begin{align*}
\lim_{\alpha\to0}\Tra(\alpha) &= \lim_{\tilde\alpha\to0}\e^{-\frac1{\tilde\alpha}}\left(\cosh\left(\frac{\sqrt{1-\tilde\alpha}}{\tilde\alpha}\right)+\frac1{\sqrt{1-\tilde\alpha}}\sinh\left(\frac{\sqrt{1-\tilde\alpha}}{\tilde\alpha}\right)\right) \\
&= \lim_{\tilde\alpha\to0}\frac12\left(1+\frac1{\sqrt{1-\tilde\alpha}}\right)\e^{\frac{\sqrt{1-\tilde\alpha}-1}{\tilde\alpha}} = \e^{-\frac12} < 1.
\end{align*}
Thus, there exists for an arbitrarily chosen $\tilde\sigma_0\in(\e^{-\frac12},1)$ a parameter $\bar\alpha_0>0$ such that $\Tra(\alpha)\le\tilde\sigma_0$ for every $\alpha\in(0,\bar\alpha_0)$.

Using further that $t\mapsto\tilde P(t;\lambda)$ is strictly decreasing, see \autoref{thSecOrdUpper}, we have for every $\alpha>0$ that
\[ \Tra(\alpha)=\tilde P(\tfrac b{2\alpha};\alpha) < \tilde P(0;\alpha)=1. \]
Thus, since $\alpha\mapsto\Tra(\alpha)$ is by definition of $\tilde P$ in \autoref{eqSecTildeP} continuous on $(0,\infty)$, we have for every $\bar\alpha>0$ that
\[ \sup_{\alpha\in(0,\bar\alpha]}\Tra(\alpha) = \max\bigg\{\sup_{\alpha\in(0,\bar\alpha_0)}\Tra(\alpha),\sup_{\alpha\in[\bar\alpha_0,\bar\alpha]}\Tra(\alpha)\bigg\}\le\max\bigg\{\tilde\sigma_0,\max_{\alpha\in[\bar\alpha_0,\bar\alpha]}\Tra(\alpha)\bigg\} < 1, \]
which shows \autoref{deGenerator}~\ref{enGeneratorLimit}.
\end{enumerate}
\end{proof}
%%% End

To be able to apply \autoref{thCR} for the regularisation method generated by $(\ra)_{\alpha>0}$ from \autoref{eqSecOrdGenerator} to the common convergence rates $\phiH_\mu$ and $\phiL_\mu$, it remains to show that they are compatible with $(\ra)_{\alpha>0}$.

%%%%%%%%%%%%%%%%%%%%%%%%%%%%%%
%%% thSecOrdCompatibility
%%%%%%%%%%%%%%%%%%%%%%%%%%%%%%
\begin{lemma}\label{thSecOrdCompatibility}
The functions $\phiH_\mu$ and $\phiL_\mu$ defined in \autoref{exCR} are for all $\mu>0$ compatible with the regularisation method $(r_\alpha)_{\alpha>0}$ defined by \autoref{eqSecOrdGenerator} in the sense of \autoref{deCompatible}.
\end{lemma}
\begin{proof}
We know from \autoref{thTransComp} that we only need to prove the statement for $\phiH_\mu$ for every $\mu>0$.
The function $\Tra$ defined in \autoref{eqSecOrdGenEnv} fulfils according to \autoref{thSecOrdEnvUpper} for arbitrary $\Lambda>0$ that 
\[ \Tra^2(\lambda) = \tilde P^2\left(\frac b{2\alpha};\lambda\right) \le \Psi_\Lambda^2\left(\frac{b\lambda}{2\alpha}\right) \le \Psi_\Lambda^2\left(\frac b2\left(\frac{\phiH_\mu(\lambda)}{\phiH_\mu(\alpha)}\right)^{\frac1\mu}\right)\text{ for every }\alpha>0,\;\lambda\in(0,\Lambda], \]
where $\Psi_\Lambda$ is given by \autoref{eqPsi}. Since $z\mapsto\Psi_\Lambda^2(\frac b2z^{\frac1\mu})$ is for every $\mu>0$ integrable, $\phiH_\mu$ is compatible with $(\ra)_{\alpha>0}$.
\end{proof}
%%% End

We can therefore apply \autoref{thCR} to the regularisation method induced by \autoref{eqODESecond}, which is the regularisation method generated by the functions $(\ra)_{\alpha>0}$ defined in \autoref{eqSecOrdGenerator}, and the convergence rate functions $\phiH_\mu$ or $\phiL_\mu$ for arbitrary $\mu>0$.
Thus, although the functions $t\mapsto\tilde\rho(t;\lambda)$ and $\lambda\mapsto\tilde\rho(t;\lambda)$ are not monotonic, we obtain optimal convergence rates of the regularisation method under variational source conditions such as in \autoref{eqCRvar}.

If we formulate it with the stronger standard source condition, see \autoref{thSsc}, we can reproduce a result similar to \cite[Theorem 5.1]{ZhaHof20}.

%%%%%%%%%%%%%%%%%%%%%%%%%%%%%%
%%% thSecOrdConclusion
%%%%%%%%%%%%%%%%%%%%%%%%%%%%%%
\begin{corollary}\label{thSecOrdConclusion}
Let $y\in\mathcal R(L)$ be given such that the corresponding minimum norm solution $x^\dag\in \mathcal X$, fulfilling $L x^\dag=y$ and $\|x^\dag\|=\inf\{\norm{x}\mid L x=y\}$, satisfies for some $\mu>0$ the source condition
\[ \xdag\in\mathcal R\big((L^*L)^{\frac\mu2}\big). \]

Then, if $\xi$ is the solution of the initial value problem in \autoref{eqODESecond},
\begin{enumerate}
\item
there exists a constant $C_1>0$ such that
\[ \norm{\xi(t;y)-x^\dag}^2 \le C_1t^{-\mu}\text{ for all }t>0; \]
\item
there exists a constant $C_2>0$ such that
\[ \inf_{t>0}\norm{\xi(t;\tilde y)-x^\dag}^2 \le C_2\norm{\tilde y-y}^{\frac{2\mu}{\mu+1}}\text{ for all }\tilde y\in \mathcal Y; \]
and
\item
there exists a constant $C_3>0$ such that
\[ \norm{L\xi(t;y)-y}^2 \le C_3t^{-\mu-1}\text{ for all }t>0. \]
\end{enumerate}
\end{corollary}
\begin{proof}
The proof follows exactly the lines of the proof of \autoref{thShoCR}, where the compatibility of~$\phiH_\mu$ is shown in \autoref{thSecOrdCompatibility} and we have here the slightly different scaling
\[ \xa(\tilde y)=\ra(L^*L)L^*\tilde y=\int_{(0,\norm{L}^2]}\frac{1-\tilde\rho(\tfrac b{2\alpha};\lambda)}\lambda\d\mathbf E_\lambda L^*\tilde y = \xi(\tfrac b{2\alpha};\tilde y) \]
between the regularised solution $\xa$, defined in \autoref{eq:reg} with the regularisation method $(\ra)_{\alpha>0}$ from \autoref{eqSecOrdGenerator}, and the solutions $\xi$ of \autoref{eqODESecond} and $\tilde\rho$ of \autoref{eqSecondOrder}; which however does not cause a change in the order of the convergence rates.
\end{proof}
%%% End

%%% End

%%%%%%%%%%%%%%%%%%%%%%%
%%% se:sflow
%%%%%%%%%%%%%%%%%%%%%%%
\section{The Vanishing Viscosity Flow}
\label{se:sflow}
We consider now the dynamical method \autoref{eqODE} for $N=2$ with the variable coefficient $a_1(t)=\frac b t$ for some parameter $b > 0$, that is, \autoref{eqODEVanishingViscosity}.
According to \autoref{thSpectral}, the solution of \autoref{eqODEVanishingViscosity} is defined via the spectral integral in \autoref{eqSolutionODE} of $\rho(t;\lambda) = \frac{1-\tilde{\rho}(t;\lambda)}{\lambda}$, where $\tilde{\rho}$ solves for every $\lambda\in(0,\infty)$ the 
initial value problem
\begin{equation} \label{eq:singular}
\begin{aligned}
\partial_{t t}\tilde\rho(t;\lambda)+\frac b t\partial_t\tilde\rho(t;\lambda)+\lambda\tilde\rho(t;\lambda)&=0 \text{ for all } t \in \ointerval{\infty},\\
\partial_t\tilde\rho(0;\lambda)&=0, \\
\tilde\rho(0;\lambda)&=1.
\end{aligned}
\end{equation}

As already noted in \cite[Section~3.2]{SuBoyCan16}, we obtain a closed form in terms of Bessel functions for the solution of \autoref{eq:singular}.

%%%%%%%%%%%%%%%%%%%%%%%%%%%%%%
%%% thSingularSolution
%%%%%%%%%%%%%%%%%%%%%%%%%%%%%%
\begin{lemma}\label{thSingularSolution}
Let $b, \lambda > 0$. Then \autoref{eq:singular} has the unique solution
\begin{equation}\label{eq:singularSolution}
\tilde\rho(t;\lambda) = u(t\sqrt\lambda)\text{ with }u(\tau)=\left(\frac2\tau\right)^{\frac12(b-1)}\Gamma(\tfrac12(b+1))J_{\frac12(b-1)}(\tau),
\end{equation}
where $\Gamma$ is the gamma function and $J_\nu$ denotes the Bessel function of first kind of order $\nu\in\R$. See \autoref{fgVanVis} for a sketch of the graph of the function $u$.
\end{lemma}
\begin{proof}
We rescale \autoref{eq:singular} by switching to the function
\begin{equation}\label{eq:singularRescaling}
v\colon[0,\infty)\times(0,\|L\|^2]\to\R,\; v(\tau;\lambda) = \tau^\kappa \tilde\rho(\sigma_\lambda\tau;\lambda)
\end{equation}
with some parameters $\sigma_\lambda\in(0,\infty)$ and $\kappa\in\R$. The function $v$ thus has the derivatives
\begin{equation}\label{eq:singularRescalingFirst}
\partial_\tau v(\tau;\lambda) = \sigma_\lambda\tau^\kappa\partial_t\tilde\rho(\sigma_\lambda\tau;\lambda)+\kappa\tau^{\kappa-1}\tilde\rho(\sigma_\lambda\tau;\lambda)
\end{equation}
and
\[ \partial_{\tau\tau}v(\tau;\lambda) = \sigma_\lambda^2\tau^\kappa\partial_{t t}\tilde\rho(\sigma_\lambda\tau;\lambda)+2\sigma_\lambda\kappa\tau^{\kappa-1}\partial_t\tilde\rho(\sigma_\lambda\tau;\lambda)+\kappa(\kappa-1)\tau^{\kappa-2}\tilde\rho(\sigma_\lambda\tau;\lambda). \]
We use \autoref{eq:singular} to replace the second derivative of $\tilde\rho$ and obtain
\[ \partial_{\tau\tau}v(\tau;\lambda) = \sigma_\lambda(2\kappa-b)\tau^{\kappa-1}\partial_t\tilde\rho(\sigma_\lambda\tau;\lambda)+(\kappa(\kappa-1)-\lambda\sigma_\lambda^2\tau^2)\tau^{\kappa-2}\tilde\rho(\sigma_\lambda\tau;\lambda), \]
which, after writing $\partial_t\tilde\rho$ and $\tilde\rho$ via \autoref{eq:singularRescalingFirst} and \autoref{eq:singularRescaling} in terms of the function $v$, becomes the differential equation
\[ \tau^2\partial_{\tau\tau}v(\tau;\lambda)+(b-2\kappa)\tau\partial_\tau v(\tau;\lambda)+(\lambda\sigma_\lambda^2\tau^2-\kappa(\kappa-1)-\kappa(b-2\kappa))v(\tau;\lambda) = 0 \]
for the function $v$. Choosing now $\kappa=\frac12(b-1)$, so that $b-2\kappa=1$, and $\sigma_\lambda=\frac1{\sqrt\lambda}$, we end up with Bessel's differential equation
\[ \tau^2\partial_{\tau\tau}v(\tau;\lambda)+\tau\partial_\tau v(\tau;\lambda)+(\tau^2-\kappa^2)v(\tau;\lambda) = 0, \]
for which every solution can be written as
\[ v(\tau;\lambda) = 
	\begin{cases}
	C_{1,\kappa}J_{|\kappa|}(\tau)+C_{2,\kappa}Y_{|\kappa|}(\tau),
		&\kappa\in\Z,\\
	C_{1,\kappa}J_{|\kappa|}(\tau)+C_{2,\kappa}J_{-|\kappa|}(\tau),
		&\kappa\in\R\setminus\Z,
	\end{cases} \]
for some constants $C_{1,\kappa},C_{2,\kappa}\in\R$, where $J_\nu$ and $Y_\nu$ denote the Bessel functions of first and second kind of order $\nu\in\R$, respectively; see, for example, \cite[Chapter 9.1]{AbrSte72}.

We can therefore write the solution $\tilde\rho$ as
\begin{equation}\label{eq:singularSolutionGeneral}
\tilde\rho(t;\lambda) =
	\begin{cases}
	C_{1,\kappa}(t\sqrt\lambda)^{-\kappa}J_{|\kappa|}(t\sqrt\lambda)+C_{2,\kappa}(t\sqrt\lambda)^{-\kappa}Y_{|\kappa|}(t\sqrt\lambda),
		&\kappa\in\Z,\\
	C_{1,\kappa}(t\sqrt\lambda)^{-\kappa}J_{|\kappa|}(t\sqrt\lambda)+C_{2,\kappa}(t\sqrt\lambda)^{-\kappa}J_{-|\kappa|}(t\sqrt\lambda),
		&\kappa\in\R\setminus\Z.
	\end{cases}
\end{equation}
To determine the constants $C_{1,\kappa}$ and $C_{2,\kappa}$ from the initial conditions, we remark that the Bessel functions have for all $\kappa\in\R\setminus(-\N)$ and all $n\in\N$ asymptotically for $\tau\to0$ the behaviour
\begin{equation}\label{eq:asymptoticsBessel}
\tau^{-\kappa}J_{\kappa}(\tau) = \frac1{2^{\kappa}\Gamma(\kappa+1)}+\mathcal O(\tau^2),\; \lim_{\tau\to0}\tau^n Y_n(\tau) = -\frac{2^n(n-1)!}\pi,\text{ and }\lim_{\tau\to0}\frac{Y_0(\tau)}{\log(\tau)} = \frac2\pi,
\end{equation}
see, for example, \cite[Formulae~9.1.10 and~9.1.11]{AbrSte72}.

We consider the cases $\kappa\ge0$ and $\kappa\in(-\frac12,0)$ separately.
\begin{itemize}
\item
In particular, the relations in \autoref{eq:asymptoticsBessel} imply that, for the last terms in \autoref{eq:singularSolutionGeneral}, we have with $\tau=t\sqrt\lambda$ asymptotically for $\tau\to0$
\begin{itemize}
\item
for $\kappa=0$:
\[ C_{2,0}Y_0(\tau) = \frac{2}{\pi}C_{2,0}\mathcal O(\log(\tau)) \]
because of the third relation in \autoref{eq:asymptoticsBessel};
\item
for $\kappa\in\N$:
\[ C_{2,\kappa}\tau^{-\kappa} Y_\kappa(\tau)= C_{2,\kappa}\tau^{-2\kappa} (\tau^\kappa Y_\kappa(\tau)) = C_{2,\kappa}\left(-\frac{2^\kappa(\kappa-1)!}{\pi}+o(1)\right)\tau^{-2\kappa} \]
because of the second relation in \autoref{eq:asymptoticsBessel}; and
\item
for $\kappa\in(0,\infty)\setminus\N$:
\[ C_{2,\kappa}\tau^{-\kappa}J_{-\kappa}(\tau)=C_{2,\kappa}\tau^{-2\kappa} (\tau^\kappa J_{-\kappa}(\tau))=C_{2,\kappa}\left(\frac{2^\kappa}{\Gamma(1-\kappa)}+\mathcal O(\tau^2)\right)\tau^{-2\kappa} \]
because of the first relation in \autoref{eq:asymptoticsBessel}.
\end{itemize}
Thus, the last terms in \autoref{eq:singularSolutionGeneral} diverge for every $\kappa\ge0$ as $t\to0$.

Since the first terms in \autoref{eq:singularSolutionGeneral} converge according to the first relation in \autoref{eq:asymptoticsBessel} for $t\to0$, the initial condition $\tilde\rho(0;\lambda)=1$ can only be fulfilled if the coefficients $C_{2,\kappa}$, $\kappa\ge0$, in front of the singular terms are all zero so that we have
\[ \tilde\rho(t;\lambda) = C_{1,\kappa}(t\sqrt\lambda)^{-\kappa}J_\kappa(t\sqrt\lambda)\text{ for all }\kappa\ge0. \]
Furthermore, the initial condition $\tilde\rho(0;\lambda)=1$ implies according to the first relation in \autoref{eq:asymptoticsBessel} that
\[ C_{1,\kappa} = 2^\kappa\Gamma(\kappa+1)\text{ for all }\kappa\ge0, \]
which gives the representation of \autoref{eq:singularSolution} for the solution $\tilde\rho$.

It remains to check that also the initial condition $\partial_t\tilde\rho(0;\lambda)=0$ is for all $\kappa\ge0$ fulfilled, which again follows directly from the first relation in \autoref{eq:asymptoticsBessel}:
\[ \partial_t\tilde\rho(0;\lambda) = \lim_{t\to0}\frac1t\left(2^\kappa\Gamma(\kappa+1)(t\sqrt\lambda)^{-\kappa}J_\kappa(t\sqrt\lambda)-1\right) = 0\text{ for all }\kappa\ge0. \]
\item
For $\kappa\in(-\frac12,0)$, we have that the first term in $\tilde\rho(t;\lambda)$ converges for $t\to0$ to $0$ because of
\[ C_{1,\kappa}(t\sqrt\lambda)^{|\kappa|}J_{|\kappa|}(t\sqrt\lambda) = C_{1,\kappa}(t\sqrt\lambda)^{2|\kappa|}\left(\frac1{2^{|\kappa|}\Gamma(|\kappa|+1)}+\mathcal O(t^2)\right), \]
which follows from the first relation of \autoref{eq:asymptoticsBessel}. Therefore, the initial condition $\tilde\rho(0;\lambda)=1$ requires that
\[ 1 = \tilde\rho(0;\lambda) = C_{2,\kappa}\lim_{t\to0}(t\sqrt\lambda)^{|\kappa|}J_{-|\kappa|}(t\sqrt\lambda)\text{ for all }\kappa\in(-\tfrac12,0), \]
from which we get with the first property in \autoref{eq:asymptoticsBessel} that
\[ C_{2,\kappa} = 2^\kappa\Gamma(\kappa+1). \]
To determine the coefficient $C_{1,\kappa}$, we remark that the first identity in \autoref{eq:asymptoticsBessel} then gives us for $t\to0$ the asymptotic behaviour
\[ \tilde\rho(t;\lambda) = 1+C_{1,\kappa}\frac{(t\sqrt\lambda)^{2|\kappa|}}{2^{|\kappa|}\Gamma(|\kappa|+1)}+\mathcal O(t^2). \]
Therefore, we have for the first derivative at $t=0$ the expression
\[ \partial_t\tilde\rho(0;\lambda) = \lim_{t\to0}C_{1,\kappa}\frac{2|\kappa|\lambda^{|\kappa|}}{2^{|\kappa|}\Gamma(|\kappa|+1)}\,\frac1{t^{1-2|\kappa|}}. \]
To satisfy the initial condition $\partial_t\tilde\rho(0;\lambda)=0$, we thus have to choose $C_{1,\kappa}=0$ for $\kappa\in(-\frac12,0)$, which leaves us again with \autoref{eq:singularSolution}.
\end{itemize}
\end{proof}
%%% End

%%%%%%%%%%%%%%%%%%%%%%%%%%%%%%
%%% fgVanVis
%%%%%%%%%%%%%%%%%%%%%%%%%%%%%%
\begin{figure}
\hfil
\includegraphics{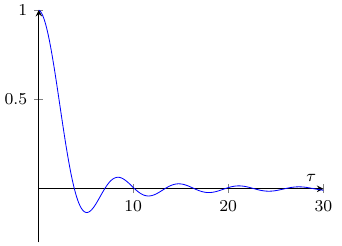}
\hfil
\caption{Graph of the function $u$, defined in \autoref{eq:singularSolution}, which gives the solution $\tilde\rho$ of \autoref{eq:singular} via $\tilde\rho(t;\lambda)=u(t\sqrt\lambda)$. As in the heavy ball method, the function~$\tilde{\rho}$ is not monotonic (in either component) so that we used the envelope of the regularisation method to obtain the optimal convergence rates.}\label{fgVanVis}
\end{figure}
%%% End

%%%%%%%%%%%%%%%%%%%%%%%%%%%%%%
%%% thSingUnique
%%%%%%%%%%%%%%%%%%%%%%%%%%%%%%
\begin{corollary}\label{thSingUnique}
The unique solution $\xi\colon[0,\infty)\times \mathcal Y\to\R$ of the vanishing viscosity flow, \autoref{eqODEVanishingViscosity}, which is twice continuously differentiable with respect to $t$ is given by
\[ \xi(t;\tilde y) = \int_{(0,\norm{L}^2]}\frac{1-u(t\sqrt\lambda)}\lambda\d\mathbf E_\lambda L^*\tilde y \]
where the function $u$ is defined by \autoref{eq:singularSolution}.
\end{corollary}

\begin{proof}
We have already seen in \autoref{thSingularSolution} that \autoref{eq:singular} has the unique solution $\tilde\rho$ 
given by $\tilde\rho(t;\lambda)=u(t\sqrt\lambda)$. To apply \autoref{thSpectral}, it is thus enough to show that $\tilde\rho$ is smooth.

Since the function $u$ has the representation
\[ u(\tau) = v(\tau^2)\text{ with } v(\tilde\tau) = \Gamma(\tfrac12(b+1))\sum_{k=0}^\infty\frac{(-\frac14\tilde\tau)^k}{k!\Gamma(\frac12(b+1)+k)}, \]
see, for example, \cite[Formula~9.1.10]{AbrSte72}, the solution $\tilde\rho\colon[0,\infty)\times[0,\infty)\to\R$ given by \autoref{eq:singularSolution} is of the form $\tilde\rho(t;\lambda)=u(t\sqrt\lambda)=v(\lambda t^2)$ and is therefore seen to be smooth. Therefore, \autoref{thSpectral} yields the claim.
\end{proof}
%%% End

Again, we want to determine a corresponding regularisation method. We start by showing that the function $\lambda\mapsto\tilde\rho(t;\lambda)$, which corresponds to the error function $\tra$ of the regularisation method, is non-negative and monotonically decreasing for sufficiently small values $\lambda$ as required for $\tra$ in \autoref{deGenerator}~\ref{enGeneratorError}.

%%%%%%%%%%%%%%%%%%%%%%%%%%%%%%
%%% thStrictlyDecreasing
%%%%%%%%%%%%%%%%%%%%%%%%%%%%%%
\begin{lemma}\label{thStrictlyDecreasing}
Let $j_{\kappa,1}\in(0,\infty)$ denote the first positive zero of the Bessel function $J_\kappa$. Then, the solution~$\tilde\rho$ given in \autoref{eq:singularSolution} fulfils
\begin{itemize}
\item
for every $\lambda>0$ that the function $t\mapsto\tilde\rho(t;\lambda)$ is strictly decreasing on the interval $(0,\frac1{\sqrt\lambda}j_{\frac12(b-1),1})$ and
\item
for every $t>0$ that the function $\lambda\mapsto\tilde\rho(t;\lambda)$ is strictly decreasing on the interval $(0,\frac1{t^2}j_{\frac12(b-1),1}^2)$.
\end{itemize}
\end{lemma}
\begin{proof}
Since we can write $\tilde\rho$ in the form $\tilde\rho(t;\lambda)=u(t\sqrt\lambda)$, see \autoref{eq:singularSolution}, it is enough to show that
\[ u'(\tau)<0\text{ for }\tau\in(0,j_{\frac12(b-1),1}). \]
This property of $u$ follows directly from the representation of the Bessel functions $J_\kappa$, $\kappa\in(-\frac12,\infty)$, 
as an infinite product, see, for example, \cite[Formula~9.5.10]{AbrSte72}:
\[ J_\kappa(\tau) = \frac{\tau^\kappa}{2^\kappa\Gamma(\kappa+1)}\prod_{\ell=1}^\infty\left(1-\frac{\tau^2}{j_{\kappa,\ell}^2}\right), \]
where $j_{\kappa,\ell}$ denotes the $\ell$th positive zero (sorted in increasing order) of $J_\kappa$; since this gives
\[ u(\tau) = \prod_{\ell=1}^\infty\left(1-\frac{\tau^2}{j_{\frac12(b-1),\ell}^2}\right), \]
which is for $\tau\in(0,j_{\frac12(b-1),1})$ a product of only positive factors. Therefore, we have
\[ u'(\tau) = -2\tau\sum_{\ell=1}^\infty\left[j_{\frac12(b-1),\ell}^{-2}
\prod_{\tilde\ell\in\N\setminus\{\ell\}}\left(1-\frac{\tau^2}{j_{\frac12(b-1),\tilde\ell}^2}\right)\right] < 0\text{ for all }
\tau\in(0,j_{\frac12(b-1),1}). \]
\end{proof}
%%% End

Furthermore, we can construct an upper bound $\tilde P(t;\lambda)=U(t\sqrt\lambda)$ of $|\tilde\rho(t;\lambda)|$, which corresponds
to the envelope value $\Tra(\lambda)$, such that $\tilde P(t;\cdot)$ is monotonically decreasing. This will give us the condition of 
\autoref{deGenerator}~\ref{enGeneratorErrorReg} for the function $\Tra$. The additionally derived explicit upper bound for $U$ helps us to 
show the compatibility of the convergence rate functions $\phiH_\mu$ and $\phiL_\mu$.

%%%%%%%%%%%%%%%%%%%%%%%%%%%%%%
%%% thSingularUpperBound
%%%%%%%%%%%%%%%%%%%%%%%%%%%%%%
\begin{lemma}\label{thSingularUpperBound}
Let $\tilde\rho$ be the solution of \autoref{eq:singular} given by \autoref{eq:singularSolution}. 
Then, there exist a constant $C>0$ and a continuous, monotonically decreasing function $U\colon[0,\infty)\to[0,1]$ so that
\begin{itemize}
\item $|\tilde\rho(t;\lambda)| \le U(t\sqrt\lambda)$ for every $t\ge0$, $\lambda>0$,
\item $U(\tau)<1$ for all $\tau>0$, and 
\item $U(\tau) \le C\tau^{-\frac b2}$ for all $\tau>0$.
\end{itemize}
\end{lemma}
\begin{proof}
We use again the function $u$ defined in \autoref{eq:singularSolution} which satisfies $\tilde\rho(t;\lambda)=u(t\sqrt\lambda)$. 
Then, we remark that the energy
\[ E(\tau) \coloneqq {u'}^2(\tau)+u^2(\tau),\;\tau \ge 0, \]
fulfils (using \autoref{eq:singular} with $\lambda=1$, $t=\tau$ and $u(\tau)=\tilde\rho(\tau;1)$)
\[ E'(\tau) = 2u'(\tau)\left(u''(\tau)+u(\tau)\right) = -\frac b\tau{u'}^2(\tau) \le 0. \]
Since we know from \autoref{thStrictlyDecreasing} that $u'(\tau)=\partial_t\tilde\rho(\tau;1)<0$ for $\tau\in(0,j_{\frac12(b-1),1})$, we have 
that $E$ is strictly decreasing on $(0,j_{\frac12(b-1),1})$ so that $E(j_{\frac12(b-1),1}) < E(0)$. For $\tau\ge j_{\frac12(b-1),1}$, 
we can therefore estimate $u$ by
\[ u^2(\tau) \le E(\tau) \le E(j_{\frac12(b-1),1}) < E(0) = 1. \]
Thus, $u$ is monotonically decreasing on $(0,j_{\frac12(b-1),1})$ and uniformly bounded by $E(j_{\frac12(b-1),1})<1$ on $[j_{\frac12(b-1),1},\infty)$. 
Therefore, we can find a monotonically decreasing function $\tilde U\colon[0,\infty)\to[0,1]$ with
\[ u(\tau) \le \tilde U(\tau)<1\text{ for every }\tau>0. \]

Since it follows from \cite[Formula 9.2.1]{AbrSte72} that there exists a constant $c>0$ such that
\[ |J_{\frac12(b-1)}(\tau)| \le c\tau^{-\frac12}\text{ for all }\tau>0, \]
which implies according to \autoref{eq:singularSolution} with $C=2^{\frac12(b-1)}\Gamma(\frac12(b+1))c$ the upper bound
\[ |u(\tau)|\le C\tau^{-\frac b2}\text{ for all }\tau>0, \]
the function $U$ defined by $U(\tau)=\min\{\tilde U(\tau),C\tau^{-\frac b2}\}$ satisfies all the properties.
\end{proof}
%%% End

To verify the condition in \autoref{deGenerator}~\ref{enGeneratorBounded} for $\ra$, we establish here the corresponding lower bound for the function $\tilde\rho$.

%%%%%%%%%%%%%%%%%%%%%%%%%%%%%%
%%% thSingularLowerBound
%%%%%%%%%%%%%%%%%%%%%%%%%%%%%%
\begin{lemma}\label{thSingularLowerBound}
Let $\tilde\rho$ be the solution of \autoref{eq:singular} given by \autoref{eq:singularSolution}. Then, there exists a constant $\tau_b\in(0,j_{\frac12(b-1),1}]$ such that
\begin{equation} \label{eq:lb}
 \tilde\rho(t;\lambda)\ge 1-\frac{t\sqrt\lambda}{2\tau_b}\text{ for all }t\ge0,\;\lambda>0.
\end{equation}
\end{lemma}
\begin{proof}
We define $u$ again by \autoref{eq:singularSolution} and choose some arbitrary $c>0$. Then, the initial conditions $u(0)=1$ and $u'(0)=0$ imply that we find a $\bar\tau>0$ such that $u(\tau)\ge 1-c\tau$ for all $\tau\in[0,\bar\tau]$. Setting now $\tau_b\coloneqq\min\{\frac1{2c},\frac{\bar\tau}4,j_{\frac12(b-1),1}\}$, we have by construction 
\[ u(\tau)\ge1-c\tau\ge 1-\frac\tau{2\tau_b}\text{ for all }\tau\in[0,\bar\tau]. \]
Moreover, the uniform bound $\abs{u(\tau)}<1$ for all $\tau>0$, shown in \autoref{thSingularUpperBound}, implies that
\[ u(\tau) > -1 \ge 1-\frac2{\bar\tau}\tau \ge 1-\frac\tau{2\tau_b}\text{ for all }\tau\in[\bar\tau,\infty). \]
Thus, $\tilde\rho(t;\lambda)=u(t\sqrt\lambda)$ yields the claim.
\end{proof}
%%% End

These estimates for $\tilde\rho$ suffice now to show that the functions $\ra$ defined by \autoref{eqSingRegMeth} generate the regularisation method corresponding to the solution $\xi$ of \autoref{eqODEVanishingViscosity}.

%%%%%%%%%%%%%%%%%%%%%%%%%%%%%%
%%% thSingReg
%%%%%%%%%%%%%%%%%%%%%%%%%%%%%%
\begin{proposition}\label{thSingReg}
Let $\tilde\rho$ be the solution of \autoref{eq:singular} given by \autoref{eq:singularSolution}, $\tau_b$ be the constant defined in \autoref{thSingularLowerBound}, and set
\begin{equation}\label{eqSingRegMeth}
\ra(\lambda) \coloneqq \frac1\lambda\left(1-\tilde\rho\left(\frac{\tau_b}{\sqrt{\alpha}};\lambda\right)\right).
\end{equation}

Then, $(\ra)_{\alpha>0}$ generates a regularisation method in the sense of \autoref{deGenerator}.
\end{proposition}

\begin{proof}
We verify the four conditions in \autoref{deGenerator}.
\begin{enumerate}
\item We know from \autoref{thSingularUpperBound} that $|\tilde\rho|\le1$, and thus it follows that
\[ \ra(\lambda) \le \frac2\lambda. \]
Moreover, it follows from \autoref{eq:lb} that
\[ \ra(\lambda) = \frac1\lambda\left(1-\tilde\rho\left(\frac{\tau_b}{\sqrt{\alpha}};\lambda\right)\right) \le \frac1{2\sqrt{\alpha\lambda}}. \]
\item
The error function $\tilde r_\alpha$ corresponding to the generator $r_\alpha$ is given by
\[ \tra(\lambda) = \tilde\rho\left(\frac{\tau_b}{\sqrt{\alpha}};\lambda\right), \]
which is a monotonically decreasing function on $(0,\frac1{\tau_b^2}j_{\frac12(b-1),1}^2\alpha)$ according to \autoref{thStrictlyDecreasing}. Since we have chosen $\tau_b\in(0,j_{\frac12(b-1),1}]$, see \autoref{thSingularLowerBound}, this in particular shows that $\tra$ is monotonically decreasing on $(0,\alpha)$.
\item
Let $U$ be the function constructed in \autoref{thSingularUpperBound}. We define
\begin{equation}\label{eqSingRegMethEnv}
\Tra(\lambda) \coloneqq U\left(\tau_b\sqrt{\frac\lambda\alpha}\right).
\end{equation}
Then, we have by \autoref{thSingularUpperBound} that $\lambda\mapsto\Tra(\lambda)$ is monotonically decreasing, $\alpha\mapsto\Tra(\lambda)$ is continuous and monotonically increasing and~$\Tra$ fulfils
\[ \tra(\lambda) = \tilde\rho\left(\frac{\tau_b}{\sqrt{\alpha}};\lambda\right) \le U\left(\tau_b\sqrt{\frac\lambda\alpha}\right) = \Tra(\lambda). \]
\item
We have again by \autoref{thSingularUpperBound} that
\[ \Tra(\alpha) = U(\tau_b) < 1\text{ for all }\alpha>0. \]
\end{enumerate}
\end{proof}
%%% End

As before, we also verify that the classical convergence rate functions $\phiH_\mu$ and $\phiL_\mu$ are compatible with the regularisation method $(\ra)_{\alpha>0}$. In contrast to Showalter's method and the heavy ball method, the compatibility for $\phiH_\mu$ only holds up to a certain saturation value for the parameter $\mu$.

%%%%%%%%%%%%%%%%%%%%%%%%%%%%%%
%%% thSingularCompatibility
%%%%%%%%%%%%%%%%%%%%%%%%%%%%%%
\begin{lemma}\label{thSingularCompatibility}
The functions $\phiH_\mu$ for all $\mu\in(0,\frac b2)$ and the functions $\phiL_\mu$ for all $\mu>0$, as defined in \autoref{exCR}, are compatible with the regularisation method $(r_\alpha)_{\alpha>0}$ defined by \autoref{eqSingRegMeth} in the sense of \autoref{deCompatible}.
\end{lemma}
\begin{proof}
As before, it is because of \autoref{thTransComp} enough to check this for the functions $\phiH_\mu$, $\mu\in(0,\frac b2)$.
The function $\Tra$ defined in \autoref{eqSingRegMethEnv} fulfils according to \autoref{thSingularUpperBound} that there exists a constant $C>0$ with
\[ \Tra^2(\lambda) = U^2\left(\tau_b\sqrt{\frac\lambda\alpha}\right) \le C^2\tau_b^{-b}\left(\frac\lambda\alpha\right)^{-\frac b2} \le C^2\tau_b^{-b}\left(\frac{\phiH_\mu(\lambda)}{\phiH_\mu(\alpha)}\right)^{-\frac b{2\mu}}, \]
which is \autoref{eqSourceConditionTail} with the compatibility function $F_\mu(z)=C^2\tau_b^{-b}z^{-\frac b{2\mu}}$. It remains to check that $F_\mu\colon[1,\infty)\to\R$ is integrable, which is the case for $\mu<\frac b2$.
\end{proof}
%%% End

We can therefore apply \autoref{thCR} to the regularisation method generated by the functions $(r_\alpha)_{\alpha>0}$ defined in \autoref{eqSingRegMeth} and the convergence rates $\phiH_\mu$, $\mu\in(0,\frac b2)$, and $\phiL_\mu$, $\mu>0$. By using that we have by construction $\xa(\tilde y)=\xi(\frac{\tau_b}{\sqrt\alpha};\tilde y)$, see \autoref{eqSingularRegSol} below, this gives us equivalent characterisations for convergence rates of the flow $\xi$ of \autoref{eqODEVanishingViscosity}.
As before for Showalter's method and the heavy ball method, we formulate the resulting convergence rates under the stronger, but more commonly used standard source condition, see \autoref{thSsc}.

%%%%%%%%%%%%%%%%%%%%%%%%%%%%%%
%%% thSingularConclusion
%%%%%%%%%%%%%%%%%%%%%%%%%%%%%%
\begin{corollary}\label{thSingularConclusion}
Let $y\in\mathcal R(L)$ be given such that the corresponding minimum norm solution $x^\dag\in \mathcal X$, fulfilling $L x^\dag=y$ and $\|x^\dag\|=\inf\{\norm{x}\mid L x=y\}$, satisfies for some $\mu\in(0,\frac b2)$ the source condition
\[ \xdag\in\mathcal R\big((L^*L)^{\frac\mu2}\big). \]

Then, if $\xi$ is the solution of the initial value problem in \autoref{eqODEVanishingViscosity},
\begin{enumerate}
\item
there exists a constant $C_1>0$ such that
\[ \norm{\xi(t;y)-x^\dag}^2 \le\frac{C_1}{t^{2\mu}}\text{ for all }t>0; \]
\item
there exists a constant $C_2>0$ such that
\[ \inf_{t>0}\norm{\xi(t;\tilde y)-x^\dag}^2 \le C_2\norm{\tilde y-y}^{\frac{2\mu}{\mu+1}}\text{ for all }\tilde y\in \mathcal Y; \]
and
\item
if $\mu<\frac b2-1$, there exists a constant $C_3>0$ such that
\[ \norm{L\xi(t;y)-y}^2 \le \frac{C_3}{t^{2(\mu+1)}}\text{ for all }t>0. \]
\end{enumerate}
\end{corollary}
\begin{proof}
The proof follows exactly the lines of the proof of \autoref{thShoCR}, where the compatibility of~$\phiH_\mu$ is shown in \autoref{thSingularCompatibility} and we have the different scaling
\begin{equation}\label{eqSingularRegSol}
\xa(\tilde y)=\ra(L^*L)L^*\tilde y=\int_{(0,\norm{L}^2]}\frac1\lambda\left(1-\tilde\rho\left(\frac{\tau_b}{\sqrt\alpha};\lambda\right)\right)\d\mathbf E_\lambda L^*\tilde y = \xi\left(\frac{\tau_b}{\sqrt\alpha};\tilde y\right)
\end{equation}
between the regularised solution $\xa$, defined in \autoref{eq:reg} with the regularisation method $(\ra)_{\alpha>0}$ from \autoref{eqSingRegMeth}, and the solutions $\xi$ of \autoref{eqODEVanishingViscosity} and $\tilde\rho$ of \autoref{eq:singular}. 
Following \autoref{thShoCR} and using the notation $d$ from \autoref{eq:dD} and 
$\tilde d$ from \autoref{eq:tilde_dD} we get
\begin{enumerate}
\item in the case of exact data the convergence rates
\[ \norm{\xi(t;y)-x^\dag}^2 = \norm{x_{\tau_b^2t^{-2}}(y)-x^\dag}^2 = d\left(\frac{\tau_b^2}{t^2}\right) \text{ and } \norm{x_{\tau_b^2t^{-2}}(y)-x^\dag}^2
\le \frac{C_d\tau_b^{2\mu}}{t^{2\mu}}.\]
\item For perturbed data we get the convergence rate
\[ \inf_{t>0}\norm{\xi(t;\tilde y)-x^\dag}^2 = \inf_{\alpha>0}\norm{\xi\left(\frac{\tau_b}{\sqrt\alpha};\tilde y\right)-x^\dag}^2\le \tilde d(\norm{\tilde y-y}) \le C_{\tilde d}\norm{\tilde y-y}^{\frac{2\mu}{\mu+1}}. \]
\item Moreover, using that for $\mu<\frac b2-1$ also $\phiH_{\mu+1}$ is compatible with $(\ra)_{\alpha>0}$, we get from \autoref{thCRImageSimple} the convergence rate
\[ \norm{L\xi(t;y)-y}^2 = \norm{L x_{\tau_b^2t^{-2}}(y)-y}^2 = q\left(\frac{\tau_b^2}{t^2}\right) \le C\tau_b^{2(\mu+1)}t^{-2(\mu+1)} \]
for the noise free residual error, where $q$ is defined in \autoref{eqQ}.
\end{enumerate}
\end{proof}
%%% End

\medskip
We end this section by a few remarks.
\begin{remark}[Comparison of Flows]
Comparing the results in \autoref{thShoCR}, \autoref{thSecOrdConclusion}, and \autoref{thSingularConclusion}, we see that the three methods we have analysed, namely Showalter's method, the heavy ball dynamics, and the vanishing viscosity flow, all give the same rate of convergence for noisy data with optimal stopping time. However, one should notice that their optimal stopping times are different. This is due to the acceleration property of the vanishing viscosity flow in comparison with the other two, which has been analysed in the literature.
\end{remark}

\begin{remark}[Saturation of Viscosity Flow]
The vanishing viscosity flow suffers from a saturation effect for the convergence rate functions $\phiH_\mu$ allowing only convergence rates up to certain values of $\mu$, which is not the case in the other two methods (because of their exponential decay of the error function at every fixed spectral value).
\end{remark}

\begin{remark}[Comparison with literature]
\autoref{eq:singular} has been investigated quite heavily in a more general context of non-smooth, convex functionals $\mathcal J$ and abstract ordinary differential equations 
of the form
\begin{equation} \label{eqODE_general}
\begin{aligned}
\xi'' (t) + \frac{b}{t} \xi'(t) + \partial \mathcal{J}(\xi(t)) &\ni 0 \text{  for all } t \in \ointerval{\infty}, \\
\xi'(0) &= 0, \\
\xi(0) &= 0,
\end{aligned}
\end{equation}
see for instance \cite{SuBoyCan16,AttChbPeyRed18,AttChbRia18,AujDosRon19,ApiAujDos18}.
\autoref{eqODE_general} corresponds to \autoref{eqODE} with $N=2$ and $a_1(t)=\frac b t$, $b>0$, for the particular energy functional
$\mathcal{J}(x)=\frac{1}{2}\norm{L x-y}^2$.

The authors prove optimality of \autoref{eqODE_general}, which, however, is a different term than in our paper:
\begin{enumerate}
\item In the above referenced paper optimality is considered with respect to all possible smooth and convex functionals $\mathcal{J}$, 
while in our work optimality is considered with respect to all possible variations of $y$ only. The papers \cite{SuBoyCan16,AujDosRon19,ApiAujDos18} consider a finite dimensional setting where 
$\mathcal{J}$ maps a subset of a finite dimensional space $\R^d$ into the extended reals.

\item The second difference in the optimality results is that we consider primarily optimal convergence rate of $\xi(t)-\xdag$ for $t \to \infty$
and not of $\mathcal{J}(\xi(t)) \to \min_{x\in \mathcal X}\mathcal J(x)$, that is, we are considering rates in the domain of $L$, while in the referenced papers convergence in the image domain is considered. Consequently, we get rates for the residual squared (which 
is the rate of $J(\xi(t))$ in the referenced papers), which are based on optimal rates (in the sense of this paper) for $\xi(t)-\xdag \to 0$. The presented rates in the image domain are, however, not necessarily optimal.
\end{enumerate}

Nevertheless, it is very interesting to note that the two cases $b \geq 3$ and $0 < b < 3$, referred to as heavy and low friction cases, do not result in a different analysis in our paper, compared to, for instance, \cite{ApiAujDos18}. This is of course not a contradiction, because we consider a different optimality terminology.
\end{remark}
%%% End

%%%%%%%%%%%%%%%%%%%%%%%%%%%%%%
%%% Conclusions
%%%%%%%%%%%%%%%%%%%%%%%%%%%%%%
\section*{Conclusions}
The paper shows that the dynamical flows provide optimal regularisation methods (in the sense explained in \autoref{se:review}).
We proved optimal convergence rates of the solutions of the flows to the minimum norm solution for $t\to \infty$ and we also 
provide convergence rates of the residuals of the regularised solutions.

We observed that the vanishing viscosity method, heavy ball dynamics, and Showalter's method provide optimal 
reconstructions for different times. In fact, eventually, for a fair numerical comparison of the results of all three 
methods one should compare the results of Showalter's method and the heavy ball dynamics, respectively, at time $t_0^2$ 
with the vanishing viscosity flow at time $t_0$.
%%% End

%%%%%%%%%%%%%%%%%%%%%%%%%%%%%%
%%% Acknowledgements
%%%%%%%%%%%%%%%%%%%%%%%%%%%%%%
\subsection*{Acknowledgements}
\begin{itemize}
\item
RB acknowledges support from the Austrian Science Fund (FWF) within the project I2419-N32 (Employing Recent Outcomes in Proximal Theory Outside the Comfort Zone).
\item
GD is funded by the Deutsche Forschungsgemeinschaft (DFG, German Research Foundation) under Germany's Excellence Strategy ``The Berlin Mathematics Research Center MATH'' (EXC-2046/1, project ID: 390685689).
\item
PE and OS are supported by the Austrian Science Fund (FWF), with SFB F68, project F6804-N36 (Quantitative Coupled Physics Imaging) and project F6807-N36 (Tomography with Uncertainties).
\item
OS acknowledges support from the Austrian Science Fund (FWF) within the national research network Geometry and Simulation, 
project S11704 (Variational Methods for Imaging on Manifolds) and I3661-N27 (Novel Error Measures and Source 
Conditions of Regularization Methods for Inverse Problems).
\end{itemize}
%%% End

%%%%%%%%%%%%%%%%%%%%%%%%%%%%%%
%%% References
%%%%%%%%%%%%%%%%%%%%%%%%%%%%%%
\section*{References}
\printbibliography[heading=none]
%%% End

\end{document}